\documentclass{amsart}
\usepackage{amsbsy,amssymb,amscd,amsfonts,latexsym,amstext,delarray, amsmath,graphicx,color,caption,enumerate,lscape,xcolor,mathabx}
\usepackage{graphicx}
\usepackage{caption}
\usepackage{subcaption}
\usepackage{epsfig}
\usepackage{tikz}
\usetikzlibrary{decorations.markings,calc,cd}
\usepackage{resizegather}
\input xy

\xyoption{all}
\usepackage{euscript}
\usepackage{multirow}
\usepackage{etex, pictexwd,dcpic}
\usepackage[makeroom]{cancel}

\allowdisplaybreaks

\newtheorem{theorem}{Theorem}
\newtheorem{definition}[theorem]{Definition}
\newtheorem{proposition}[theorem]{Proposition}
\newtheorem{lemma}[theorem]{Lemma}
\newtheorem{corollary}[theorem]{Corollary}
\newtheorem{example}[theorem]{Example}
\newtheorem{conjecture}[theorem]{Conjecture}
\newtheorem{remark}[theorem]{Remark}

\definecolor{darkgreen}{rgb}{0, 0.5, 0}

\tikzset{->-/.style={decoration={
  markings,
  mark=at position .5 with {\arrow{>}}},postaction={decorate}}}
\tikzset{-<-/.style={decoration={
  markings,
  mark=at position .5 with {\arrow{<}}},postaction={decorate}}}
\tikzset{-->-/.style={decoration={
  markings,
  mark=at position .75 with {\arrow{>}}},postaction={decorate}}}
\tikzset{->--/.style={decoration={
  markings,
  mark=at position .25 with {\arrow{>}}},postaction={decorate}}}
\tikzset{->>-/.style={decoration={
  markings,
  mark=at position .57 with {\arrow{>>}}},postaction={decorate}
}}
\tikzset{->>>-/.style={decoration={
  markings,
  mark=at position .58 with {\arrow{>>>}}},postaction={decorate}
}}
\tikzset{->>>>-/.style={decoration={
  markings,
  mark=at position .62 with {\arrow{>>>>}}},postaction={decorate}
}}

\newcommand{\Oa}{\mathcal O_A}
\newcommand{\Ob}{\mathcal O_B}
\newcommand{\R}{\mathcal R}

\newcommand{\bigslant}[2]{{\raisebox{.2em}{$#1$}\left/\raisebox{-.2em}{$#2$}\right.}}

\title{An Elliptic Generalization of $A_1$ Spherical DAHA at $K=2$}
\author{S.~Arthamonov}
\address{S.A.: Department of Mathematics, University of Toronto. Toronto ON, Canada.}
\email{semeon.artamonov@utoronto.ca}

\author{Sh.~Shakirov}
\address{Sh.Sh.: Institute for Information Transmission Problems. Moscow, Russia.}
\email{shakirov.work@gmail.com}

\begin{document}

\begin{abstract}
We construct an algebra that is an elliptic generalization of $A_1$ spherical DAHA acting on its finite-dimensional module at $t=-q^{-K/2}$ with $K=2$. We prove that $PSL(2,{\mathbb Z})$ acts by automorphisms of the algebra we constructed, and provide an explicit representation of automorphisms and algebra operators alike by $3 \times 3$ matrices of elliptic functions. A relation of this construction to the $K$-theory character of affine Laumon space is conjectured. We point out two potential applications, respectively to $SL(3,{\mathbb Z})$ symmetry of Felder-Varchenko functions and to new elliptic invariants of torus knots and Seifert manifolds.    
\end{abstract}
\maketitle

\section{Introduction}

Double Affine Hecke Algebra (DAHA) \cite{Cherednik'1992, Cherednik'2005} is an important algebra with applications in modern mathematics and mathematical physics. Famously it has been used to give a general proof of Macdonald conjectures for all root systems \cite{Cherednik'1995}. In algebraic geometry, DAHA is related to the study of the Hilbert scheme of points \cite{GorskyNegut'2015}. Lately, DAHA has been used to construct knot invariants for torus and iterated torus knots, the so called DAHA-Jones polynomials \cite{Cherednik'2013,CherednikDanilenko'2016}. This construction admits connections to physics, in particular knot invariants of refined Chern-Simons theory \cite{AganagicShakirov'2015}.

Action of the group $SL(2,{\mathbb Z})$ by automorphisms is an inherent feature of DAHA. It can be observed already in the simplest example of type $A_1$ \emph{spherical} DAHA, which can be presented as an associative algebra generated by two difference operators  

\begin{align}
\Oa = \dfrac{1 - t x^2}{1 - x^2} \delta + \dfrac{1 - t x^{-2}}{1 - x^{-2}} \delta^{-1}, \ \ \ \ \ \Ob = x + x^{-1}
\end{align}
\smallskip\\
acting on a reasonable class of functions of a single variable $x$ where $\delta f(x) = f( q^{1/2} x)$ and $q,t$ are complex parameters. In this  example, one may define automorphisms $\mathbf a,\mathbf b$ by polynomial formulas

\begin{align}
\mathbf a\big(\Oa\big) := \Oa, \ \ \ \ \ \mathbf a \big(\Ob\big) := \dfrac{q^{1/2} \Ob \Oa - q^{-1/2} \Oa \Ob}{q^{1/2} - q^{-1/2}}
\end{align}
\begin{align}
\mathbf b\big(\Oa\big) := \dfrac{q^{1/2} \Oa \Ob - q^{-1/2} \Ob \Oa}{q^{1/2} - q^{-1/2}}, \ \ \ \ \ \mathbf b\big(\Ob\big) := \Ob 
\end{align}
\smallskip\\
and it is easy to prove by direct check in above presentation that these automorphisms satisfy $\mathbf a\mathbf b\mathbf a = \mathbf b\mathbf a\mathbf b$ and $(\mathbf{aba})^2 = 1$, the defining relations of the group $PSL(2,{\mathbb Z})$.

This simple example admits generalizations. First and foremost, it can be generalized from spherical (i.e. where symmetry between $x$ and $x^{-1}$ is present) to generic DAHA that admits a little more involved difference operator presentation known as Dunkl operators. Second, it can be generalized from the root system $A_1$ to arbitrary root systems \cite{Cherednik'1995,Cherednik'1995-Inventiones,Cherednik'1995-IMRN,Sahi'1999}. There exist further generalizations of DAHA that preserve some of its interesting structure: a rich group of automorphisms, presentations by difference operators, and connections to knot theory and mathematical physics. For example, in \cite{ArthamonovShakirov'2019} we suggested an algebra with two complex parameters $q,t$ that admits a genus 2 mapping class group action by automorphisms, generalizing the action of $SL(2,{\mathbb Z})$, a genus 1 mapping class group. 

In this paper, we suggest yet another generalization. Spherical DAHA of type $A_1$ has finite-dimensional modules at particular values of parameters $q,t$ satisfying $t = - q^{-K/2}$ for $K = 1,2,3,\ldots$. For each $K$, this allows to define an algebra that is a quotient of the spherical DAHA of type $A_1$ by additional relations specific for the action on the corresponding module. We suggest a generalization of the $K=2$ case. The algebra that we suggest admits a representation of $SL(2,{\mathbb Z})$ with elliptic functions, justifying the title of present paper. At trivial value of the elliptic parameter, this representation reduces to the standard $K=2$ representation of $SL(2,{\mathbb Z})$. We show that the algebra that we construct is related to an essential part of modern algebraic geometry -- equivariant K-theory character of affine Laumon space. In the Conclusion section we  discuss potential applications of our construction. One is relation to the work \cite{FelderVarchenko'2001} on elliptic conformal blocks and $SL(3,{\mathbb Z})$ symmetry. Another is relation to knot theory. 

This paper is structured as follows. In Section \ref{sec:Preliminaries} we review the basic material about infinite products and $q$-Pochhammer symbols. Throughout the text we face the necessity to work with the extended range of elliptic parameters, so we use this section to introduce a slightly more general version of Plethystic exponent which serves as a major building block of explicit formulas in our text. A reader, who is well familiar with the objects involved in, say, Example \ref{ex:PlethysticTwoWays} can safely skip Section \ref{sec:Preliminaries} on the first reading and refer to it only when necessary.

The main body of this text then starts with a short Section \ref{sec:SL2ZActionOnMatR3} where we propose an elliptic deformation with two elliptic parameters $p=\mathrm e^{2\pi\mathrm i\tau},s=\mathrm e^{2\pi\mathrm i\sigma}$ of particular projective $SL(2,\mathbb Z)$-representation coming from the action of the modular group of the torus by automorphisms of DAHA. Namely, we define a pair of matrix-valued difference operators $\widehat D_A,\widehat D_B$ and show that these two difference operators satisfy defining relations of an $SL(2,\mathbb Z)$ (up to a scalar factor independent of $p,s$). Matrix part of the two operators deforms the $K=2$ projective $SL(2,\mathbb Z)$-representation from \cite{AganagicShakirov'2015}.

Section \ref{sec:AlgebraK2} introduces our main algebra $\mathcal A$ over the field $\mathbf k=\mathbb C(Q)$ of rational functions in one variable. This section apriori is completely independent of the preceding two sections as there are no elliptic parameters $p,s$ involved. We define the algebra $\mathcal A$ as an associative algebra with countably many generators $\Oa^{(g)},\Ob^{(g)},\;g\in Free_2$ subject to countably many relations and show that $PSL(2,\mathbb Z)$ acts by automorphisms of this algebra. We refer to $\mathcal A$ as an elliptic generalization of $A_1$ spherical DAHA due to the following property: The quotient of $\mathcal A$ by an ideal generated by $\Oa^{(g)}-\Oa^{(1)},\Ob^{(g)}-\Ob^{(1)}$ for all $g\in Free_2$ together with a Casimir element is isomorphic to a level $K=2$ finite-dimensional representation of an $A_1$ spherical DAHA.

In Section \ref{sec:MatrixRepresentation} we bring together the two preceding sections by constructing a matrix representation of the algebra $\mathcal A$ and showing that the $PSL(2,\mathbb Z)$-action by automorphisms of $\mathcal A$ on the representation level is nothing but the conjugation with matrix-valued difference operators $\widehat D_A,\widehat D_B$ introduced in Section \ref{sec:SL2ZActionOnMatR3}. This constitutes the statement of our key Theorem \ref{th:HomomorphismToMatrixRepresentation} which concludes the main body of the text.

In Section \ref{sec:RelationToLaumon} we briefly discuss the origin of our formulas and relation to Affine Laumon spaces. The K-theory characters of the latter are given by a system of functions which reduce to Macdonald polynomials in the zero limit of elliptic parameters $p,s\rightarrow 0$. Although we leave the questions related to Macdonald polynomials and their generalizations beyond the scope of the current manuscript, we use this section to highlight that matrix representation of our algebra was first obtained by replacing Macdonald polynomials by K-theory character of an Affine Laumon space.

At last, we finish our paper with a summary of future directions given in Section \ref{sec:Conclusion}. We also provide a Mathematica package \cite{ArthamonovShakirov'2023-GitHub-EllipticDAHA} which realizes the main construction of the current paper with some examples which might be of some relevance for future applications.

\section{Preliminaries}
\label{sec:Preliminaries}

\subsection{Iterated $q$-Pochhammer symbols}
\begin{definition}
For every $n$-tuple of parameters
\begin{equation}
p_1,\dots,p_n\;\in\; \mathbb D^\times:=\mathbb D\backslash\{0\}=\big\{w\in\mathbb C\;\big|\;0<|w|<1\big\}
\end{equation}
define iterated $q$-Pochhammer symbol as an infinite product
\begin{equation}
    (z;p_1,\dots,p_n):=\prod_{i_1,\dots,i_n\in\mathbb Z_{\geq0}}(1-z p_1^{i_1}\dots p_n^{i_n}).
    \label{eq:IteratedPochhammerLessThanOne}
\end{equation}
\label{def:IteratedPochhammerInfiniteProduct}
\end{definition}
Note that the infinite product  (\ref{eq:IteratedPochhammerLessThanOne}) converges absolutely for all $z\in\mathbb C$ and $p_1,\dots,p_n\in\mathbb D^\times$ and hence it converges uniformly on every compact subset $K\subset\mathbb C\times\big(\mathbb D^\times\big)^n$ of its domain. As a result, we get analytic properties:
\begin{lemma}
The infinite product (\ref{eq:IteratedPochhammerLessThanOne}) provides a nonconstant holomorphic function of $n+1$ variables on $\mathbb C\times\left(\mathbb D^\times\right)^n$.
\end{lemma}
\begin{lemma}
For a fixed values of parameters $p_1,\dots,p_n\in\mathbb D^\times$, the iterated $q$-Pochhammer symbol (\ref{eq:IteratedPochhammerLessThanOne}) defines a nonconstant entire function of $z$.
\label{lemm:PochhammerNonConstantEntire}
\end{lemma}

Throughout this paper it is often convenient for us to work with functions of several complex variables on general open subsets $U\subset\mathbb C^n$, not necessarily connected. For this reason, we have to abstain from using the word ``domain'' in the complex-analytic sense. In a few instances where we absolutely have to use the word ``domain'', we always refer to set-theoretic domains. Now let $U=\bigcup_{\alpha\in I} U_\alpha$ be a decomposition of an open set $U\subset\mathbb C^n$ into connected components $U_\alpha$. Because $\mathbb C^n$ is locally connected, each $U_\alpha$ must be open. This, in particular, allows us to extend the definition of meromorphic functions as follows
\begin{definition} We say that $f:W\rightarrow\mathbb C$ is meromorphic on the open set $U\supset W$, if for every connected component $U_\alpha\subset U$, the restriction $f|_{W\cap U_\alpha}$ is meromorphic on $U_\alpha$. (In the sense of standard Definition VI.2.1 from \cite{FischerLieb'2012})
\end{definition}

\begin{remark}
Note that meromorphic functions on general open set $U\subset\mathbb C^n$ still form a ring, which is isomorphic to the product of rings over all connected components. This ring, however, is no longer a field and not even an integral domain unless $U$ is connected.
\end{remark}

\begin{lemma}
Iterated $q$-Pochhammer symbol $(z;p_1,\dots,p_n)$ uniquely extends to a larger set of parameters
\begin{equation*}
p_1,\dots,p_n\;\in\;\widecheck{\mathbb C}:=\big\{w\in\mathbb C\;\big|\;|w|\neq0,1\big\}
\end{equation*}
as a meromorphic function of $n+1$ variables on $\mathbb C\times\big(\widecheck{\mathbb C}\big)^n$ satisfying
\begin{equation}
(z;p_1,\dots,p_{j-1},p_j^{-1},p_{j+1},\dots,p_n)=\frac1{(p_j z,p_1,\dots,p_{j-1},p_j,p_{j+1},\dots,p_n)}.
\label{eq:PochhammerParameterInversion}
\end{equation}
\end{lemma}
\begin{proof}
Note that $\mathbb C\times\big(\widecheck{\mathbb C}\big)^n$ has $2^n$ connected components
\begin{equation}
\Omega_{\boldsymbol\epsilon}:=\mathbb C\times\big(\mathbb D^\times\big)^{\epsilon_1}\times\dots\times\big(\mathbb D^{\times}\big)^{\epsilon_1}
\label{eq:ConnectedComponentsOmegaEpsilon}
\end{equation}
labelled by $n$-tuples $\boldsymbol\epsilon=(\epsilon_1,\dots,\epsilon_n)\in \{\pm1\}^n$. Consider a meromorphic function of $n+1$-variables defined piecewise
\begin{equation}
(z;p_1,\dots,p_n):=\Big(\big(p_1^{(\epsilon_1-1)/2}\dots p_n^{(\epsilon_n-1)/2}z;p_1^{\epsilon_1},\dots,p_n^{\epsilon_n}\big)\Big)^{\epsilon_1\dots\epsilon_n}\qquad\textrm{for}\qquad (z,p_1,\dots,p_n)\in\Omega_{\boldsymbol\epsilon}.
\label{eq:GeneralPochhammerExplicit}
\end{equation}
This function satisfies (\ref{eq:PochhammerParameterInversion}). On the contrary, (\ref{eq:PochhammerParameterInversion}) implies (\ref{eq:GeneralPochhammerExplicit}), so this solution is unique.
\end{proof}

Finally, combining formula (\ref{eq:GeneralPochhammerExplicit}) with Lemma \ref{lemm:PochhammerNonConstantEntire} we also get single-variable property
\begin{corollary}
    For each fixed $n$-tuple of parameters $p_1,\dots,p_n\in\widecheck{\mathbb C}$, the extended $q$-Pochhammer symbol defines nonzero meromorphic function of one complex variable $z$ on $\mathbb C$.
    \label{cor:qPochhammerNonzeroMeromorphicFunctionOfOneVariable}
\end{corollary}

In what follows it is also convenient for us to augment the definitions above with the case of $n=0$ parameters, where we manually set
\begin{equation*}
(z;\emptyset):=\frac1{1-z}.
\end{equation*}

\begin{lemma}
For all $n\in\mathbb Z_{\geq0}$, the iterated $q$-Pochhammer symbols satisfy the following identities 
\begin{subequations}
\begin{align}
&(z;p_1,\dots,p_n)=(z;p_{\nu(1),\dots,\nu(n)})\qquad\textrm{for all}\quad\nu\in S_n,
\label{eq:qPochhammerPermutationSymmetry}\\
&(z;p_1,\dots,p_n)=(z;p_2,\dots,p_n)(zp_1;p_1,\dots,p_n),
\label{eq:IteratedPochhammerProductRules}
\end{align}
\end{subequations}
as meromorphic functions on $\mathbb C\times\big(\widecheck{\mathbb C}\big)^n$. Note that $(z;p_2,\dots,p_n)$ is viewed here as a meromorphic function of $n+1$ variables which is independent of $p_1$.
\end{lemma}
\begin{proof}
Due to the property (\ref{eq:qPochhammerPermutationSymmetry}) it is enough for us to prove the statement only on one connected component
\begin{equation*}
\Omega_{(+1,\dots,+1)}=\mathbb C\times\big(\mathbb D^\times\big)^n \quad\subset\quad \mathbb C\times\big(\widecheck{\mathbb C}\big)^n.
\end{equation*}
On this connected component, the iterated $q$-Pochhammer symbol is given by the infinite product formula (\ref{eq:IteratedPochhammerLessThanOne}) which makes identity (\ref{eq:qPochhammerPermutationSymmetry}) manifest. On the other hand (\ref{eq:IteratedPochhammerProductRules}) holds for partial products with rectangular bounds. But then since all three infinite products converge absolutely on $\Omega_{(+1,\dots,+1)}$ we conclude that (\ref{eq:IteratedPochhammerProductRules}) holds on $\Omega_{(+1,\dots,+1)}$.
\end{proof}

\subsection{Holomorphic functions on open dense subsets.}

Consider the following set of holomorphic functions of three variables
\begin{equation}
\mathbf S:=\left\{f:U\rightarrow\mathbb C\;\Big|\; f\;\textrm{is holomorphic on open dense}\;U\subset\big(\mathbb C^\times\big)^3\right\}.
\label{eq:HolomorphicFunctionsOnOpenDenseSubsets}
\end{equation}
Note that we do not require $U$ to be connected, in fact, in all our relevant examples we will always start with disconnected $U$, even when the function actually admits analytic continuation to a connected domain.
\begin{definition}
We say that the two functions $f_1:U_1\rightarrow\mathbb C$ and $f_2:U_2\rightarrow\mathbb C$ in $\mathbf S$ are \textit{equivalent} if there exists an open dense subset $V\subset U_1\cap U_2$ such that  $f_1|_V=f_2|_V$. In this case we write $f_1\sim f_2$.
\label{def:SimEquivalenceRelation}
\end{definition}
\begin{lemma}
The above relation $\sim$ is reflexive, symmetric, and transitive.
\end{lemma}
\begin{proof}
First two properties are satisfied by construction. For transitivity suppose that $f_1\sim f_2$ and $f_2\sim f_3$. Then there exists a pair of open dense subsets $V_{12},V_{13}\subset\big(\mathbb C^\times\big)^3$ such that
\begin{equation*}
    f_1\big|_{V_{12}}=f_2\big|_{V_{12}},\qquad f_2\big|_{V_{23}}=f_3\big|_{V_{23}}.
\end{equation*}
Because $V_{12}\cap V_{13}\subset\big(\mathbb C^\times\big)^3$ is open dense and
\begin{equation*}
f_1\big|_{V_{12}\cap V_{23}}=f_3\big|_{V_{12}\cap V_{23}}
\end{equation*}
we conclude that $f_1\sim f_3$.
\end{proof}
\begin{definition}
Hereinafter we denote the set of equivalence classes of functions in $\mathbf S$ by
\begin{equation*}
\mathbf R:=\bigslant{\mathbf S}{\sim}.
\end{equation*}
\end{definition}

\begin{lemma}
Let $f_1:U_1\rightarrow\mathbb C$ and $f_2:U_2\rightarrow\mathbb C$ be a pair of elements of $\mathbf S$ and suppose further that $f_1\sim f_2$, then
\begin{equation*}
f_1|_{U_1\cap U_2}=f_2|_{U_1\cap U_2}
\end{equation*}
\label{lemm:AgrreOnIntersection}
\end{lemma}
\begin{proof}
Because $f_1\sim f_2$, there exists an open dense subset $V\subset U_1\cap U_2$ such that $f_1|_V=f_2|_V$. Fix $p\in U_1\cap U_2$, because $V$ is dense in $\big(\mathbb C^\times\big)^3$ we can choose a sequence of points $\{p_i\}_{i\in\mathbb N}\subset V$ converging to $p$. Since both $f_1$ and $f_2$ are continuous at $p$ we get
\begin{equation*}
f_1(p)=\lim f_1(p_i)=\lim f_2(p_i)=f_2(p).
\end{equation*}
\end{proof}
\begin{corollary}
Each equivalence class $\underline f\in\mathbf R$ has unique maximal analytic continuation $F\in\underline{f}$.
\end{corollary}
\begin{proof}
Denote by $W\subset\big(\mathbb C^\times\big)^3$ the set of all points $w\in\big(\mathbb C^\times\big)^3$ such that in $\underline{f}$ there exists at least one representative $g_w:U_w\rightarrow\mathbb C$ with $U_w\ni w$. Note that $W\subset\big(\mathbb C^\times\big)^3$ must be open and dense.

Now let $F:W\rightarrow\mathbb C$ be defined at each point as $F(w):=g_w(w)$. Note that by Lemma \ref{lemm:AgrreOnIntersection}, this definition doesn't depend on a particular choice of representative $g_w$. By construction $F$ is analytic on $W$ and, moreover, for every element $g:U\rightarrow\mathbb C$ of the equivalence class $\underline{f}$ we necessarily have $U\subset W$ and $g=F|_U$. Hence $F\in\underline{f}$ is a common analytic continuation. For uniqueness, note that any other maximal analytic continuation $F':U'\rightarrow\mathbb C$ must have a domain $U'\subset W$ included in $W$.
\end{proof}

\begin{proposition}
The set $\mathbf R$ of equivalence classes forms a ring with respect to addition and multiplication defined by
\begin{equation}
\underline{f_1}+\underline{f_2}:=\underline{f_1\big|_{U_1\cap U_2}+f_2\big|_{U_1\cap U_2}},\qquad \underline{f_1}*\underline{f_2}:=\underline{f_1\big|_{U_1\cap U_2}f_2\big|_{U_1\cap U_2}}.
\label{eq:AdditionMultiplicationR}
\end{equation}
Here $\underline{f_1},\underline{f_2}\in\mathbb R$ denote the equivalence classes of functions $f_1:U_1\rightarrow\mathbb C$ and $f_2:U_2\rightarrow\mathbb C$ respectively.
\end{proposition}
\begin{proof}
First, note that (\ref{eq:AdditionMultiplicationR}) doesn't depend on the choice of representative of the equivalence class and thus indeed defines a pair of binary operations $+,*:\mathbf R\times\mathbf R\rightarrow\mathbf R$. All necessary properties of the two binary operations then follow from the properties of addition and multiplication of the $\mathbb C$-valued functions with the common domain.
\end{proof}

Note that ring $\mathbf R$ is not a field and not even an integral domain; it has zero divisors because we have allowed for disconnected domains in (\ref{eq:HolomorphicFunctionsOnOpenDenseSubsets}). However, the subset $\mathbf R^\times\subset\mathbf R$ of elements which have multiplicative inverse in $\mathbf R$ can be characterized by a simple criterion which, remarkably enough, can be formulated in terms of an arbitrary representative.
\begin{lemma}
An element $\underline f\in\mathbf R$ represented by a holomorphic function $f:U\rightarrow\mathbb C$ is a unit 
\begin{equation}
    \underline f\in\mathbf R^\times\qquad\textrm{iff}\qquad f\big|_{U_\alpha}\not\equiv0\quad\textrm{for every connected component}\quad U_\alpha\subset U.
\label{eq:UnitInRCriterion}
\end{equation}
\label{lemm:UnitInRCriterion}
\end{lemma}
\begin{proof}
Let $U=\bigcup_{\alpha\in I}U_{\alpha}$ be a decomposition of $U$ into connected components $U_\alpha$ labelled by the elements of the indexing set $I$. Because $\big(\mathbb C^\times\big)^3$ is locally connected, each $U_\alpha$ must be an open connected set.

First we show that condition in the statement of the Lemma is necessary. To this end, suppose that the right hand side of (\ref{eq:UnitInRCriterion}) is not satisfied, then there exists at least one connected component $U_{\alpha_0}$ such that $f|_{U_{\alpha_0}}\equiv0$ is identically zero on it.
But, on the other hand, characteristic function $\chi_{U_{\alpha_0}}:U\rightarrow\{0,1\}$ is a holomorphic function on $U$ which is not equivalent to zero $\underline{\chi_{U_{\alpha_0}}}\in\mathbf R\backslash\{0\}$.
Because $\chi_{U_{\alpha_0}}f\equiv0$ we conclude that $\underline{f}$ is a zero divisor in $\mathbf R$ and hence cannot have a multiplicative inverse.

It remains to show that the right hand side of (\ref{eq:UnitInRCriterion}) is sufficient. For each connected component, let
\begin{equation*}
V_\alpha:=\{u\in U_\alpha\;|\;f(u)=0\}
\end{equation*}
be a set of zeros of $f$ on this component. Then $V_\alpha$ must be a relatively closed nowhere dense subset of $U_\alpha$. As a result
\begin{equation*}
W:=\bigcup_{\alpha\in I}U_{\alpha}\backslash V_\alpha\;\subset\;\big(\mathbb C^\times\big)^3
\end{equation*}
is an open dense subset of $\big(\mathbb C^\times\big)^3$. Define a holomorphic function $g:W\rightarrow\mathbb C$ by $g(w)=1/f(w)$ and note that $\underline f*\underline g=1$ in $\mathbf R$.
\end{proof}

\begin{proposition}
    Let $U\subset\big(\mathbb C^\times\big)^3$ be an open dense subset. Suppose further that $f:W\rightarrow\mathbb C$ be a meromorphic function on $U\supset W$ such that
    \begin{equation}
    f\big|_{W\cap U_\alpha}\not\equiv0\quad\textrm{for every connected component}\quad U_\alpha\subset U.
    \label{eq:NonzeroMeromorphicFunctionOnEveryConnectedComponent}
    \end{equation}
    Then $\underline f\in\mathbf R^\times$ is a unit of the ring $\mathbf R$.
\label{prop:MeromorphicFunctionIsAUnitCriterion}
\end{proposition}
\begin{proof}
Let $U=\bigcup_{\alpha\in I}U_\alpha$ be a decomposition of $U$ into connected components. On each connected component $U_\alpha$, the restriction $f|_{W\cap U_\alpha}$ is meromorphic in the sense of conventional Definition VI.2.1 from \cite{FischerLieb'2012}. In particular, conventional definition implies that $U_\alpha\backslash W$ is an open dense subset of $U_\alpha$ and hence
\begin{equation}
W=\bigcup_{\alpha\in I}W\cap U_\alpha\;\subset\; U\;\subset\;\big(\mathbb C^\times\big)^3
\label{eq:DecompositionWConnectedComponents}
\end{equation}
is an open dense subset of $U$. Next, because $U\subset\big(\mathbb C^\times\big)^3$ is dense in $\big(\mathbb C^\times\big)^3$, so is $W\subset\big(\mathbb C^\times\big)^3$ and hence $f\in\mathbf S$ is a holomorphic function on an open dense subset of $\big(\mathbb C^\times\big)^3$.

Now by Proposition 1.3 from \cite{FischerLieb'2012}, $U_{\alpha}\cap W$ must be connected, hence (\ref{eq:DecompositionWConnectedComponents}) is, in fact, a decomposition of $W$ into connected components. Taking this into account we conclude that (\ref{eq:NonzeroMeromorphicFunctionOnEveryConnectedComponent}) implies that $f:W\rightarrow\mathbb C$ satisfies condition of Lemma \ref{lemm:UnitInRCriterion} and thus represents a unit $\underline f\in\mathbf R^\times$ of the ring $\mathbf R$.
\end{proof}

A particular specialization of the $q$-Pochhammer symbol will serve as a major building block of formulas in this paper. Below we apply Proposition \ref{prop:MeromorphicFunctionIsAUnitCriterion} to show that it is, in fact, an invertible element of the ring $\mathbf R$.
\begin{corollary}
Fix an $n$-tuple of nonzero integer vectors
\begin{equation*}
(a_i,b_i)\in \mathbb Z\times\mathbb Z\backslash\{(0,0)\}\qquad \textrm{for}\quad 1\leq i\leq n.
\end{equation*}
The following specialization of the iterated $q$-Pochhammer symbol represents a unit of the ring $\mathbf R$
\begin{equation}
\underline{(Q^{m_1}p^{m_2}s^{m_3};p^{a_1}s^{b_1},\dots,p^{a_n}s^{b_n})}\in\mathbf R^\times\qquad\textrm{for all}\quad m_1\in\mathbb Z\backslash\{0\},\quad m_2,m_3\in\mathbb Z.
\label{eq:SpecializedPochhammerEquivalenceClassMultiplicative}
\end{equation}
\end{corollary}
\begin{proof}
For an arbitrary nonzero integer vector $(a,b)\in\mathbb Z\times\mathbb Z\backslash\{(0,0)\}$ define
\begin{equation*}
V_{a,b}:=\left\{(p,s)\in\big(\mathbb C^\times\big)^2\;\Big|\;|p^as^b|\neq1\right\}.
\end{equation*}
Now consider an open dense subset of $\big(\mathbb C^\times\big)^2$ given by
\begin{equation*}
W:=\big(\mathbb C^\times\big)^2\Big\backslash\bigcup_{i=1}^n V_{a_i,b_i}
\end{equation*}
and let $W=\bigcup_{\alpha\in I}W_\alpha$ be a decomposition of $W$ into connected components. Next, we define
\begin{equation*}
U:=\bigcup_{\alpha\in I}U_\alpha,\qquad\textrm{where}\quad U_\alpha:=\mathbb C^\times\times W_\alpha.
\end{equation*}
Because $\mathbb C^\times$ is connected, $U_\alpha$ are precisely the connected components of $U$. We will show that (\ref{eq:SpecializedPochhammerEquivalenceClassMultiplicative}) defines a meromorphic function of three variables on $U$ which is not identically zero on every connected component $U_\alpha\subset U$. By Proposition \ref{prop:MeromorphicFunctionIsAUnitCriterion} it is enough to conclude that (\ref{eq:SpecializedPochhammerEquivalenceClassMultiplicative}) is a unit of the ring $\mathbf R$.

Consider an auxiliary holomorphic map
\begin{equation*}
g:U\rightarrow\mathbb C\times\big(\widecheck{\mathbb C}\big)^n,\quad g(Q,p,s):=(Q^{m_1}p^{m_2}s^{m_3},p^{a_1}s^{b_1},\dots,p^{a_n}s^{b_n}).
\end{equation*}
Now fix an arbitrary connected component $U_\alpha\subset U$. Because $g$ is continuous, the image of the connected component
\begin{equation*}
g(U_\alpha)\subset \Omega_{\boldsymbol{\epsilon}},\qquad \textrm{for some}\quad  \boldsymbol{\epsilon}\in\{\pm1\}^n.
\end{equation*}
must be contained in exactly one of the connected components (\ref{eq:ConnectedComponentsOmegaEpsilon}) of the codomain $\mathbb C\times\big(\widecheck{\mathbb C}\big)^n$.

From (\ref{eq:GeneralPochhammerExplicit}) we know that on each connected component $\Omega_\epsilon$, the iterated $q$-Pochhammer is given by a ratio of two holomorphic functions $F_{\boldsymbol\epsilon},G_{\boldsymbol\epsilon}:\Omega_\epsilon\rightarrow\mathbb C$, where $F_{\boldsymbol\epsilon},G_{\boldsymbol\epsilon}\not\equiv0$ as a holomorphic function of $n+1$ variables. Precomposing the two functions with $g$ we define another pair of holomorphic functions, now of three variables
\begin{equation*}
\widetilde F_{\alpha},\widetilde G_{\alpha}: U_\alpha\rightarrow\mathbb C,\qquad \widetilde F_{\alpha}:=F_{\boldsymbol\epsilon}\circ \Big(g\big|_{U_\alpha}\Big),\qquad \widetilde G_{\alpha}:=G_{\boldsymbol\epsilon}\circ \Big(g\big|_{U_\alpha}\Big).
\end{equation*}

Note that in order to show that $\widetilde F_\alpha,\widetilde G_\alpha\not\equiv0$ is not identically zero it is enough to show that at least one $(p,s)$-slice of this functions is not identically zero. Fix further an arbitrary pair of parameters $(p,s)\in W_\alpha$ and let
\begin{align*}
&h_{p,s}:\mathbb C^\times\rightarrow\mathbb C,\qquad h_{p,s}=Q^{m_1}p^{m_2}s^{m_3},\qquad m_1\neq0,\\
&H_{p,s}:\mathbb C\rightarrow\mathbb C,\qquad H_{p,s}(z):=G_{\boldsymbol\epsilon}(z,p^{a_1}s^{b_1},\dots,p^{a_n}s^{b_n}).
\end{align*}
The corresponding $(p,s)$-slice of $\widetilde G_{\alpha}$ then factors as
\begin{equation}
    \widetilde G_\alpha(Q,p,s)=H_{p,s}(h_{p,s}(Q))\qquad\textrm{for all}\quad Q\in\mathbb C^\times.
    \label{eq:psSliceGEpsilon}
\end{equation}

By Lemma \ref{lemm:PochhammerNonConstantEntire} combined with formula (\ref{eq:GeneralPochhammerExplicit}) we know that $H_{p,s}$ must be nonconstant holomorphic function of a single complex variable. In a single variable case we have a luxury of the Open Mapping Theorem to conclude that $H_{p,s}(\mathbb C)\subset\mathbb C$ is open. Because $h_{p,s}(\mathbb C^\times)=\mathbb C$, for the fixed $(p,s)\in W_\alpha$ the image
\begin{equation*}
\left\{\widetilde{G}_\alpha(Q,p,s)\;\big|\;Q\in\mathbb C^\times\right\}=H_{p,s}(h_{p,s}(\mathbb C^\times))=H_{p,s}(\mathbb C)\quad\subset\quad\mathbb C
\end{equation*}
is an open subset of the complex plane. Which, in particular, implies that $\widetilde{G}_\alpha$ cannot be a zero function.
\end{proof}

\subsection{Plethystic exponent beyond formal power series}

Consider a ring of Laurent polynomials in two variables
\begin{equation*}
\mathbf P:=\mathbb Z[p^{\pm1},s^{\pm}].
\end{equation*}
Now let $\mathbf M$ be a $\mathbf P$-module
\begin{equation*}
\mathbf M:=\bigoplus_{m\in\mathbb Z\backslash\{0\}}Q^m\mathbb Z[p^{\pm1},s^{\pm1}].
\end{equation*}
which consists of Laurent polynomials in three variables $(Q,p,s)$ which have no zero degree terms in $Q$.

\begin{definition}
Let $n\in\mathbf Z_{\geq0}$ be a nonnegative integer, $\mu\in\mathbf M$ be a Laurent polynomial, and
\begin{equation*}
    (a_i,b_i)\in\mathbb Z\times\mathbb Z\backslash\{(0,0)\}\qquad\textrm{for}\quad i\in\{1,\dots,n\}
\end{equation*}
be an $n$-tuple of nonzero integer vectors. We refer to $n+1$-tuple of elements
\begin{equation*}
\left(\mu\,\mathbf{:}\,1-p^{a_1}s^{b_1}\,\mathbf{:}\,1-p^{a_2}s^{b_2}\,\mathbf{:}\,\dots\,\mathbf{:}\,1-p^{a_n}s^{b_n}\right)
\end{equation*}
as a formal fraction $(\mu\,\mathbf{:}\,\Delta)$ with numerator $\mu$ and denominator $\Delta=(1-p^{a_1}s^{b_1},\dots,1-p^{a_n}s^{b_n}).$ Hereinafter, we denote the set of all possible formal fractions by $\mathbf Q$.
\label{def:FormalFraction}
\end{definition}
Note that, for the moment we distinguish different ``factorizations'' of denominator.

\begin{definition}
Consider a function $\mathrm{pexp}:\mathbf Q\rightarrow\mathbf R^\times$, which is first defined on formal fractions with monomial numerators
\begin{equation}
\mathrm{pexp}\left(Q^{m_1}p^{m_2}s^{m_3}\,\mathbf{:}\, 1-p^{a_1}s^{b_1}\,\mathbf{:}\,\dots\,\mathbf{:}\,1-p^{a_n}s^{b_n}\right):=\underline{(Q^{m_1}p^{m_2}s^{m_3};p^{a_1}s^{b_1},\dots,p^{a_n}s^{b_n})}\qquad\in\quad\mathbf R^\times
\label{eq:pexpFormalFractionsMonomialDefinition}
\end{equation}
where $m_1\neq0$ and then uniquely extended to general formal fractions by the exponential property:
\begin{subequations}
\begin{align}
\mathrm{pexp}(0:\Delta)&=1,\\
\mathrm{pexp}(-\mu:\Delta)&=\frac1{\mathrm{pexp}(\mu:\Delta)},\\
\mathrm{pexp}(\mu+\mu':\Delta)&=\mathrm{pexp}(\mu:\Delta)\mathrm{pexp}(\mu':\Delta).
\label{eq:pexpFormalFractionsAdditivity}
\end{align}
\label{eq:ExponentialPropertyFormalFractions}
\end{subequations}
satisfied for all formal fractions $(\mu:\Delta)\in\mathbf Q$ and $(\mu':\Delta)\in\mathbf Q$.
\label{def:PExpFormalFraction}
\end{definition}

\begin{lemma}
The above function $\mathrm{pexp}:\mathbf Q\rightarrow\mathbf R^\times$ satisfies the following identities
\begin{subequations}
\begin{align}
&\mathrm{pexp}(\mu\,\mathbf{:}\,1-p^{a_1}s^{b_1}\,\mathbf{:}\,\dots\,\mathbf{:}\,1-p^{a_n}s^{b_n})=\mathrm{pexp}(\mu\,\mathbf{:}\,1-p^{a_{\nu(1)}}s^{b_{\nu(1)}}\,\mathbf{:}\,\dots\,\mathbf{:}\,1-p^{a_{\nu(n)}}s^{b_{\nu(n)}}),
\label{eq:pexpFormalFactorizationsOfDenominatorIdentity}\\
&\mathrm{pexp}(\mu (1-p^{a_0}s^{b_0})\,\mathrm{:}\,1-p^{a_0}s^{b_0}\,\mathbf{:}\,\dots\,\mathbf{:}\,1-p^{a_n}s^{b_n})=\mathrm{pexp}(\mu\,\mathbf{:}\,1-p^{a_1}s^{b_1}\,\mathbf{:}\,\dots,\,\mathbf{:}\,1-p^{a_n}s^{b_n}),
\label{eq:pexpCancellationFormalDenominator}
\end{align}
\end{subequations}
for all $\mu\in\mathbf M,\;\nu\in S_n,$ and $n+1$-tuples of nonzero vectors
\begin{equation*}
(a_i,b_i)\in\mathbb Z\times\mathbb Z\backslash\{(0,0)\}\qquad i\in\{0,\dots,n\}.
\end{equation*}
\label{lemm:pexpEquivalenceSingleFactor}
\end{lemma}
\begin{proof}
First note that because of the exponential property (\ref{eq:ExponentialPropertyFormalFractions}) it is enough to prove the Lemma only for the case when $\mu=Q^{m_1}p^{m_2}s^{m_3}$ is a monomial with $m_1\in\mathbb Z\backslash\{0\},\;m_2,m_3\in\mathbb Z$. The first identity (\ref{eq:pexpFormalFactorizationsOfDenominatorIdentity}) then follows from the symmetry (\ref{eq:qPochhammerPermutationSymmetry}) of the iterated $q$-Pochhammer symbol. As for the second identity (\ref{eq:pexpCancellationFormalDenominator}), we have
\begin{align*}
&\mathrm{pexp}(Q^{m_1}p^{m_2}s^{m_3}(1-p^{a_0}s^{b_0})\,\mathbf{:}\,1-p^{a_0}s^{b_0}\,\mathbf{:}\,\dots\,\mathbf{:}\,1-p^{a_n}s^{b_n})\\[0.3em]
&\quad\overset{(\ref{eq:ExponentialPropertyFormalFractions})}{=} \frac{\mathrm{pexp}(Q^{m_1}p^{m_2}s^{m_3}\,\mathbf{:}\,1-p^{a_0}s^{b_0}\,\mathbf{:}\,\dots\,\mathbf{:}\,1-p^{a_n}s^{b_n})}{\mathrm{pexp}(Q^{m_1}p^{m_1+a_0}s^{m_2+b_0}\,\mathbf{:}\,1-p^{a_0}s^{b_0}\,\mathbf{:}\,\dots\,\mathbf{:}\,1-p^{a_n}s_{b_n})}\\[0.3em]
&\quad\overset{(\ref{eq:pexpFormalFractionsMonomialDefinition})}{=} \frac{\underline{(Q^{m_1}p^{m_2}s^{m_3};p^{a_0}s^{b_0},\dots,p^{a_n}s^{b_n})}}{\;\underline{(Q^{m_1}p^{m_1+a_0}s^{m_2+b_0};p^{a_0}s^{b_0},\dots,p^{a_n}s_{b_n}})\;}\\[0.3em]
&\quad\overset{(\ref{eq:IteratedPochhammerProductRules})}{=}\underline{(Q^{m_1}p^{m_2}s^{m_3};p^{a_1}s^{b_1},\dots,p^{a_n}s^{b_n})}\\[0.3em]
&\quad\overset{(\ref{eq:pexpFormalFractionsMonomialDefinition})}{=}\mathrm{pexp}(Q^{m_1}p^{m_2}s^{m_3}\,\mathbf{:}\,1-p^{a_1}s^{b_1}\,\mathbf{:}\,\dots\,\mathbf{:}\,1-p^{a_n}s^{a_n}).
\end{align*}
\end{proof}

Recall that the ring of Laurent polynomials in two variables $\mathbf P=\mathbb Z[p^{\pm1},s^{\pm1}]$ is a localization of an integral domain, so it must be an integral domain itself. Consider a multiplicative subset of $\mathbf P$ generated by two term polynomials of the special type
 \begin{equation*}
 \mathbb S_\times:=\left\langle1- p^as^b\;\big|\; (a,b)\in\mathbb Z\times\mathbb Z\backslash\{(0,0)\}\right\rangle\quad\subset\quad\mathbf P
 \end{equation*}
and let $\mathbb S_\times^{-1}\mathbf M$ be a $\mathbb S_\times^{-1}\mathbf P$-module obtained by localization of $\mathbf M$.

Every formal fraction $(\mu:\Delta)\in\mathbf Q$ represents an element of the localized module
\begin{equation*}
\overline{(\mu:\Delta)}=\overline{(\mu\,\mathbf{:}\,1-p^{a_1}s^{b_1}\,\mathbf{:}\,\dots\,\mathbf{:}\,1-p^{a_n}s^{b_n})}:=\frac{\mu}{\prod_{i=1}^n(1-p^{a_i}s^{b_i})}\qquad\in\quad\mathbb S_\times^{-1}\mathbf M.
\end{equation*}
Our next goal is to show that value of $\mathrm{pexp}$ does not depend on the particular choice of representative.

\begin{proposition}
    Suppose that two formal fractions $(\mu:\Delta)\in\mathbf Q$ and $(\mu':\Delta')\in\mathbf Q$ represent the same element of the localized module:
    \begin{equation}
    \frac{\mu}{\prod_{i=1}^n(1-p^{a_i}s^{b_i})}=\frac{\mu'}{\prod_{i=1}^{n'}(1-p^{a_i'}s^{b_i'})}\qquad\in\quad\mathbb S_\times^{-1}\mathbf M.
    \label{eq:FormalFractionsEquivalent}
    \end{equation}
    Then
    \begin{equation*}
    \mathrm{pexp}(\mu\,\mathbf{:}\,\Delta)=\mathrm{pexp}(\mu'\,\mathbf{:}\,\Delta')\qquad \in\quad \mathbf R^\times
    \end{equation*}
\label{prop:pexpFormalFractionsExquivalence}
\end{proposition}
\begin{proof}
Use Lemma \ref{lemm:pexpEquivalenceSingleFactor} to bring both formal fractions to the common denominator.
\end{proof}

In other words, the plethystic exponential map $\mathrm{pexp}:\mathbf Q\rightarrow\mathbf R^\times$ factors through the map
\begin{equation*}
\mathrm{pexp}:\;\mathbb S_\times^{-1}\mathbf M\rightarrow\mathbf R^\times
\end{equation*}
which we denote by the same symbol throughout the rest of the text.

\begin{example}
Consider two equivalent ways of computing the following Plethystic exponent
\begin{equation*}
\mathrm{pexp}\left(\frac Q{p+p^{-1}}\right)=\mathrm{pexp}\left(\frac{Qp^{-1}(1-p^{-2})}{1-p^{-4}}\right)=\frac{\underline{(Qp^{-1};p^{-4})}}{\;\underline{(Qp^{-3};p^{-4})}\;} \overset{(\ref{eq:PochhammerParameterInversion})}{=} \frac{\underline{(Qp;p^4)}}{\;\underline{(Qp^3;p^4)}\;} =\mathrm{pexp}\left(\frac{Qp(1-p^2)}{(1-p^4)}\right)
\end{equation*}
\label{ex:PlethysticTwoWays}
\end{example}

\section{Action of $SL(2,\mathbb Z)$}
\label{sec:SL2ZActionOnMatR3}
\subsection{Operators $\widehat D_A$ and $\widehat D_B$}

\begin{definition}
Consider a family of holomorphic automorphisms of $\big(\mathbb C^\times\big)^3$ labelled by elements of $SL(2,\mathbb Z)$:
\begin{align}
&\delta_{(p,s)\mapsto (p^as^b,p^cs^d)}:\big(\mathbb C^\times\big)^3\rightarrow\big(\mathbb C^\times\big)^3,\quad (Q,p,s)\mapsto(Q,p^as^b,p^cs^d),\qquad\textrm{where}\quad \left(\begin{array}{cc}a&b\\c&d\end{array}\right)\in SL(2,\mathbb Z).
\label{eq:DeltaC3Definiton}
\end{align}
\end{definition}
\begin{lemma}
The induced action of (\ref{eq:DeltaC3Definiton}) on functions gives rise to an automorphism of the ring $\mathbf R$
\begin{equation}
\widehat{\delta}_{(p,s)\mapsto(p^as^b,p^cs^d)}:\mathbf R\rightarrow\mathbf R,\qquad \underline{f}\mapsto\underline{f\circ\delta_{(p,s)\mapsto(p^as^b,p^cs^d)}}.
\label{eq:DeltaHatRDefinition}
\end{equation}
\end{lemma}
\begin{proof}
Because (\ref{eq:DeltaC3Definiton}) is a holomorphic authomorphism of $\big(\mathbb C^\times\big)^3$, it maps open dense subsets to open dense subsets, so we have a map
\begin{equation}
\widehat\delta_{(p,s)\mapsto(p^as^b,p^cs^d)}:\mathbf S\rightarrow\mathbf S,\qquad \widehat\delta_{(p,s)}(f):=f\circ \delta_{(p,s)\mapsto(p^as^b,p^cs^d)}.
\label{eq:DeltaHatSDefinition}
\end{equation}
At the same time, note that (\ref{eq:DeltaHatSDefinition}) preserves equivalence relation $\sim$ from Definition \ref{def:SimEquivalenceRelation} and hence descends to an automorphism (\ref{eq:DeltaHatRDefinition}) of the quotient space $\mathbf R=\mathbf S\big/\sim$.
\end{proof}

\begin{definition}
Consider a pair of invertible matrices $D_A,D_B\in Mat_{3\times 3}(\mathbf R)$ given by
\begin{align*}
D_A := 
\left( 
\begin{array}{ccc} 
1 & 0 & 0 \\
0 & -i Q & 0 \\
0 & 0 & -1
\end{array} 
\right)
\end{align*}

\begin{gather*}
\begin{aligned} 
D_B:=\left( 
\begin{array}{ccc} 
\mathrm{pexp}\left( \dfrac{-(Q^8 - Q^{-8})p s^{-1}(s^2 + 2 ps + p)}{(1-s^2)(1-p^2/s^2)} \right) & \dfrac{2Q}{1-Q^4} \mathrm{pexp}\left( \dfrac{-(Q^8 - Q^{-8})p s^{-1}(s + 2 ps + p)}{(1-s^2)(1-p^2/s^2)} \right) & \dfrac{2Q^4}{(1-Q^4)^2} \mathrm{pexp}\left( \dfrac{-(Q^8 - Q^{-8})p s^{-1}(s^2 + 2 ps + p)}{(1-s^2)(1-p^2/s^2)} \right) \\
-\dfrac{(1-Q^4)}{Q^{3}} \mathrm{pexp}\left( \dfrac{-(Q^8 - Q^{-8})p(s + 2 p + 1)}{(1-s^2)(1-p^2/s^2)} \right)  & 0 & \dfrac{2Q}{1-Q^4} \mathrm{pexp}\left( \dfrac{-(Q^8 - Q^{-8})p(s + 2 p + 1)}{(1-s^2)(1-p^2/s^2)} \right) \\
\dfrac{(1-Q^4)^2}{2Q^4} \mathrm{pexp}\left( \dfrac{-(Q^8 - Q^{-8})p s^{-1}(s^2 + 2 ps + p)}{(1-s^2)(1-p^2/s^2)} \right) & \dfrac{-(1-Q^4)}{Q^{3}} \mathrm{pexp}\left( \dfrac{-(Q^8 - Q^{-8})p s^{-1}(s + 2 ps + p)}{(1-s^2)(1-p^2/s^2)} \right) & \mathrm{pexp}\left( \dfrac{-(Q^8 - Q^{-8})p s^{-1}(s^2 + 2 ps + p)}{(1-s^2)(1-p^2/s^2)} \right)
\end{array} 
\right)
\end{aligned}
\end{gather*}
and let $\widehat D_A,\widehat D_B:\mathbf R^3\rightarrow\mathbf R^3$ be a pair of operators defined as
\begin{equation*}
\widehat D_A:=D_A\circ \widehat{\delta}_{(p,s)\mapsto (ps,s)},\qquad \widehat D_B:=D_B\circ\widehat\delta_{(p,s)\mapsto(p,s/p)}.
\end{equation*}
\end{definition}
In what follows, it will also be convenient for us to introduce the following notation
\begin{equation}
\widehat S:=\widehat D_A\widehat D_B\widehat D_A.
\label{eq:SHatDefinition}
\end{equation}

\begin{theorem}
    Operators $\widehat D_A,\widehat D_B:\mathbf R^3\rightarrow\mathbf R^3$ define a projective representation of $SL(2,\mathbb Z)$, namely
    \begin{equation}
    \widehat D_A\widehat D_B\widehat D_A\;\propto\; \widehat D_B\widehat D_A\widehat D_B,\qquad \left(\widehat D_A\widehat D_B\widehat D_A\right)^4\;\propto\;\mathrm{Id}.
    \label{eq:SL2ZRelationsHatD}
    \end{equation}
\label{th:SL2ZRelationsHatD}
\end{theorem}
\begin{proof}
Equation (\ref{eq:SL2ZRelationsHatD}) is equivalent to the following condition on $3\times 3$-matrices
\begin{equation*}
D_A D_B(ps,s) D_A\;\propto\; D_B(p,s) D_A D_B(s,s/p).
\label{eq:usefulrelation}
\end{equation*}
This identity is proved by direct computation, where we expand both sides and use exponential properties (\ref{eq:ExponentialPropertyFormalFractions}).

To prove the second identity, we first take the square of the element (\ref{eq:SHatDefinition}), which reads
\begin{equation}
\Big(\widehat S\Big)^2=\left(\begin{array}{ccc}
 4 \textrm{pexp}\left(\frac{\left(Q^8-Q^{-8}\right) (2 s^2+s)}{\left(1-s^2\right)}\right) & 0 & 0 \\
 0 & 4 \textrm{pexp}\left(\frac{\left(Q^8-Q^{-8}\right) 2s^2}{\left(1-s^2\right)}\right) & 0 \\
 0 & 0 & 4 \textrm{pexp}\left(\frac{\left(Q^8-Q^{-8}\right) (2 s^2+s)}{Q^8 \left(1-s^2\right)}\right)
\end{array}\right)\circ\widehat\delta_{(p,s)\mapsto(p^{-1},s^{-1})}.
\label{eq:SHat2Explicit}
\end{equation}
Now taking the square of the above we get $\big(\widehat S\big)^4=16\mathrm{pexp}\left(2(Q^8-Q^{-8})\right)\mathrm{Id}\;\propto\;\mathrm{Id}$ which concludes the proof.
\end{proof}

\begin{corollary}
    We have an action of $SL(2,\mathbf Z)$ by automorphisms of $Mat_{3\times 3}(\mathbf R)$
    \begin{equation*}
    \varphi:SL(2,\mathbb Z)\rightarrow \mathrm{Aut}\left(Mat_{3\times3}(\mathbf R)\right)
    \end{equation*}
    defined on generators of $SL(2,\mathbb Z)$
    \begin{equation*}
        d_A:=\left(\begin{array}{cc}1&1\\0&1\end{array}\right),\qquad d_B:=\left(\begin{array}{cc}1&0\\-1&1\end{array}\right)
    \end{equation*}
    as
    \begin{equation*}
    \varphi(d_A):M\mapsto \widehat D_A^{-1}M\widehat D_A,\qquad \varphi(d_B):M\mapsto \widehat D_B^{-1}M\widehat D_B,\qquad\textrm{for all}\quad M\in Mat_{3\times3}(\mathbf R).
    \end{equation*}
\end{corollary}

\section{The Algebra}
\label{sec:AlgebraK2}
\subsection{Generators and Relations}
Denote by $Free_2=\langle a,b\rangle$ a free group with two generators and let $\sigma$ be an order two automorphism
\begin{equation*}
\sigma:Free_2\rightarrow Free_2,\qquad \sigma:\left\{\begin{array}{l}a\mapsto b,\\b\mapsto a.\end{array}\right.
\end{equation*}

Fix a ground field $\mathbf k:=\mathbb C(Q)$ to be the field of rational functions in variable $Q$ and let
\begin{align*}
\mathcal F=\mathbf k\big\langle\Oa^{(g)},\Ob^{(g)}\;\big|\;g\in Free_2\big\rangle.
\end{align*}
be a free $\mathbf k$-algebra with countably many generators. Consider countably many polynomials in $\mathcal F$ labelled by various tuples of elements $g,g_1,g_2,g_3\in Free_2$
\begin{subequations}
\begin{align}
&\R_{0,A}^{(g)}:=\Oa^{(ag)}-\Oa^{(g)},\qquad\R_{0,B}^{(g)}:=\Ob^{(bg)}-\Ob^{(g)},
\label{eq:Relation0Xg}\\[1em]
&\R_{1,X}^{(g)}:=\mathcal O_X^{(g)}\mathcal O_X^{(g)}-\mathcal O_X^{(1)}\mathcal O_X^{(1)},\qquad X\in\{A,B\},
\label{eq:Relation1Xg}\\[1em]
&\R_2^{(g_1,g_2,g_3)}:=\Ob^{(g_1)}\Oa^{(g_2)}\Ob^{(g_3)},
\label{eq:Relation2g1g2g3}\\
&\R_3^{(g_1,g_2,g_3)}:=\Oa^{(g_1)}\Ob^{(g_2)}\Oa^{(g_3)},
\label{eq:Relation3g1g2g3}\\[1em]
&\R_{4,X}^{(g)}:=\mathcal O_X^{(1)}\mathcal O_X^{(1)}\mathcal O_X^{(g)}+\big(Q^2-Q^{-2}\big)^2\mathcal O_X^{(g)},\qquad X\in\{A,B\},
\label{eq:Relation4Xg}\\
&\R_{5,X}^{(g)}:=\mathcal O_X^{(g)}\mathcal O_X^{(1)}\mathcal O_X^{(1)}+\big(Q^2-Q^{-2}\big)^2\mathcal O_X^{(g)},\qquad X\in\{A,B\},
\label{eq:Relation5Xg}\\[1em]
&\R_6^{(g)}:=\Oa^{(1)}\Oa^{(1)}\Ob^{(g)}+\big(Q^2-Q^{-2}\big)^2\Ob^{(g)}+\Ob^{(g)}\Oa^{(1)}\Oa^{(1)},
\label{eq:Relation6g}\\
&\R_7^{(g)}:=\Ob^{(1)}\Ob^{(1)}\Oa^{(g)}+\big(Q^2-Q^{-2}\big)^2\Oa^{(g)}+\Oa^{(g)}\Ob^{(1)}\Ob^{(1)},
\label{eq:Relation7g}\\[1em]
&\R_8^{(g_1,g_2)}:=\Oa^{(g_2)}\Oa^{(g_1abag_2)}\Ob^{(ag_2)}+\Ob^{(ag_2)}\Oa^{(g_1abag_2)}\Oa^{(g_2)}-\Ob^{(g_2)}\Ob^{(\sigma(g_1)g_2)}\Ob^{(g_2)},\\
&\R_9^{(g_1,g_2)}:=\Ob^{(g_2)}\Ob^{(g_1babg_2)}\Oa^{(bg_2)}+\Oa^{(bg_2)}\Ob^{(g_1babg_2)}\Ob^{(g_2)}-\Oa^{(g_2)}\Oa^{(\sigma(g_1)g_2)}\Oa^{(g_2)},\\[1em]
&\R_{10}^{(g_1,g_2)}:=\Oa^{(g_2)}\Oa^{(g_1a^{-1}b^{-1}a^{-1}g_2)}\Ob^{(a^{-1}g_2)}+\Ob^{(a^{-1}g_2)}\Oa^{(g_1a^{-1}b^{-1}a^{-1}g_2)}\Oa^{(g_2)}-\Ob^{(g_2)}\Ob^{(\sigma(g_1)g_2)}\Ob^{(g_2)},\\
&\R_{11}^{(g_1,g_2)}:=\Ob^{(g_2)}\Ob^{(g_1b^{-1}a^{-1}b^{-1}g_2)}\Oa^{(b^{-1}g_2)}+\Oa^{(b^{-1}g_2)}\Ob^{(g_1b^{-1}a^{-1}b^{-1}g_2)}\Ob^{(g_2)}-\Oa^{(g_2)}\Oa^{(\sigma(g_1)g_2)}\Oa^{(g_2)}.
\end{align}
\label{eq:DefiningRelationsK2Full}
\end{subequations}

\begin{definition}
Let $\mathcal I\subset\mathcal F$ be an ideal generated by all relations (\ref{eq:DefiningRelationsK2Full}), consider the quotient algebra
\begin{equation*}
\mathcal A:=\mathcal F/\mathcal I.
\end{equation*}
\end{definition}
This algebra is the main object of the current section. We start by observing that relation (\ref{eq:Relation1Xg}) along with relations (\ref{eq:Relation6g}) and (\ref{eq:Relation7g}) imply that for all $g\in Free_2$, the generators $\Oa^{(g)},\Ob^{(g)}\in\mathcal A$, up to a scalar factor, are nothing but the square roots of one of the two idempotents:
\begin{align}
e_A:=-\frac1{\big(Q^2-Q^{-2}\big)^2}\Oa^{(1)}\Oa^{(1)},\qquad e_B:=-\frac1{\big(Q^2-Q^{-2}\big)^2}\Ob^{(1)}\Ob^{(1)},\qquad e_A^2=e_A,\qquad e_B^2=e_B\quad\textrm{in}\;\mathcal A.
\label{eq:IdempotentsDef}
\end{align}

Relations (\ref{eq:Relation6g}) and (\ref{eq:Relation7g}), also known as $q$-Serre relations, allow one to simplify expressions involving idempotents.
\begin{lemma}
For all $n\in\mathbb Z_{\geq0}$ and elements $g,h_1,\dots,h_n\in Free_2$ we have the following identities in $\mathcal A$
\begin{subequations}
\begin{align}
\Oa^{(1)}\Oa^{(1)}\Ob^{(h_1)}\dots\Ob^{(h_n)}\Oa^{(g)}=&\left\{\begin{array}{cl}
-(Q^2-Q^{-2})^2\,\Ob^{(h_1)}\dots\Ob^{(h_n)}\Oa^{(g)},&n\;\textrm{is even},\\
0,&n\;\textrm{is odd},
\end{array}\right.
\label{eq:IdempotentALeftCancellation}\\[0.25em]
\Oa^{(g)}\Ob^{(h_1)}\dots\Ob^{(h_n)}\Oa^{(1)}\Oa^{(1)}=&\left\{\begin{array}{cl}
-(Q^2-Q^{-2})^2\,\Oa^{(g)}\Ob^{(h_1)}\dots\Ob^{(h_n)},&n\;\textrm{is even},\\
0,&n\;\textrm{is odd}.
\end{array}\right.
\label{eq:IdempotentARightCancellation}
\end{align}
\end{subequations}
\end{lemma}
\begin{proof}
We will use induction in $n$. The base cases for $n=0,1$ are given by relations (\ref{eq:Relation4Xg}), (\ref{eq:Relation5Xg}) and (\ref{eq:Relation3g1g2g3}) respectively. The step of induction then follows by relation (\ref{eq:Relation6g}).
\end{proof}

\subsection{Action of the automorphisms}

\begin{lemma}
The following defines an order two automorphism of the algebra $\mathcal A$:
\begin{equation}
\mathbf s:\mathcal A\rightarrow\mathcal A,\qquad \mathbf s:\left\{\begin{array}{l}
\Oa^{(g)}\mapsto \Ob^{(\sigma(g))},\\[0.25em]
\Ob^{(g)}\mapsto \Oa^{(\sigma(g))}.
\end{array}\right.
\label{eq:sActionK2Algebra}
\end{equation}
\end{lemma}
\begin{proof}
First, consider an order two automorphism of the free algebra $\mathbf s:\mathcal F\rightarrow\mathcal F$ defined by the same action on generators (\ref{eq:sActionK2Algebra}) and denoted by the same letter. We will show that ideal $\mathcal I\subset\mathcal F$ is invariant under the action of $\mathbf s$:
\begin{equation*}
\mathbf s(\mathcal I)\subset\mathcal I.
\end{equation*}
Indeed, for the action of $\mathbf s$ on the generators of ideal $\mathcal I$ we get
\begin{align}
\mathbf s:\left\{\begin{array}{ccc}
\R_{0,A}^{(g)}&\leftrightarrow& \R_{0,B}^{(\sigma(g))}\\
\R_{1,A}^{(g)}&\leftrightarrow& \R_{1,B}^{(\sigma(g))}\\
\R_2^{(g_1,g_2,g_3)}&\leftrightarrow& \R_3^{(\sigma(g_1),\sigma(g_2),\sigma(g_3))}\\
\R_{4,A}^{(g)}&\leftrightarrow& \R_{4,B}^{\sigma(g)}\\
\R_{5,A}^{(g)}&\leftrightarrow& \R_{5,B}^{\sigma(g)}\\
\R_6^{(g)}&\leftrightarrow& \R_7^{(\sigma(g))}\\
\R_8^{(g_1,g_2)}&\leftrightarrow& \R_9^{(\sigma(g_1),\sigma(g_2))}\\
\R_{10}^{(g_1,g_2)}&\leftrightarrow& \R_{11}^{(\sigma(g_1),\sigma(g_2))}
\end{array}\right.
\label{eq:Z2SymmetryDefiningIdeal}
\end{align}
Hence the action of the automorphism $\mathbf s:\mathcal F\rightarrow\mathcal F$ descends to the automorphism of the quotient algebra $\mathcal A=\mathcal F/\mathcal I$.
\end{proof}

\begin{proposition}
The following defines an automorphism $\mathbf a:\mathcal A\rightarrow\mathcal A$ of the algebra $\mathcal A=\mathcal F/\mathcal I$:
\begin{subequations}
\begin{equation}
\mathbf a\big(\Oa^{(g)}\big)= -\frac{\Oa^{(1)}\Oa^{(ga^{-1})}\Oa^{(1)}}{\big(Q^2-Q^{-2}\big)^2},\qquad
\mathbf a\big(\Ob^{(g)}\big)= \frac{Q\Oa^{(1)}\Ob^{(ga^{-1})}-Q^{-1}\Ob^{(ga^{-1})}\Oa^{(1)}}{Q^2-Q^{-2}}.
\label{eq:aActionK2Algebra}
\end{equation}
The action of its inverse $a^{-1}:\mathcal A\rightarrow\mathcal A$ on generators is given by
\begin{equation}
\mathbf a^{-1}\big(\Oa^{(g)}\big)=-\frac{\Oa^{(1)}\Oa^{(ga)}\Oa^{(1)}}{\big(Q^2-Q^{-2}\big)^2},\qquad
\mathbf a^{-1}\big(\Ob^{(g)}\big)=\frac{-Q^{-1}\Oa^{(1)}\Ob^{(ga)}+Q\Ob^{(ga)}\Oa^{(1)}}{Q^2-Q^{-2}}.
\label{eq:aInverseActionK2Algebra}
\end{equation}
\label{eq:aaInverseActionK2Algebra}
\end{subequations}
\label{prop:ATwistAlgebrAutomorphism}
\end{proposition}
\begin{proof}
Consider a homomorphism $\mathbf a:\mathcal F\rightarrow\mathcal F$ of a free algebra defined by the same action on generators as in (\ref{eq:aActionK2Algebra}) and denoted by the same letter. First observe that all parts of (\ref{eq:aaInverseActionK2Algebra}) are invariant with respect to left translations $g\mapsto a^{\pm1}g$ and $g\mapsto b^{\pm1}g$, hence the ideal generated by relations $\R_{0,A}^{(g)},\R_{0,B}^{(g)},\;g\in Free_2$ is preserved by $\mathbf a$ and $\mathbf a^{-1}$. Throughout the rest of the text we simply identify the corresponding generators without mentioning:
\begin{equation*}
\Oa^{(a^kg)}=\Oa^{(g)},\qquad\Ob^{(b^kg)}=\Ob^{(g)},\qquad\textrm{for all}\quad k\in\mathbb Z,\quad g\in Free_2.
\end{equation*}

Using the above conventions we show that ideal $\mathcal I\subset\mathcal F$ is preserved by $\mathbf a$, hence the latter descends to a homomorphism $\mathbf a$ of the quotient algebra $\mathcal A=\mathcal F/\mathcal I$. For the sake of brevity we omit this tedious but straightforward calculations from the main text and present them in Appendix \ref{sec:IdealIInvariateAutomorphismA}.

Next, consider a homomorphism $\widetilde{\mathbf a}:\mathcal F\rightarrow\mathcal F$ defined by the same action on generators as (\ref{eq:aInverseActionK2Algebra}). Replacing $Q\mapsto -Q^{-1},a\mapsto a^{-1},b\mapsto b^{-1}, \R_8\leftrightarrow \R_{10}, \R_9\leftrightarrow\R_{11}$ in all formulas of Appendix \ref{sec:IdealIInvariateAutomorphismA} we show that $\widetilde{\mathbf a}$ preserves ideal $\mathcal I$ and thus descends to a homomorphism of the quotient algebra $\widetilde{\mathbf a}:\mathcal F/\mathcal I\rightarrow\mathcal F/\mathcal I$.

It remains to show that $\mathbf a$ and $\widetilde{\mathbf a}$ are mutually inverse modulo $\mathcal I$. Indeed, on generators of $\mathcal A$ we get:
\begin{align*}
\widetilde{\mathbf a}\big(\mathbf a(\Oa^{(g)})\big)=&\mathbf a\big(\widetilde{\mathbf a}\big(\Oa^{(g)}\big)\big)\\
=&\frac1{\big(Q^2-Q^{-2}\big)^8}\Oa^{(1)}\Oa^{(1)}\Oa^{(1)}\Oa^{(1)}\Oa^{(g)}\Oa^{(1)}\Oa^{(1)}\Oa^{(1)}\Oa^{(1)}\\
=&\Oa^{(g)}-\frac1{\big(Q^2-Q^{-2}\big)^2}\R_{4,A}^{(g)}-\frac1{\big(Q^2-Q^{-2}\big)^4}\Oa^{(1)}\Oa^{(1)}\R_{4,A}^{(g)}-\frac1{\big(Q^2-Q^{-2}\big)^6}\Oa^{(1)}\Oa^{(1)}\Oa^{(1)}\Oa^{(1)}\R_{5,A}^{(g)}\\
&-\frac1{\big(Q^2-Q^{-2}\big)^8}\Oa^{(1)}\Oa^{(1)}\Oa^{(1)}\Oa^{(1)}\Oa^{(g)}\Oa^{(1)}\R_{4,A}^{(1)}\\
\equiv&\Oa^{(g)}\;\bmod\;\mathcal I,\\[1.25em]
\widetilde{\mathbf a}\big(\mathbf a\big(\Ob^{(g)}\big)\big)=&\frac1{\big(Q^2-Q^{-2}\big)^4}\Big(\Oa^{(1)}\Oa^{(1)}\Oa^{(1)}\Oa^{(1)}\Ob^{(g)}-Q^2\Oa^{(1)}\Oa^{(1)}\Oa^{(1)}\Ob^{(g)}\Oa^{(1)}\\
&\qquad-Q^{-2}\Oa^{(1)}\Ob^{(g)}\Oa^{(1)}\Oa^{(1)}\Oa^{(1)}+\Ob^{(g)}\Oa^{(1)}\Oa^{(1)}\Oa^{(1)}\Oa^{(1)}\Big)\\
=&\Ob^{(g)}+\frac1{\big(Q^2-Q^{-2}\big)^2}\Big(-\R_{6}^{(g)}+Q^{-2}\R_{3}^{(1,g,1)}\Big)+\frac1{\big(Q^2-Q^{-2}\big)^4}\Big(\Oa^{(1)}\Oa^{(1)}\R_{6}^{(g)}-Q^2\Oa^{(1)}\Oa^{(1)}\R_{3}^{(1,g,1)}\\
&\qquad-Q^{-2}\Oa^{(1)}\Ob^{(g)}\R_{4,A}^{(1)}-\Oa^{(1)}\R_{3}^{(1,g,1)}\Oa^{(1)}+\Ob^{(g)}\Oa^{(1)}\R_{4,A}^{(1)}\Big)\\
\equiv&\Ob^{(g)}\;\bmod\;\mathcal I,\\[1.25em]
\mathbf a\big(\widetilde{\mathbf a}\big(\mathcal O_B^{(g)}\big)\big)=&\frac1{\big(Q^2-Q^{-2}\big)^4}\Big(\Oa^{(1)}\Oa^{(1)}\Oa^{(1)}\Oa^{(1)}\Ob^{(g)}-Q^{-2}\Oa^{(1)}\Oa^{(1)}\Oa^{(1)}\Ob^{(g)}\Oa^{(1)}\\
&\qquad-Q^2\Oa^{(1)}\Ob^{(g)}\Oa^{(1)}\Oa^{(1)}\Oa^{(1)}+\Ob^{(g)}\Oa^{(1)}\Oa^{(1)}\Oa^{(1)}\Oa^{(1)}\Big)\\
=&\Ob^{(g)}+\frac1{\big(Q^2-Q^{-2}\big)^2}\Big(-\R_{6}^{(g)}+Q^2\R_{3}^{(1,g,1)}\Big)+\frac1{\big(Q^2-Q^{-2}\big)^4}\Big(\Oa^{(1)}\Oa^{(1)}\R_{6}^{(g)}-Q^{-2}\Oa^{(1)}\Oa^{(1)}\R_{3}^{(1,g,1)}\\
&\qquad-Q^2\Oa^{(1)}\Ob^{(g)}\R_{4,A}^{(1)}-\Oa^{(1)}\R_{3}^{(1,g,1)}\Oa^{(1)}+\Ob^{(g)}\Oa^{(1)}\R_{4,A}^{(1)}\Big)\\
\equiv&\Ob^{(g)}\bmod\mathcal I.
\end{align*}
\end{proof}

\begin{theorem}
Automorphisms $\mathbf a,\mathbf s:\mathcal A\rightarrow\mathcal A$ satisfy defining relations of $PSL(2,\mathbb Z)$
\begin{equation*}
(\mathbf a\mathbf s)^3=\mathrm{Id}=\mathbf s^2
\end{equation*}
\label{th:AlgebraPSL2ZActionSA}
\end{theorem}
\begin{proof}
The second identity is satisfied by construction (\ref{eq:sActionK2Algebra}), so we will only have to prove the first one. With that in mind, the equivalent form of the remaining identity reads
\begin{equation}
(\mathbf a\mathbf s)^2=\mathbf s\mathbf a^{-1}
\label{eq:EquivalentFormSA3}
\end{equation}
Let $g\in Free_2$ be an arbitrary element of the free group. First, we apply both sides of (\ref{eq:EquivalentFormSA3}) to $\Oa^{(g)}$, this gives
\begin{equation}
\begin{aligned}
\mathbf a\big(\mathbf s\big(\mathbf a\big(\mathbf s\big(\Oa^{(g)}\big)\big)\big)\big)=&\mathbf a\big(\mathbf s\big(\mathbf a\big(\Ob^{(\sigma(g))}\big)\big)\big)\\
=&\frac1{Q^2-Q^{-2}}\mathbf a\Big(\mathbf s\Big(Q\Oa^{(1)}\Ob^{(\sigma(g)a^{-1})}-Q^{-1}\Ob^{(\sigma(g)a^{-1})}\Oa^{(1)}\Big)\Big)\\
=&\frac1{Q^2-Q^{-2}}\mathbf a\Big(Q\Ob^{(1)}\Oa^{(gb^{-1})}-Q^{-1}\Oa^{(gb^{-1})}\Ob^{(1)}\Big)\\
=&\frac1{\big(Q^2-Q^{-2}\big)^4}\Big(\Oa^{(1)}\Oa^{(gb^{-1}a^{-1})}\Oa^{(1)}\Oa^{(1)}\Ob^{(a^{-1})}-Q^{-2}\Oa^{(1)}\Oa^{(gb^{-1}a^{-1})}\Oa^{(1)}\Ob^{(a^{-1})}\Oa^{(1)}\\
&\qquad-Q^2\Oa^{(1)}\Ob^{(a^{-1})}\Oa^{(1)}\Oa^{(gb^{-1}a^{-1})}\Oa^{(1)}+\Ob^{(a^{-1})}\Oa^{(1)}\Oa^{(1)}\Oa^{(gb^{-1}a^{-1})}\Oa^{(1)}\Big)\\
=&\frac1{\big(Q^2-Q^{-2}\big)^2}\Big(-\Oa^{(1)}\Oa^{(gb^{-1}a^{-1})}\Ob^{(a^{-1})}-\Ob^{(a^{-1})}\Oa^{(gb^{-1}a^{-1})}\Oa^{(1)}\Big)\\
&+\frac1{\big(Q^2-Q^{-2}\big)^4}\Big(\Oa^{(1)}\Oa^{(gb^{-1}a^{-1})}\R_{6}^{(a^{-1})}-Q^{-2}\Oa^{(1)}\Oa^{(gb^{-1}a^{-1})}\R_{3}^{(1,a^{-1},1)}\\
&\qquad-\Oa^{(1)}\R_{3}^{(gb^{-1}a^{-1},a^{-1},1)}\Oa^{(1)}+\Ob^{(a^{-1})}\R_{4,A}^{(gb^{-1}a^{-1})}\Oa^{(1)}-Q^2\R_{3}^{(1,a^{-1},1)}\Oa^{(gb^{-1}a^{-1})}\Oa^{(1)}\Big)\\
\equiv&\frac1{\big(Q^2-Q^{-2}\big)^2}\Big(-\Oa^{(1)}\Oa^{(gb^{-1}a^{-1})}\Ob^{(a^{-1})}-\Ob^{(a^{-1})}\Oa^{(gb^{-1}a^{-1})}\Oa^{(1)}\Big)\quad\bmod\mathcal I,
\end{aligned}
\label{eq:asasOAg}
\end{equation}
and
\begin{equation}
\mathbf s\big(\mathbf a^{-1}\big(\Oa^{(g)}\big)\big)=-\frac1{\big(Q^2-Q^{-2}\big)^2}\mathbf s\Big(\Oa^{(1)}\Oa^{(ga)}\Oa^{(1)}\Big)=-\frac1{\big(Q^2-Q^{-2}\big)^2}\Ob^{(1)}\Ob^{(\sigma(g)b)}\Ob^{(1)}.
\label{eq:saInverseOAg}
\end{equation}
Combining (\ref{eq:asasOAg}) with (\ref{eq:saInverseOAg}) we get
\begin{equation}
\begin{aligned}
\mathbf a\big(\mathbf s\big(\mathbf a\big(\mathbf s\big(\Oa^{(g)}\big)\big)\big)\big)-\mathbf s\big(\mathbf a^{-1}\big(&\Oa^{(g)}\big)\big)\\
\equiv& \frac1{\big(Q^2-Q^{-2}\big)^2}\Big(-\Oa^{(1)}\Oa^{(gb^{-1}a^{-1})}\Ob^{(a^{-1})}+\Ob^{(1)}\Ob^{(\sigma(g)b)}\Ob^{(1)}-\Ob^{(a^{-1})}\Oa^{(gb^{-1}a^{-1})}\Oa^{(1)}\Big)\\
=&- \frac1{\big(Q^2-Q^{-2}\big)^2}\R_{10}^{(ga,1)}\equiv 0\quad\bmod\mathcal I.
\label{eq:R10ToProveMCG}
\end{aligned}
\end{equation}

Next, we apply both sides of (\ref{eq:EquivalentFormSA3}) to $\Ob^{(g)}$, this gives
\begin{equation}
\begin{aligned}
\mathbf a\big(\mathbf s\big(\mathbf a&\big(\mathbf s\big(\Ob^{(g)}\big)\big)\big)\big)\\
=&\mathbf a\big(\mathbf s\big(\mathbf a\big(\Oa^{(\sigma(g))}\big)\big)\big)\\
=&-\frac1{\big(Q^2-Q^{-2}\big)^2}\mathbf a\Big(\mathbf s\Big(\Oa^{(1)}\Oa^{(\sigma(g)a^{-1})}\Oa^{(1)}\Big)\Big)\\
=&-\frac1{\big(Q^2-Q^{-2}\big)^2}\mathbf a\Big(\Ob^{(1)}\Ob^{(gb^{-1})}\Ob^{(1)}\Big)\\
=&\frac1{\big(Q^2-Q^{-2}\big)^5}\Big(-Q^3\Oa^{(1)}\Ob^{(a^{-1})}\Oa^{(1)}\Ob^{(gb^{-1}a^{-1})}\Oa^{(1)}\Ob^{(a^{-1})}+Q\Oa^{(1)}\Ob^{(a^{-1})}\Oa^{(1)}\Ob^{(gb^{-1}a^{-1})}\Ob^{(a^{-1})}\Oa^{(1)}\\
&\qquad+Q\Oa^{(1)}\Ob^{(a^{-1})}\Ob^{(gb^{-1}a^{-1})}\Oa^{(1)}\Oa^{(1)}\Ob^{(a^{-1})}-Q^{-1}\Oa^{(1)}\Ob^{(a^{-1})}\Ob^{(gb^{-1}a^{-1})}\Oa^{(1)}\Ob^{(a^{-1})}\Oa^{(1)}\\
&\qquad+Q\Ob^{(a^{-1})}\Oa^{(1)}\Oa^{(1)}\Ob^{(gb^{-1}a^{-1})}\Oa^{(1)}\Ob^{(a^{-1})}-Q^{-1}\Ob^{(a^{-1})}\Oa^{(1)}\Oa^{(1)}\Ob^{(gb^{-1}a^{-1})}\Ob^{(a^{-1})}\Oa^{(1)}\\
&\qquad-Q^{-1}\Ob^{(a^{-1})}\Oa^{(1)}\Ob^{(gb^{-1}a^{-1})}\Oa^{(1)}\Oa^{(1)}\Ob^{(a^{-1})}+Q^{-3}\Ob^{(a^{-1})}\Oa^{(1)}\Ob^{(gb^{-1}a^{-1})}\Oa^{(1)}\Ob^{(a^{-1})}\Oa^{(1)}\Big),
\end{aligned}
\label{eq:asasOBg}
\end{equation}
and
\begin{equation}
\begin{aligned}
\mathbf s\Big(\mathbf a^{-1}\Big(\Ob^{(g)}\Big)\Big)=&\frac1{Q^2-Q^{-2}}\mathbf s\Big(Q\Ob^{(ga)}\Oa^{(1)}-Q^{-1}\Oa^{(1)}\Ob^{(ga)}\Big)\\
=&\frac1{Q^2-Q^{-2}}\Big( Q\Oa^{(\sigma(g)b)}\Ob^{(1)}-Q^{-1}\Ob^{(1)}\Oa^{(\sigma(g)b)}\Big).
\end{aligned}
\label{eq:saInverseOBg}
\end{equation}
To show that the right hand sides of (\ref{eq:asasOBg}) and (\ref{eq:saInverseOBg}) are equal to each other modulo $\mathcal I$ we compute the following difference
\begin{equation*}
\begin{aligned}
\mathbf a\big(\mathbf s\big(\mathbf a\big(\mathbf s\big(&\Ob^{(g)}\big)\big)\big)\big)-\mathbf s\big(\mathbf a^{-1}\big(\Ob^{(g)}\big)\big)+\frac1{\big(Q^2-Q^{-2}\big)^3}\Big(-Q\Oa^{(1)}\R_{8}^{(\sigma(g)a^{-1},a^{-1})}+Q^{-1}\R_{8}^{(\sigma(g)a^{-1},a^{-1})}\Oa^{(1)}\Big)\\
=&\frac1{\big(Q^2-Q^{-2}\big)^3}\Big(Q^{-1}\Oa^{(1)}\R_{3}^{(\sigma(g)b,1,1)}+Q^{-1}\Ob^{(1)}\R_{5,A}^{(\sigma(g)b)}-Q\R_{4,A}^{(\sigma(g)b)}\Ob^{(1)}\\
&\qquad+Q^{-1}\R_{2}^{(a^{-1},1,gb^{-1}a^{-1})}\Ob^{(a^{-1})}-Q\R_{3}^{(1,1,\sigma(g)b)}\Oa^{(1)}\Big)\\
&+\frac1{\big(Q^2-Q^{-2}\big)^5}\Big(-Q^3\Oa^{(1)}\Ob^{(a^{-1})}\Oa^{(1)}\R_{2}^{(gb^{-1}a^{-1},1,a^{-1})}+Q\Oa^{(1)}\Ob^{(a^{-1})}\Ob^{(gb^{-1}a^{-1})}\R_{6}^{(a^{-1})}\\
&\qquad-Q^{-1}\Oa^{(1)}\Ob^{(a^{-1})}\Ob^{(gb^{-1}a^{-1})}\R_{3}^{(1,a^{-1},1)}+Q\Oa^{(1)}\R_{2}^{(a^{-1},1,gb^{-1}a^{-1})}\Ob^{(a^{-1})}\Oa^{(1)}\\
&\qquad+Q\Ob^{(a^{-1})}\Oa^{(1)}\Oa^{(1)}\R_{2}^{(gb^{-1}a^{-1},1,a^{-1})}-Q^{-1}\Ob^{(a^{-1})}\Oa^{(1)}\Ob^{(gb^{-1}a^{-1})}\R_{6}^{(a^{-1})}\\
&\qquad+Q^{-3}\Ob^{(a^{-1})}\Oa^{(1)}\Ob^{(gb^{-1}a^{-1})}\R_{3}^{(1,a^{-1},1)}+Q^{-1}\Ob^{(a^{-1})}\Ob^{(gb^{-1}a^{-1})}\Oa^{(1)}\R_{3}^{(1,a^{-1},1)}\\
&\qquad-Q^{-1}\Ob^{(a^{-1})}\R_{6}^{(gb^{-1}a^{-1})}\Ob^{(a^{-1})}\Oa^{(1)}+Q^{-1}\R_{2}^{(a^{-1},1,gb^{-1}a^{-1})}\Ob^{(a^{-1})}\Oa^{(1)}\Oa^{(1)}\\
&\qquad-Q\Oa^{(1)}\Ob^{(a^{-1})}\Ob^{(gb^{-1}a^{-1})}\Ob^{(a^{-1})}\Oa^{(1)}\Oa^{(1)}\Big)\\
\equiv&-\frac1{\big(Q^2-Q^{-2}\big)^5}Q\Oa^{(1)}\Ob^{(a^{-1})}\Ob^{(gb^{-1}a^{-1})}\Ob^{(a^{-1})}\Oa^{(1)}\Oa^{(1)}\;\;\overset{(\ref{eq:IdempotentARightCancellation})}{\equiv}\;\;0\;\;\bmod\mathcal I.
\end{aligned}
\end{equation*}

We have thus shown that action of both automorphisms $\mathbf a\mathbf s\mathbf a\mathbf s:\mathcal A\rightarrow\mathcal A$ and $\mathbf s\mathbf a^{-1}:\mathcal A\rightarrow\mathcal A$ agree on generators $\Oa^{(g)},\Ob^{(g)}\in\mathcal A$ for all $g\in Free_2$. This implies (\ref{eq:EquivalentFormSA3}) and, as a corollary, the statement of the Theorem.
\end{proof}

Now we can introduce the following composition of the two automorphisms 
\begin{equation*}
\mathbf b:=\mathbf s\mathbf a\mathbf s:\mathcal A\rightarrow\mathcal A.
\end{equation*}
The action of $b$ and its inverse on generators is given by
\begin{subequations}
\begin{equation}
\mathbf b\big(\Oa^{(g)}\big)=\frac{Q\Ob^{(1)}\Oa^{(gb^{-1})}-Q^{-1}\Oa^{(gb^{-1})}\Ob^{(1)}}{Q^2-Q^{-2}},\qquad \mathbf b\big(\Ob^{(g)}\big)=-\frac{\Ob^{(1)}\Ob^{(gb^{-1})}\Ob^{(1)}}{\big(Q^2-Q^{-2}\big)^2},
\end{equation}
\begin{equation}
\mathbf b^{-1}\big(\Oa^{(g)}\big)=\frac{Q\Oa^{(gb)}\Ob^{(1)}-Q^{-1}\Ob^{(1)}\Oa^{(gb)}}{Q^2-Q^{-2}},\qquad\mathbf b^{-1}\big(\Ob^{(g)}\big)=-\frac{\Ob^{(1)}\Ob^{(gb)}\Ob^{(1)}}{\big(Q^2-Q^{-2}\big)^2}.
\end{equation}
\end{subequations}
Theorem \ref{th:AlgebraPSL2ZActionSA} implies that we have an alternative presentation of the $PSL(2,\mathbb Z)$-action by automorphisms of $\mathcal A$
\begin{corollary}
Authomorphisms $\mathbf a,\mathbf b:\mathcal A\rightarrow\mathcal A$ satisfy defining relations of equivalent presentation of $PSL(2,\mathbb Z)$
\begin{equation*}
\mathbf a\mathbf b\mathbf a=\mathbf b\mathbf a\mathbf b,\qquad (\mathbf a\mathbf b\mathbf a)^2=\mathrm{Id}.
\end{equation*}
\end{corollary}

\subsection{Casimir element}

Consider the following combination of idemponents
\begin{equation*}
C:=e_Be_A-e_A-e_B+1=\frac1{\big(Q^2-Q^{-2}\big)^4}\Ob^{(1)}\Ob^{(1)}\Oa^{(1)}\Oa^{(1)}+\frac1{\big(Q^2-Q^{-2}\big)^2}\Big(\Oa^{(1)}\Oa^{(1)}+\Ob^{(1)}\Ob^{(1)}\Big)+1\qquad\in\quad\mathcal A.
\end{equation*}
\begin{proposition}
Element $C$ has the following properties:
\begin{itemize}
    \item It is an idempotent
    \begin{equation*}
    C^2=C.
    \end{equation*}
    \item It is fixed by the action of $PSL(2,\mathbb Z)$ on $\mathcal A$, namely
    \begin{equation*}
        \mathbf a(C)=C=\mathbf s(C).
    \end{equation*}
    \item It is annihilated by multiplication with any other generator
    \begin{equation*}
        \Oa^{(g)}C=C\Oa^{(g)}=\Ob^{(g)}C=C\Ob^{(g)}=0\qquad\textrm{for all}\quad g\in Free_2.
    \end{equation*}
\end{itemize}
\end{proposition}
\begin{proof}
The first property is immediate from the definition through idempotents $e_A,e_B$. To get the second property we first note that
\begin{align*}
s(C)-C=&e_Ae_B-e_Be_A=\frac1{\big(Q^2-Q^{-2}\big)^4}\Big(\Oa^{(1)}\Oa^{(1)}\Ob^{(1)}\Ob^{(1)}-\Ob^{(1)}\Ob^{(1)}\Oa^{(1)}\Oa^{(1)}\Big)\\
=&\frac1{\big(Q^2-Q^{-2}\big)^4}\Big(\Oa^{(1)}\R_{7}^{(1)}-\R_{7}^{(1)}\Oa^{(1)}\Big)\equiv 0\qquad\bmod\quad\mathcal I.
\end{align*}
Next, we consider the action of automorphism $\mathbf a$ on individual idempotents
\begin{equation}
\begin{aligned}
\mathbf a(e_A)-e_A=&\frac1{\big(Q^2-Q^{-2}\big)^2}\Oa^{(1)}\Oa^{(1)}-\frac1{\big(Q^2-Q^{-2}\big)^6}\Oa^{(1)}\Oa^{(1)}\Oa^{(1)}\Oa^{(1)}\Oa^{(1)}\Oa^{(1)}\\
=&\frac1{\big(Q^2-Q^{-2}\big)^4}\Oa^{(1)}\R_{4,A}^{(1)}-\frac1{\big(Q^2-Q^{-2}\big)^6}\Oa^{(1)}\Oa^{(1)}\Oa^{(1)}\R_{4,A}^{(1)}\equiv0\qquad\bmod\quad\mathcal I,
\end{aligned}
\label{eq:aeAAction}
\end{equation}
\begin{equation}
\begin{aligned}
\mathbf a(e_B)+2e_Be_A-e_A-e_B=&\frac1{\big(Q^2-Q^{-2}\big)^2}\Big(\Oa^{(1)}\Oa^{(1)}+\Ob^{(1)}\Ob^{(1)}\Big)\\
&+\frac1{\big(Q^2-Q^{-2}\big)^4}\Big(-Q^2\Oa^{(1)}\Ob^{(a^{-1})}\Oa^{(1)}\Ob^{(a^{-1})}+\Oa^{(1)}\Ob^{(a^{-1})}\Ob^{(a^{-1})}\Oa^{(1)}\\
&\qquad+2\Ob^{(1)}\Ob^{(1)}\Oa^{(1)}\Oa^{(1)}+\Ob^{(a^{-1})}\Oa^{(1)}\Oa^{(1)}\Ob^{(a^{-1})}-Q^{-2}\Ob^{(a^{-1})}\Oa^{(1)}\Ob^{(a^{-1})}\Oa^{(1)}\Big)\\
=&-\frac1{\big(Q^2-Q^{-2}\big)^2}\R_{1,B}^{(a^{-1})}+\frac1{\big(Q^2-Q^{-2}\big)^4}\Big(-Q^2\Oa^{(1)}\R_{2}^{(a^{-1},1,a^{-1})}+\Ob^{(a^{-1})}\R_{6}^{(a^{-1})}\\
&\qquad-Q^{-2}\Ob^{(a^{-1})}\R_{3}^{(1,a^{-1},1)}+\R_{7}^{(1)}\Oa^{(1)}+\Oa^{(1)}\R_{1,B}^{(a^{-1})}\Oa^{(1)}-\R_{1,B}^{(a^{-1})}\Oa^{(1)}\Oa^{(1)}\Big)\\
\equiv&\;0\qquad\bmod\quad\mathcal I.
\end{aligned}
\label{eq:aeBAction}
\end{equation}
Combining (\ref{eq:aeAAction}) with (\ref{eq:aeBAction}) we get
\begin{equation*}
\mathbf a(C)=2e_Be_A^2-e_Be_A+e_A^2-2e_A-e_B+1=e_Be_A-e_A-e_B+1=C.
\end{equation*}

Finally, to prove the last statement fix an arbitrary element $g\in Free_2$ and compute
\begin{align*}
\Oa^{(g)}C=&\Oa^{(g)}+\frac1{\big(Q^2-Q^{-2}\big)^2}\Big(\Oa^{(g)}\Oa^{(1)}\Oa^{(1)}+\Oa^{(g)}\Ob^{(1)}\Ob^{(1)}\Big)+\frac1{\big(Q^2-Q^{-2}\big)^4}\Oa^{(g)}\Ob^{(1)}\Ob^{(1)}\Oa^{(1)}\Oa^{(1)}\\
=&\frac1{\big(Q^2-Q^{-2}\big)^2}\R_{7}^{(g)}+\frac1{\big(Q^2-Q^{-2}\big)^4}\Big(-\Ob^{(1)}\Ob^{(1)}\R_{5,A}^{(g)}+\R_{7}^{(g)}\Oa^{(1)}\Oa^{(1)}\Big)\equiv0\qquad\bmod\quad\mathcal I,\\[0.7em]
C\Oa^{(g)}=&\Oa^{(g)}+\frac1{\big(Q^2-Q^{-2}\big)^2}\Big(\Oa^{(1)}\Oa^{(1)}\Oa^{(g)}+\Ob^{(1)}\Ob^{(1)}\Oa^{(g)}\Big)+\frac1{\big(Q^2-Q^{-2}\big)^4}\Ob^{(1)}\Ob^{(1)}\Oa^{(1)}\Oa^{(1)}\Oa^{(g)}\\
=&\frac1{\big(Q^2-Q^{-2}\big)^2}\R_{4,A}^{(g)}+\frac1{\big(Q^2-Q^{-2}\big)^4}\Ob^{(1)}\Ob^{(1)}\R_{4,A}^{(g)}\equiv0\qquad\bmod\quad\mathcal I,\\[0.7em]
\Ob^{(g)}C=&\Ob^{(g)}+\frac1{\big(Q^2-Q^{-2}\big)^2}\Big(\Ob^{(g)}\Oa^{(1)}\Oa^{(1)}+\Ob^{(g)}\Ob^{(1)}\Ob^{(1)}\Big)+\frac1{\big(Q^2-Q^{-2}\big)^4}\Ob^{(g)}\Ob^{(1)}\Ob^{(1)}\Oa^{(1)}\Oa^{(1)}\\
=&\frac1{\big(Q^2-Q^{-2}\big)^2}\R_{5,B}^{(g)}+\frac1{\big(Q^2-Q^{-2}\big)^4}\R_{5,B}^{(g)}\Oa^{(1)}\Oa^{(1)}\equiv0\qquad\bmod\quad\mathcal I,\\[0.7em]
C\Ob^{(g)}=&\Ob^{(g)}+\frac1{\big(Q^2-Q^{-2}\big)^2}\Big(\Oa^{(1)}\Oa^{(1)}\Ob^{(g)}+\Ob^{(1)}\Ob^{(1)}\Ob^{(g)}\Big)+\frac1{\big(Q^2-Q^{-2}\big)^4}\Ob^{(1)}\Ob^{(1)}\Oa^{(1)}\Oa^{(1)}\Ob^{(g)}\\
=&\frac1{\big(Q^2-Q^{-2}\big)^2}\R_{6}^{(g)}+\frac1{\big(Q^2-Q^{-2}\big)^4}\Big(\Ob^{(1)}\Ob^{(1)}\R_{6}^{(g)}-\R_{4,B}^{(g)}\Oa^{(1)}\Oa^{(1)}\Big)\equiv0\qquad\bmod\quad\mathcal I.
\end{align*}
\end{proof}

\subsection{Relation to the undeformed algebra}

We conclude this section with an observation which will not be used elsewhere in the text. It is, however, necessary to clarify the relation between the algebra $\mathcal A$ introduced in this section and the $K=2$ finite dimensional module of an $A_1$ spherical Double Affine Hecke algebra.

Consider a quotient algebra
\begin{equation*}
\mathcal A_0:=\;\bigslant{\mathcal A}{\Big(\Oa^{(g)}-\Oa^{(1)},\Ob^{(g)}-\Ob^{(1)},C\Big)}
\end{equation*}
obtained by identifying generators for different elements $g\in Free_2$ and setting Casimir element to zero.

\begin{proposition}
The algebra $\mathcal A_0$ is isomorphic to a finite-dimensional representation of the $A_1$ spherical Double Affine Hecke algebra at the level $K=2$. Moreover, this isomorphism is equivariant with respect to the $PSL(2,\mathbb Z)$-action by automorphisms.
\end{proposition}
\begin{proof}
Denote the natural image of the generators of $\mathcal A$ in $\mathcal A_0$ by $\Oa,\Ob$. We get an equivalent presentation of $\mathcal A_0$ in terms of the above generators and relations
\begin{equation}
\begin{aligned}
\Oa\Ob\Oa=0&=\Ob\Oa\Ob,\\
\Oa^3+\big(Q^2-Q^{-2}\big)^2\Oa=0&=\Ob^3+\big(Q^2-Q^{-2}\big)^2\Ob,\\
\Oa^2\Ob+\Ob\Oa^2+\big(Q^2-Q^{-2}\big)^2\Ob=0&=\Ob^2\Oa+\Oa\Ob^2+\big(Q^2-Q^{-2}\big)^2\Oa,\\
\Ob^2\Oa^2+\big(Q^2-Q^{-2}\big)^2\big(\Oa^2+\Ob^2\big)+\big(Q^2-Q^{-2}\big)^4=0&.
\end{aligned}
\label{eq:A0DefiningRelations}
\end{equation}
This implies that $\mathcal A_0$ is finite-dimensional and spanned by the following nine elements
\begin{equation}
1,\quad \Oa,\quad \Ob,\quad \Oa^2,\quad \Oa\Ob,\quad \Ob\Oa,\quad \Ob^2,\quad \Oa^2\Ob,\quad \Oa\Ob^2.
\label{eq:A0SpanningSet}
\end{equation}

On the other hand, matrices of the action of $A_1$ spherical DAHA on $K=2$ finite dimensional module \cite{Cherednik'2005} satisfy relations (\ref{eq:A0DefiningRelations}). In other words, we have a homomorphism $\Psi_0:\mathcal A_0\rightarrow\mathrm{Mat}_{3\times3}(\mathbb C(Q))$ defined on generators as
\begin{equation*}
\Psi_0(\Oa)=\left(\begin{array}{ccc}
-\mathrm i \big(Q^{2}-Q^{-2}\big) & 0 & 0 \\
0 & 0 & 0 \\
0 & 0 & -\mathrm i \big(Q^2-Q^{-2}\big)
\end{array}\right),
\qquad
\Psi_0(\Ob)=\left(\begin{array}{ccc}
0 & 1 & 0 \\
-\frac12\big(Q^2-Q^{-2}\big)^2  & 0 & 1 \\
0 & -\frac12\big(Q^2-Q^{-2}\big)^2 & 0
\end{array}\right).
\end{equation*}

After verifying that images under $\Psi_0$ of elements (\ref{eq:A0SpanningSet}) are linearly independed we conclude that $\Psi_0$ is injective and hence $\mathcal A_0\simeq\Psi_0(\mathcal A_0)$.
\end{proof}

\section{Matrix Representation of the algebra $\mathcal A$}
\label{sec:MatrixRepresentation}

In this section we show that algebra $\mathcal A$ has a finite dimensional representation in $3\times3$ matrices over $\mathbf R$. To this end we recursively define a family of matrices $O_A^{(g)},O_B^{(g)}\in\mathrm{Mat}_{3\times 3}(\mathbf R)$ labelled by elements $g\in Free_2$ of the free group with two generators and show that homomorphism
\begin{equation*}
\Psi:\mathcal F\rightarrow\mathrm{Mat}_{3\times 3}(\mathbf R),\qquad \Psi(\mathcal O_A^{(g)})=O_A^{(g)},\qquad \Psi(\mathcal O_B^{(g)})=O_B^{(g)},\qquad \textrm{for all}\quad g\in Free_2
\end{equation*}
annihilates defining ideal $\mathcal I$ of the quotient algebra $\mathcal A=\mathcal F/\mathcal I$. Or, in other words, that matrices $O_A^{(g)},O_B^{(g)}$ satisfy all relations (\ref{eq:DefiningRelationsK2Full}). Moreover, we show that the $PSL(2,\mathbb Z)$-action by automorphisms of $\mathcal A$ on the level of its matrix representation coincides with conjugation by operators $\widehat D_A,\widehat D_B$ introduced earlier in Section \ref{sec:SL2ZActionOnMatR3}.

Our point of departure will be the following pair of elements $O_A^{(1)},O_B^{(1)}\in\mathrm{Mat}_{3\times3}(\mathbf R)$ given by explicit formulas
\begin{subequations}
\begin{align}
O_A^{(1)} := 
\left( 
\begin{array}{ccc} 
-\mathrm i \big(Q^2 - Q^{-2}\big) & 0 & 0 \\
0 & 0 & 0 \\
0 & 0 & i \big(Q^2 - Q^{-2}\big)
\end{array} 
\right)
\label{eq:Oa1}
\end{align}
\begin{gather}
\begin{aligned} 
O_B^{(1)}:=\left( 
\begingroup\arraycolsep=0pt
\begin{array}{ccc} 
0 & \mathrm{pexp}\left( \dfrac{(Q^8 - Q^{-8})p s (1-p)}{(1-s^2)(1-p^2)} \right) & 0 \\
-\dfrac{\big(Q^2-Q^{-2}\big)^2}2 \mathrm{pexp}\left( \dfrac{-(Q^8 - Q^{-8})p s (1-p)}{(1-s^2)(1-p^2)} \right)  & 0 & \mathrm{pexp}\left( \dfrac{-(Q^8 - Q^{-8})p s (1-p)}{(1-s^2)(1-p^2)} \right) \\
0 & -\dfrac{\big(Q^2-Q^{-2}\big)^2}2 \mathrm{pexp}\left( \dfrac{(Q^8 - Q^{-8})p s (1-p)}{(1-s^2)(1-p^2)} \right) & 0
\end{array}
\endgroup
\right).
\end{aligned}
\label{eq:Ob1}
\end{gather}
\label{eq:MatricesOa1Ob1}
\end{subequations}

\begin{lemma}
    Pair of matrices $O_A^{(1)},O_B^{(1)}$ satisfies particular case of relations (\ref{eq:Relation0Xg})--(\ref{eq:Relation7g}) for $g=g_1=g_2=g_3=1$, namely
\begin{subequations}
\begin{align}
O_A^{(1)}O_B^{(1)}O_A^{(1)}=\;&0=O_B^{(1)}O_A^{(1)}O_B^{(1)},
\label{eq:MatrixRelations23Forg1}\\
O_A^{(1)}O_A^{(1)}O_A^{(1)}+\big(Q^2-Q^{-2}\big)^2 O_A^{(1)}=\;&0=O_B^{(1)}O_B^{(1)}O_B^{(1)}+\big(Q^2-Q^{-2}\big)^2 O_B^{(1)},
\label{eq:MatrixRelations45Forg1}\\
O_A^{(1)}O_A^{(1)}O_B^{(1)}+\big(Q^2-Q^{-2}\big)^2O_B^{(1)}+O_B^{(1)}O_A^{(1)}O_A^{(1)}=\;&0=O_B^{(1)}O_B^{(1)}O_A^{(1)}+\big(Q^2-Q^{-2}\big)^2O_A^{(1)}+O_A^{(1)}O_B^{(1)}O_B^{(1)}
\label{eq:MatrixRelations67Forg1}
\end{align}
\label{eq:MatrixRelations234567Forg1}
\end{subequations}
\label{lemm:MatrixRelations234567Forg1}
\end{lemma}
\begin{proof}
All relations are verified by direct computation in $\mathbf R$ using exponential properties (\ref{eq:ExponentialPropertyFormalFractions}) of the $\mathrm{pexp}$.
\end{proof}

\begin{lemma} We have the following identities
\begin{equation}
\widehat S^{-1}O_A^{(1)}\widehat S=O_B^{(1)},\qquad \widehat S^{-1}O_B^{(1)}\widehat S=O_A^{(1)}.
\label{eq:SConjugationg1}
\end{equation}
\label{lemm:SConjugationg1}
\end{lemma}
\begin{proof}
The first identity is equivalent to the following equation on matrices
\begin{equation}
O_A^{(1)}S(1/s,p)=S(1/s,p)O_B^{(1)}(p,s),
\label{eq:Ob1ConsitencyMatrixCheck}
\end{equation}
where
\begin{equation*}
S(p,s):=D_AD_B(ps,s)D_A^{-1}\quad\in\quad\mathrm{Mat}_{3\times3}(\mathbf R),\qquad \widehat S=S(p,s)\circ\widehat\delta_{(p,s)\mapsto(s,p^{-1})}
\end{equation*}
is a matrix of an $\widehat S$ operator. Identity (\ref{eq:Ob1ConsitencyMatrixCheck}) is proved by direct computation in $\mathbf R$ using exponential properties (\ref{eq:ExponentialPropertyFormalFractions}).

To prove the second identity we conjugate the first one with $\widehat S$ and show that
\begin{equation}
O_A^{(1)}=\widehat S^{-2}O_A^{(1)}\widehat S^2.
\label{eq:SHat2ActsTriviallyOnOa1}
\end{equation}
To this end recall that $O_A^{(1)}$ is diagonal and independent of $p,s$, combining this fact with expression (\ref{eq:SHat2Explicit}) for $\widehat S^2$ we conclude that (\ref{eq:SHat2ActsTriviallyOnOa1}) holds.
\end{proof}

\subsection{Recursive definition of the image of generators}

Consider a homomorphism from the free group with two generators to $SL(2,\mathbb Z)$:
\begin{align*}
\varphi: \mathrm{Free}_2 \rightarrow SL(2,{\mathbb Z}), \ \ \ \ \ \varphi(a) = \left( \begin{array}{cc} 1 & -1 \\ 0 & 1 \end{array} \right), \ \ \ \ \ \varphi(b) = \left( \begin{array}{cc} 1 & 0 \\ 1 & 1 \end{array} \right).
\end{align*}
For a pair of elements $g,h\in Free_2$ we introduce two auxiliary parameters
\begin{subequations}
\begin{align}
c_+^{(h,g)} =& \mathrm{pexp}\left( \dfrac{(Q^8-Q^{-8}) }{(1 + s^{-\varphi(g)_{2,2}} p^{\varphi(g)_{2,1}})(1 + s^{-\varphi(hg)_{1,2}} p^{\varphi(hg)_{1,1}})(1 + s^{\varphi(hg)_{2,2}} p^{-\varphi(hg)_{2,1}})} \right)\qquad\in\quad\mathbf R,
\end{align}
\begin{align}
c_-^{(h,g)} =& \mathrm{pexp}\left(- \dfrac{(Q^8-Q^{-8}) }{(1 + s^{-\varphi(g)_{2,2}} p^{\varphi(g)_{2,1}})(1 + s^{\varphi(hg)_{1,2}} p^{-\varphi(hg)_{1,1}})(1 + s^{\varphi(hg)_{2,2}} p^{-\varphi(hg)_{2,1}})} \right)\qquad\in\quad\mathbf R.
\end{align}
\label{eq:cpluscminus}
\end{subequations}
Note that both constants behave well with respect to the right multiplication of $g$ by generators $a,b\in Free_2$:
\begin{equation}
c_{\pm}^{(h,ga)}(p,s)=c_{\pm}^{(h,g)}(p s,s),\qquad c_{\pm}^{(h,gb)}(p,s)=c_{\pm}^{(h,g)}(p,s/p)\qquad\textrm{for all}\quad h,g\in Free_2.
\label{eq:cpcmRightMultiplication}
\end{equation}

\begin{definition}
We define a family of matrices $O_A^{(g)}\in\mathrm{Mat}_{3\times 3}(\mathbf R)$ labelled by elements $g\in Free_2$ of the free group with two generators as a solution to the recursive relation
\begin{subequations}
\begin{align}
O_A^{(ba^k g)} =&\, c_+^{(ba^k,g)} O_A^{(g)} + \frac{c_+^{(ba^k,g)} - \left(c_+^{(ba^k,g)}\right)^{-1}}{\big(Q^2-Q^{-2}\big)^2} O_B^{(1)}O_B^{(1)} O_A^{(g)},
\label{eq:RecursiveRelationOab}\\
O_A^{(b^{-1} a^k g)} =&\, c_-^{(b^{-1}a^k,g)} O_A^{(g)} + \frac{c_-^{(b^{-1}a^k,g)} - \left(c_-^{(b^{-1}a^k,g)}\right)^{-1}}{\big(Q^2-Q^{-2}\big)^2} O_B^{(1)}O_B^{(1)} O_A^{(g)},
\label{eq:RecursiveRelationOabInverse}
\end{align}
\label{eq:RecursiveRelationOa}
\end{subequations}
satisfied for all $k\in\mathbb Z$ and reduced words $g$ of the special form
\begin{equation}
g=b^{\epsilon_1}a^{k_1}\dots a^{k_{n-1}}b^{\epsilon_n}a^{k_n},\qquad \epsilon_1,\dots,\epsilon_n\in\{\pm1\},\qquad k_1,\dots,k_n\in\mathbb Z,\qquad n\in\mathbb Z_{\geq0},
\label{eq:BReducedWordFree2}
\end{equation}
with initial condition (\ref{eq:Oa1}).

\label{def:Oag}
\end{definition}
Note that recursive relation (\ref{eq:RecursiveRelationOa}) does not necessarily hold for general words $g\in Free_2$, it must hold only for the ones which in reduced form start with $b^{\pm1}$.

\begin{lemma}
Recursive relation (\ref{eq:RecursiveRelationOa}) is self-consistent and admits a unique solution. Moreover, we have
\begin{equation}
O_A^{(a^kg)}=O_A^{(g)},\qquad\textrm{for all}\quad k\in\mathbb Z,\quad g\in Free_2.
\label{eq:MatrixOaLeftInvariance}
\end{equation}
\end{lemma}
\begin{proof}
Because in recursive relation (\ref{eq:RecursiveRelationOa}) we start with a reduced word $g$ of the form (\ref{eq:BReducedWordFree2}), the only two possibilities for the cancellation in $Free_2$ are described by the following identity
\begin{equation}
b^{-1}a^0(ba^k(g))=a^k(g)=ba^{0}(b^{-1}a^k(g)),
\label{eq:Free2CancellationForConsistencyProof}
\end{equation}
where without loss of generality we can assume that $g$ is a reduced word of the form (\ref{eq:BReducedWordFree2}). In order to prove self-consistency of the recursion we have to show that two-step application of recursion (\ref{eq:RecursiveRelationOa}) for the left hand side of (\ref{eq:Free2CancellationForConsistencyProof}) provides the same result as the one corresponding to the right hand side of (\ref{eq:Free2CancellationForConsistencyProof}).

Denote the matrix elements of $\varphi(g)$ by
\begin{equation*}
\varphi(g)=\left(\begin{array}{cc}
g_{11}&g_{12}\\
g_{21}&g_{22}
\end{array}\right),
\end{equation*}
the relevant constants (\ref{eq:cpluscminus}) then read
\begin{equation}
c_+^{(ba^k,g)}=\mathrm{pexp}\left(\frac{Q^8-Q^{-8}}{\left(1+p^{g_{2,1}} s^{-g_{2,2}}\right) \left(1+p^{-(1-k) g_{2,1}-g_{1,1}} s^{(1-k) g_{2,2}+g_{1,2}}\right) \left(1+p^{g_{1,1}-k g_{2,1}} s^{k g_{2,2}-g_{1,2}}\right)}\right)=\frac1{c_-^{(b^{-1},ba^kg)}}=:\lambda
\label{eq:cpRelevantFirstArgumentExplicitFormula}
\end{equation}
The left hand side of (\ref{eq:Free2CancellationForConsistencyProof}) corresponds to the following two-step application of the recursion
\begin{align*}
O_A^{(b^{-1}a^0(ba^{k}(g)))}\;\overset{(\ref{eq:RecursiveRelationOabInverse})}{=}&\;\lambda^{-1}O_A^{(ba^kg)}-\frac{\lambda-\lambda^{-1}}{\big(Q^2-Q^{-2}\big)^2}O_B^{(1)}O_B^{(1)}O_A^{(ba^kg)}\\
\overset{(\ref{eq:RecursiveRelationOab})}{=}&\;O_A^{(g)}-\frac{\left(\lambda-\lambda^{-1}\right)^2}{\big(Q^2-Q^{-2}\big)^2}O_B^{(1)}O_B^{(1)}O_A^{(g)}-\frac{\left(\lambda-\lambda^{-1}\right)^2}{\big(Q^2-Q^{-2}\big)^4}O_B^{(1)}O_B^{(1)}O_B^{(1)}O_B^{(1)}O_B^{(1)}O_A^{(g)}\\
\overset{(\ref{eq:MatrixRelations45Forg1})}{=}&\;O_A^{(g)}.
\end{align*}
We omit similar calculation corresponding to the right hand side of (\ref{eq:Free2CancellationForConsistencyProof}) which provides identical result.

Note that as a byproduct of the above calculations we have shown that (\ref{eq:MatrixOaLeftInvariance}) holds for special case, when $g$ is represented by a reduced word of the form (\ref{eq:BReducedWordFree2}). Because the latter assumption can be made in (\ref{eq:MatrixOaLeftInvariance}) without loss of generality this concludes the proof of the lemma.
\end{proof}

\begin{definition}
    We define a second family of matrices labelled by elements $g\in Free_2$:
    \begin{equation}
    O_B^{(g)}:=\widehat S^{-1}O_A^{(\sigma(g))}\widehat S
    \label{eq:MatrixOBgDef}
    \end{equation}
\label{def:Obg}
\end{definition}
Note that in particular case of $g=1$, the above definition is consistent with an explicit formula (\ref{eq:Ob1}) for $O_B^{(1)}$ given earlier; this was shown in Lemma \ref{lemm:SConjugationg1}. Also, as an immediate corollary of Definition \ref{def:Obg} we get the dual version of property (\ref{eq:MatrixOaLeftInvariance}):
\begin{equation}
O_B^{(b^kg)}=O_B^{(g)},\qquad\textrm{for all}\quad k\in\mathbb Z,\qquad g\in Free_2.
\label{eq:MatrixObLeftInvariance}
\end{equation}

\begin{lemma}
Family of matrices $O_B^{(g)},\;g\in Free_2$ satisfies the dual version of recursive relations (\ref{eq:RecursiveRelationOa}), namely
\begin{subequations}
\begin{align}
O_B^{(ab^k,g)}=&\,\widetilde c_+^{(ab^k,g)}O_B^{(g)}+\frac{\widetilde c_+^{(ab^k,g)}-\left(\widetilde c_+^{(ab^k,g)}\right)^{-1}}{\big(Q^2-Q^{-2}\big)^2}O_A^{(1)}O_A^{(1)}O_B^{(g)},\\
O_B^{(a^{-1}b^k,g)}=&\,\widetilde c_-^{(a^{-1}b^k,g)}O_B^{(g)}+\frac{\widetilde c_-^{a^{-1}b^k,g}-\left(\widetilde c_-^{a^{-1}b^k,g}\right)^{-1}}{\big(Q^2-Q^{-2}\big)^2}O_A^{(1)}O_A^{(1)}O_B^{(g)}.
\end{align}
\label{eq:RecursiveRelationOb}
\end{subequations}
for reduced words $g$ of the dual special type
\begin{equation}
g=a^{\epsilon_1}b^{k_1}\dots a^{\epsilon_n}b^{k_n},\qquad \epsilon_1,\dots,\epsilon_n\in\{\pm1\},\qquad k_1,\dots,k_n\in\mathbb Z,\qquad n\in\mathbb Z_{\geq0},
\label{eq:AReducedWordFree2}
\end{equation}
where the dual constants are given by
\begin{equation}
\widetilde c_{\pm}^{(h,g)}:=\delta_{(p,s)\mapsto(s,p^{-1})}\left(c_{\pm}^{(\sigma(h),\sigma(g))}\right)\qquad\in\quad\mathbf R,\qquad h,g\in Free_2.
\label{eq:ctildeDef}
\end{equation}
\end{lemma}
\begin{proof}
Conjugate both sides of (\ref{eq:RecursiveRelationOa}) with $\widehat S$ and use Lemma \ref{lemm:SConjugationg1}.
\end{proof}

\subsection{Equivariance of $\Psi$ and annihilation of elementary relations.}

\begin{lemma}
\begin{subequations}
\begin{align}
\Psi(\mathbf a(\Oa^{(1)}))&=\widehat D_A^{-1}O_A^{(1)}\widehat D_A,\qquad \Psi(\mathbf s(\Oa^{(1)}))=\widehat S^{-1} O_A^{(1)}\widehat S,\\
\Psi(\mathbf a(\Ob^{(1)}))&=\widehat D_A^{-1}O_B^{(1)}\widehat D_A,\qquad \Psi(\mathbf s(\Ob^{(1)}))=\widehat S^{-1}O_B^{(1)}\widehat S.
\end{align}
\label{eq:HomomorphismPropertyg1}
\end{subequations}
\label{lemm:MCGEquivarianceBaseOfInduction}
\end{lemma}
\begin{proof}
The top left identity follows from $\mathbf a(\Oa^{(1)})\overset{(\ref{eq:Relation0Xg}),(\ref{eq:Relation4Xg})}{=}\Oa^{(1)}$ and the fact that matrices $O_A^{(1)},D_A\in\mathrm{Mat}_{3\times3}(\mathbf R)$ are both diagonal and independent of $p,s$. The two right identities follow by (\ref{eq:SConjugationg1}).

Finally, the bottom left identity is equivalent to the following equation on matrices
\begin{equation}
\frac{Q O_A^{(1)}O_B^{(a^{-1})}-Q^{-1}O_B^{(a^{-1})}\Oa^{(1)}}{Q^2-Q^{-2}}=D_A^{-1}O_B^{(1)}(p/s,s)D_A
\label{eq:PSiaOb1EqualsConjugation}
\end{equation}
To compute $O_B^{(a^{-1})}$, recall that particular case of recursive relation (\ref{eq:RecursiveRelationOab}) for $k=0$ and $g=1$ reads
\begin{equation}
O_A^{(b^{-1})}=\rho O_A^{(1)}+\frac{\rho-\rho^{-1}}{\big(Q^2-Q^{-2}\big)^2} O_B^{(1)}O_B^{(1)}O_A^{(1)},\qquad\textrm{where}\quad \rho=\mathrm{pexp}\left(-\frac{Q^8-Q^{-8}}{(1+p^{-1})(1+s^{-1})(1+ps)}\right).
\label{eq:MatrixOabExplicit}
\end{equation}
Conjugating both sides of (\ref{eq:MatrixOabExplicit}) with $\widehat S$ and using (\ref{eq:SHat2ActsTriviallyOnOa1}) we get
\begin{equation}
O_B^{(a^{-1})}=\widetilde\rho O_B^{(1)}+\frac{\widetilde\rho-\widetilde\rho^{-1}}{\big(Q^2-Q^{-2}\big)^2}O_A^{(1)}O_A^{(1)}O_B^{(1)},\qquad\textrm{where}\quad\widetilde\rho=\mathrm{pexp}\left(-\frac{Q^8-Q^{-8}}{(1+p^{-1})(1+s)(1+ps^{-1})}\right).
\label{eq:MatrixObaExplicit}
\end{equation}
Finally, substituting (\ref{eq:MatrixObaExplicit}) into (\ref{eq:PSiaOb1EqualsConjugation}) and simplifying the product of matrices over $\mathbf R$ using exponential properties (\ref{eq:ExponentialPropertyFormalFractions}) we conclude the proof.
\end{proof}

\begin{proposition}
For all elements $g\in Free_2$
\begin{subequations}
\begin{align}
\Psi(\mathbf a(\Oa^{(g)}))&=\widehat D_A^{-1}O_A^{(g)}\widehat D_A,\qquad \Psi(\mathbf s(\Oa^{(g)}))=\widehat S^{-1} O_A^{(g)}\widehat S,\\
\Psi(\mathbf a(\Ob^{(g)}))&=\widehat D_A^{-1}O_B^{(g)}\widehat D_A,\qquad \Psi(\mathbf s(\Ob^{(g)}))=\widehat S^{-1}O_B^{(g)}\widehat S.
\end{align}
\label{eq:PsiEquivariance}
\end{subequations}
\label{prop:MCGEquivariancePsi}
\end{proposition}
\begin{proof}
The right two identities can equivalently be written as
\begin{equation}
O_B^{(\sigma(g))}=\widehat S^{-1}O_A^{(g)}\widehat S,\qquad O_A^{(\sigma(g))}=\widehat S^{-1}O_B^{(g)}\widehat S\qquad\textrm{for all}\quad g\in Free_2.
\label{eq:SHatConjugationFlipsAB}
\end{equation}
The first one is precisely Definition \ref{def:Obg}, while the second one is equivalent to the following
\begin{equation}
O_A^{(g)}=\widehat S^{-2}O_A^{(g)}\widehat S^2\qquad\textrm{for all}\quad g\in Free_2.
\label{eq:SHat2ActsTriviallyOnOag}
\end{equation}
We will prove the latter by induction in $g$. The base case $g=1$ is given by (\ref{eq:SHat2ActsTriviallyOnOa1}). To show the step of induction, assume that (\ref{eq:SHat2ActsTriviallyOnOag}) is satisfied for all reduced words of length strictly less than $n$ and let $g$ be a reduced word of length $n$. In light of (\ref{eq:MatrixOaLeftInvariance}), we can assume without loss of generality that $g=b^{\pm1}a^kf,$ where $f$ is a reduced word of a special type (\ref{eq:BReducedWordFree2}). We will utilize recursive relation (\ref{eq:RecursiveRelationOa}) to reduce the identity (\ref{eq:SHat2ActsTriviallyOnOag}) for $g$ to the identity for $f$. To this end, first note that formula (\ref{eq:cpRelevantFirstArgumentExplicitFormula}) implies the following symmetry of relevant auxiliary constants
\begin{equation}
c_{\pm}^{(b^{\pm1}a^k,f)}(p,s)=c_{\pm}^{(b^{\pm1}a^k,f)}(1/p,1/s)\qquad\textrm{for all}\quad f\in Free_2.
\label{eq:cpcmZ2Symmetry}
\end{equation}
We then get
\begin{align*}
O_A^{(b^{\pm1}a^kf)}\quad\overset{(\ref{eq:RecursiveRelationOa})}{=}\quad\;\;\,&c_{\pm}^{(b^{\pm1}a^k,f)}(p,s)O_A^{(f)}+\frac{c_{\pm}^{(b^{\pm1}a^kf)}(p,s)+\left(c_{\pm}^{(b^{\pm1}a^kf)}(p,s)\right)^{-1}}{\big(Q^2-Q^{-2}\big)^2}O_B^{(1)}O_B^{(1)}O_A^{(f)}\\
\overset{(\ref{eq:SHat2ActsTriviallyOnOag})}{\underset{\textrm{for}\;g=f}{=}}\quad&c_{\pm}^{(b^{\pm1}a^k,f)}(p,s)\widehat S^{-2}O_A^{(f)}\widehat S^2+\frac{c_{\pm}^{(b^{\pm1}a^k,f)}(p,s)+\left(c_{\pm}^{(b^{\pm1}a^k,f)}(p,s)\right)^{-1}}{\big(Q^2-Q^{-2}\big)^2}O_B^{(1)}O_B^{(1)}\widehat S^{-2}O_A^{(f)}\widehat S^2\\
\overset{(\ref{eq:HomomorphismPropertyg1}),(\ref{eq:cpcmZ2Symmetry})}{\underset{(\ref{eq:RecursiveRelationOa})}{=}}\;\;\,&\widehat S^{-2}O_A^{(b^{\pm1}s^kf)}\widehat S^2
\end{align*}

Now we have to prove the two left identities of (\ref{eq:PsiEquivariance}). First, we apply automorphism $\mathbf a$ followed by a homomorphism $\Psi$ to obtain an equivalent form of the above identities. In terms of matrices in $\mathrm{Mat_{3\times3}}(\mathbf R)$ we get
\begin{subequations}
\begin{align}
-\frac{O_A^{(1)}O_A^{(ga^{-1})}O_A^{(1)}}{\big(Q^2-Q^{-2}\big)^2}&=D_A^{-1}O_A^{(g)}(p/s,s)D_A,
\label{eq:MatrixIdentitiesMCGEquivarianceOa}\\
\frac{QO_A^{(1)}O_B^{(ga^{-1})}-Q^{-1}O_B^{(ga^{-1})}O_A^{(1)}}{Q^2-Q^{-2}}&=D_A^{-1}O_B^{(g)}(p/s,s)D_A.
\label{eq:MatrixIdentitiesMCGEquivarianceOb}
\end{align}
\label{eq:MatrixIdentitiesMCGEquivariance}
\end{subequations}
Note that by (\ref{eq:MatrixOaLeftInvariance}) it is enough to prove (\ref{eq:MatrixIdentitiesMCGEquivarianceOa}) only for reduced words $g$ of the special type (\ref{eq:BReducedWordFree2}). Similarly, by (\ref{eq:MatrixObLeftInvariance}), it is enough to prove (\ref{eq:MatrixIdentitiesMCGEquivarianceOb}) only for reduced words $g$ of the special type (\ref{eq:AReducedWordFree2}). This will allow us to prove each of the identities (\ref{eq:MatrixIdentitiesMCGEquivariance}) by induction in $g$ utilizing recursive relations (\ref{eq:RecursiveRelationOa}) and (\ref{eq:RecursiveRelationOb}) respectively.

For (\ref{eq:MatrixIdentitiesMCGEquivarianceOa}) our base of induction will be the statement that (\ref{eq:MatrixIdentitiesMCGEquivarianceOa}) holds for $g=1$ as well as for all $g=b^{\pm1}a^k,\;k\in\mathbb Z$. Lemma \ref{lemm:MCGEquivarianceBaseOfInduction} already covers the case $g=1$. We can reduce the case of $g=b^{\pm1}a^k$ to $g=1$ by observing relation on auxiliary constants
\begin{subequations}
\begin{align}
c_+^{(ba^{k-1},1)}(p,s)=&\,\mathrm{pexp}\left(\frac{Q^8-Q^{-8}}{(1+s^{-1})(1+p^{-1}s^{2-k})(1+ps^{k-1})}\right)=c_+^{(ba^k,1)}(p/s,s),\\
c_-^{(b^{-1}a^{k-1},1)}(p,s)=&\,\mathrm{pexp}\left(-\frac{Q^8-Q^{-8}}{(1+s^{-1})(1+p^{-1}s^{1-k})(1+ps^k)}\right)=c_-^{(b^{-1}a^k,1)}(p/s,s),
\end{align}
\label{eq:AuxiliaryConstantsRelationIdentity1Casef1}
\end{subequations}
which gives
\begin{align*}
-\frac{O_A^{(1)}O_A^{(b^{\pm}a^{k-1})}(p,s)O_A^{(1)}}{\big(Q^2-Q^{-2}\big)^2}\quad\overset{(\ref{eq:RecursiveRelationOa})}{=}\quad&-\frac{c_{\pm}^{(b^{\pm1}a^{k-1},1)}(p,s)}{\big(Q^2-Q^{-2}\big)^2}O_A^{(1)}O_A^{(1)}O_A^{(1)}-\frac{c_{\pm}^{(b^{\pm1}a^{k-1},1)}(p,s)-\left(c_{\pm}^{(b^{\pm1}a^{k-1},1)}(p,s)\right)^{-1}}{\big(Q^2-Q^{-2}\big)^4}\\[0.25em]
&\qquad\qquad\times O_A^{(1)}O_B^{(1)}(p,s)O_B^{(1)}(p,s)O_A^{(1)}O_A^{(1)}\\[0.25em]
\quad\overset{(\ref{eq:DehnTwistActionIdempotentObMatrices}),(\ref{eq:AuxiliaryConstantsRelationIdentity1Casef1})}{=}&-\frac{c_{\pm}^{(b^{\pm1}a^k,1)}(p/s,s)}{\big(Q^2-Q^{-2}\big)^2}O_A^{(1)}O_A^{(1)}O_A^{(1)}-\frac{c_{\pm}^{(b^{\pm1}a^k,1)}(p/s,s)-\left(c_{\pm}^{(b^{\pm1}a^k,1)}(p/s,s)\right)^{-1}}{\big(Q^2-Q^{-2}\big)^4}\\[0.25em]
&\qquad\qquad\times D_A^{-1}O_B^{(1)}(p/s,s)O_B^{(1)}(p/s,s)D_AO_A^{(1)}O_A^{(1)}O_A^{(1)}\\[0.25em]
\overset{(\ref{eq:MatrixIdentitiesMCGEquivarianceOa})}{\underset{\textrm{for}\,g=1}{=}}\;&D_A^{-1}\Big(c_{\pm}^{(b^{\pm1}a^k,1)}(p/s,s)O_A^{(1)}+\frac{c_{\pm}^{(b^{\pm1}a^k,1)}(p/s,s)-\left(c_{\pm}^{(b^{\pm1}a^k,1)}(p/s,s)\right)^{-1}}{\big(Q^2-Q^{-2}\big)^2}\\[0.25em]
&\qquad\qquad\times O_B^{(1)}(p/s,s)O_B^{(1)}(p/s,s)O_A^{(1)}\Big)D_A\\[0.25em]
\overset{(\ref{eq:RecursiveRelationOa})}{=}\quad&D_A^{-1}O_A^{(b^{\pm1}a^k)}(p/s,s)D_A.
\end{align*}

To show the step of induction for (\ref{eq:MatrixIdentitiesMCGEquivarianceOa}) we assume that (\ref{eq:MatrixIdentitiesMCGEquivarianceOa}) holds for all reduced words (\ref{eq:BReducedWordFree2}) with length strictly less than $N$. Now let $g$ be a reduced word of type (\ref{eq:BReducedWordFree2}) with of length $N$, hence $g=b^{\pm1}a^kf$, where $f$ is some reduced word of a special type (\ref{eq:BReducedWordFree2}). Note that by our base of induction, we can assume without loss of generality that $f$ is nontrivial reduced word of the special type (\ref{eq:BReducedWordFree2}) and so $fa^{-1}$ must also be of the special type (\ref{eq:BReducedWordFree2}).

The following identity in $Mat_{3\times3}(\mathbf R)$ can be verified by direct computation in $\mathbf R$ using exponential properties (\ref{eq:ExponentialPropertyFormalFractions})
\begin{equation}
D_A^{-1}O_B^{(1)}(p/s,s)O_B^{(1)}(p/s,s)D_AO_A^{(1)}=O_A^{(1)}O_B^{(1)}(p,s)O_B^{(1)}(p,s)
\label{eq:DehnTwistActionIdempotentObMatrices}
\end{equation} With (\ref{eq:DehnTwistActionIdempotentObMatrices}) we can now deduce (\ref{eq:MatrixIdentitiesMCGEquivarianceOa}) for $g$ from the same identitiy for $f$ by expanding both elements $O_A^{(ga)}=O_A^{(b^{\pm1}a^kfa^{-1})}$ and $O_A^{(g)}=O_A^{(b^{\pm1}a^kf)}$ via recursive relations (\ref{eq:RecursiveRelationOa}). We get
\begin{align*}
-\frac{O_A^{(1)}O_A^{(b^{\pm}a^kfa^{-1})}(p,s)O_A^{(1)}}{\big(Q^2-Q^{-2}\big)^2}&\quad\overset{(\ref{eq:RecursiveRelationOa})}{=}\quad-\frac{c_{\pm}^{(b^{\pm1}a^k,fa^{-1})}(p,s)}{\big(Q^2-Q^{-2}\big)^2}O_A^{(1)}O_A^{(fa^{-1})}(p,s)O_A^{(1)}\\
&-\frac{c_{\pm}^{(b^{\pm1}a^k,fa^{-1})}(p,s)-\left(c_{\pm}^{(b^{\pm1}a^k,fa^{-1})}(p,s)\right)^{-1}}{\big(Q^2-Q^{-2}\big)^4} O_A^{(1)}O_B^{(1)}(p,s)O_B^{(1)}(p,s)O_A^{(fa^{-1})}(p,s)O_A^{(1)}\\
\quad\overset{(\ref{eq:cpcmRightMultiplication}),(\ref{eq:DehnTwistActionIdempotentObMatrices})}{=}&-\frac{c_{\pm}^{(b^{\pm1}a^k,f)}(p/s,s)}{\big(Q^2-Q^{-2}\big)^2}O_A^{(1)}O_A^{(fa^{-1})}O_A^{(1)}-\frac{c_{\pm}^{(b^{\pm1}a^k,f)}(p/s,s)-\left(c_{\pm}^{(b^{\pm1}a^k,f)}(p/s,s)\right)^{-1}}{\big(Q^2-Q^{-2}\big)^4}\\[0.25em]
&\qquad\qquad\times D_A^{-1}O_B^{(1)}(p/s,s)O_B^{(1)}(p/s,s)D_AO_A^{(1)}O_A^{(fa^{-1})}O_A^{(1)}\\[0.25em]
\overset{(\ref{eq:MatrixIdentitiesMCGEquivarianceOa})}{\underset{\textrm{for}\,g=f}{=}}\;&D_A^{-1}\Big(c_{\pm}^{(b^{\pm1}a^k,f)}(p/s,s)O_A^{(f)}\\
&\qquad+\frac{c_{\pm}^{(b^{\pm1}a^k,f)}(p/s,s)-\left(c_{\pm}^{(b^{\pm1}a^k,f)}(p/s,s)\right)^{-1}}{\big(Q^2-Q^{-2}\big)^2}O_B^{(1)}(p/s,s)O_B^{(1)}(p/s,s)O_A^{(f)}\Big)D_A\\
\overset{(\ref{eq:RecursiveRelationOa})}{=}\quad&D_A^{-1}O_A^{(b^{\pm1}a^kf)}(p/s,s)D_A.
\end{align*}

To prove identity (\ref{eq:MatrixIdentitiesMCGEquivarianceOb}) our base of induction will be the case $g=1$ given by Lemma \ref{lemm:MCGEquivarianceBaseOfInduction}. To show the step of induction we assume that (\ref{eq:MatrixIdentitiesMCGEquivarianceOb}) holds for all reduced words of the dual special type (\ref{eq:AReducedWordFree2}) of length strictly less that $N$ and now let $g$ be a reduced word (\ref{eq:AReducedWordFree2}) of length $N$. Then $g=a^{\pm1}b^kf$, where $f$ is a reduced word of the dual special type (\ref{eq:AReducedWordFree2}).

Combining (\ref{eq:ctildeDef}) with (\ref{eq:cpcmRightMultiplication}) we obtain identity on relevant auxiliary constants which holds for all $h,g\in Free_2$:
\begin{equation}
\widetilde c_{\pm}^{(h,ga^{-1})}(p,s)\overset{(\ref{eq:ctildeDef})}{=}c_{\pm}^{(\sigma(h),\sigma(ga^{-1}))}(s,p^{-1})=c_{\pm}^{(\sigma(h),\sigma(g)b^{-1})}(s,p^{-1})\overset{(\ref{eq:cpcmRightMultiplication})}{=}c_{\pm}^{(\sigma(h),\sigma(g))}(s,s/p)\overset{(\ref{eq:ctildeDef})}{=}\widetilde c_{\pm}^{(h,g)}(p/s,s).
\label{eq:cpcmTildeRightMultiplicationa}
\end{equation}
We get
\begin{align*}
&\frac{QO_A^{(1)}O_B^{(a^{\pm1}b^kfa^{-1})}(p,s)-Q^{-1}O_B^{(a^{\pm1}b^kfa^{-1})}(p,s)O_A^{(1)}}{\big(Q^2-Q^{-2}\big)^2}\\
&\quad\overset{(\ref{eq:RecursiveRelationOb})}{=}\frac{Q}{\big(Q^2-Q^{-2}\big)^2}O_A^{(1)}\Big(\widetilde c_{\pm}^{(a^{\pm1}b^k,fa^{-1})}(p,s) O_B^{(fa^{-1})}(p,s)\\
&\qquad\qquad\quad+\frac{\widetilde c_{\pm}^{(a^{\pm1}b^k,fa^{-1})}(p,s)-\left(\widetilde c_{\pm}^{(a^{\pm1}b^k,fa^{-1})}(p,s)\right)^{-1}}{\big(Q^2-Q^{-2}\big)^2}O_A^{(1)}O_A^{(1)}O_B^{(fa^{-1})}(p,s)\Big)\\
&\qquad\quad-\frac{Q^{-1}}{\big(Q^2-Q^{-2}\big)^2}\Big(\widetilde c_{\pm}^{(a^{\pm1}b^k,fa^{-1})}(p,s) O_B^{(fa^{-1})}(p,s)\\
&\qquad\qquad\quad+\frac{\widetilde c_{\pm}^{(a^{\pm1}b^k,fa^{-1})}(p,s)-\left(\widetilde c_{\pm}^{(a^{\pm1}b^k,fa^{-1})}(p,s)\right)^{-1}}{\big(Q^2-Q^{-2}\big)^2}O_A^{(1)}O_A^{(1)}O_B^{(fa^{-1})}(p,s)\Big)O_A^{(1)}\\
&\;\overset{(\ref{eq:MatrixIdentitiesMCGEquivarianceOb})}{\underset{\textrm{for}\;g=f}{=}}D_A^{-1}\Big(\widetilde c_{\pm}^{(a^{\pm1}b^k,fa^{-1})}(p,s) O_B^{(f)}(p/s,s)\\
&\qquad\qquad\quad+\frac{\widetilde c_{\pm}^{(a^{\pm1}b^k,fa^{-1})}(p,s)-\left(\widetilde c_{\pm}^{(a^{\pm1}b^k,fa^{-1})}(p,s)\right)^{-1}}{\big(Q^2-Q^{-2}\big)^2}O_A^{(1)}O_A^{(1)}O_B^{(f)}(p/s,s)\Big)D_A\\
&\;\overset{(\ref{eq:cpcmTildeRightMultiplicationa}),(\ref{eq:RecursiveRelationOb})}{=}D_A^{-1}O_B^{(a^{\pm1}b^kf)}(p/s,s)D_A.
\end{align*}
\end{proof}

\begin{proposition}
    Relators (\ref{eq:Relation0Xg})--(\ref{eq:Relation7g}) are annihilated by the homomorphism $\Psi:\mathcal F\rightarrow\mathrm{Mat}_{3\times3}(\mathbf R)$.
\label{prop:SevenRelationsAnnihilated}
\end{proposition}
\begin{proof}
By Proposition \ref{prop:MCGEquivariancePsi}, homomorphism $\Psi:\mathcal F\rightarrow\mathrm{Mat}_{3\times3}(\mathbf R)$ is equivariant with respect to $\mathbb Z_2$-action given by automorphism $\mathbf s$ on $\mathcal F$ and $\mathrm{Ad}_{\widehat S}$ on $\Psi(\mathcal F)$. Hence it will be enough to prove that only the left half of relators in (\ref{eq:Z2SymmetryDefiningIdeal}) are annihilated by $\Psi$.

For relators of the first type we get
\begin{equation*}
\Psi\big(\R_{0,A}^{(g)}\big)=O_A^{(ag)}-O_A^{(g)}\overset{(\ref{eq:MatrixOaLeftInvariance})}{=}0.
\end{equation*}

Now we can prove that remaining relators belong to $\ker\Psi$ by induction in $g$. The base case $g=1$ is provided by Lemma \ref{lemm:MatrixRelations234567Forg1}. For the step of induction we assume that all relators (\ref{eq:Relation0Xg})--(\ref{eq:Relation7g}) belong to $\ker\Psi$ for all tuples of words $g,g_1,g_2,g_3\in Free_2$ of the total length strictly less than some $N\in\mathbb Z_{\geq0}$ and then show that each of the relations must hold for tuples of the total length $N$.

Note that due to (\ref{eq:MatrixOaLeftInvariance}) and (\ref{eq:MatrixObLeftInvariance}) we can assume without loss of generality that words $g,g_1,g_2,g_3$ entering solely as an argument of $O_A$ are of the special type (\ref{eq:BReducedWordFree2}), while words entering solely as an argument of $O_B$ are of the dual special type (\ref{eq:AReducedWordFree2}). For each of the calculations below we assume that one of the words is nontrivial and of the type specified above. We then use recursive relation (\ref{eq:RecursiveRelationOa}) or (\ref{eq:RecursiveRelationOb}) to prove the step of induction.
\begin{enumerate}[\qquad1.]
\item Relation (\ref{eq:Relation1Xg}) for $X=A$. Let $g=b^{\pm1}a^kf$ and $c_1:=c_{\pm}^{(b^{\pm1}a^k,f)}$, we get
\begin{align*}
\Psi\big(\R_{1,A}^{(b^{\pm1}a^kf)}\big)=&O_A^{(b^{\pm1}a^kf)}O_A^{(b^{\pm1}a^kf)}-O_A^{(1)}O_A^{(1)}\\
=&c_1^2O_A^{(f)}O_A^{(f)}-O_A^{(1)}O_A^{(1)}\\
&+\frac1{\big(Q^2-Q^{-2}\big)^2}\Big((c_1^2-1)O_A^{(f)}O_B^{(1)}O_B^{(1)}O_A^{(f)}+(c_1^2-1)O_B^{(1)}O_B^{(1)}O_A^{(f)}O_A^{(f)}\Big)\\
&+\frac1{\big(Q^2-Q^{-2}\big)^4}\Big(\frac{(c_1^2-1)^2}{c_1^2}O_B^{(1)}O_B^{(1)}O_A^{(f)}O_B^{(1)}O_B^{(1)}O_A^{(f)}\Big)\\
=&\Psi\big(\R_{1,A}^{(f)}\big)+\frac1{\big(Q^2-Q^{-2}\big)^2}\Big((c_1^2-1)\Psi\big(\R_{7}^{(f)}\big)O_A^{(f)}-\frac{(c_1^2-1)^2}{c_1^2}O_B^{(1)}O_B^{(1)}\Psi\big(\R_{1,A}^{(f)}\big)\Big)\\
&+\frac1{\big(Q^2-Q^{-2}\big)^4}\Big(\frac{(c_1^2-1)^2}{c_1^2}O_B^{(1)}O_B^{(1)}\Psi\big(\R_{7}^{(f)}\big)O_A^{(f)}-\frac{(c_1^2-1)^2}{c_1^2}O_B^{(1)}\Psi\big(\R_{4,B}^{(1)}\big)O_A^{(1)}O_A^{(1)}\\
&\qquad-\frac{(c_1^2-1)^2}{c_1^2}O_B^{(1)}O_B^{(1)}O_B^{(1)}O_B^{(1)}\Psi\big(\R_{1,A}^{(f)}\big)\Big)=0.
\end{align*}
\item Relation (\ref{eq:Relation2g1g2g3}), first argument. Let $g_1=a^{\pm1}b^kf$ and $c_2:=\widetilde c_{\pm}^{(a^{\pm1}b^k,f)}$, we get
\begin{align*}
\Psi\big(\R_{2}^{(a^{\pm1}b^kf,g_2,g_3)}\big)=&O_B^{(a^{\pm1}b^kf)}O_A^{(g_2)}O_B^{(g_3)}\\
=&c_2O_B^{(f)}O_A^{(g_2)}O_B^{(g_3)}+\frac1{\big(Q^2-Q^{-2}\big)^2}\Big((c_2-c_2^{-1})O_A^{(1)}O_A^{(1)}O_B^{(f)}O_A^{(g_2)}O_B^{(g_3)}\Big)\\
=&c_2\Psi\big(\R_{2}^{(f,g_2,g_3)}\big)+\frac1{\big(Q^2-Q^{-2}\big)^2}\Big((c_2-c_2^{-1})O_A^{(1)}O_A^{(1)}\Psi\big(\R_{2}^{(f,g_2,g_3)}\big)\Big)=0.
\end{align*}
\item Relation (\ref{eq:Relation2g1g2g3}), second argument. Let $g_2=b^{\pm1}a^kf$ and $c_3:=c_{\pm}^{(b^{\pm1}a^k,f)}$, we get
\begin{align*}
\Psi\big(\R_{2}^{(g_1,g_2,g_3)}\big)=&O_B^{(g_1)}O_A^{(b^{\pm1}a^kf)}O_B^{(g_3)}\\
=&c_3O_B^{(g_1)}O_A^{(f)}O_B^{(g_3)}+\frac1{\big(Q^2-Q^{-2}\big)^2}\Big((c_3-c_3^{-1})O_B^{(g_1)}O_B^{(1)}O_B^{(1)}O_A^{(f)}O_B^{(g_3)}\Big)\\
=&c_3\Psi\big(\R_{2}^{(g_1,f,g_3)}\big)+\frac1{\big(Q^2-Q^{-2}\big)^2}\Big((c_3-c_3^{-1})O_B^{(g_1)}O_B^{(1)}\Psi\big(\R_{2}^{(1,f,g_3)}\big)\Big)=0.
\end{align*}
\item Relation (\ref{eq:Relation2g1g2g3}), third argument. Let $g_3=a^{\pm1}b^kf$ and $c_4:=\widetilde c_{\pm}^{(a^{\pm1}b^k,f)}$, we get
\begin{align*}
\Psi\big(\R_{2}^{(g_1,g_2,g_3)}\big)=&O_B^{(g_1)}O_A^{(g_2)}O_B^{(a^{\pm1}b^kf)}\\
=&c_4O_B^{(g_1)}O_A^{(g_2)}O_B^{(f)}+\frac1{\big(Q^2-Q^{-2}\big)^2}\Big((c_4-c_4^{-1})O_B^{(g_1)}O_A^{(g_2)}O_A^{(1)}O_A^{(1)}O_B^{(f)}\Big)\\
=&\frac{1}{c_4}\Psi\big(\R_{2}^{(g_1,g_2,f)}\big)+\frac1{\big(Q^2-Q^{-2}\big)^2}\Big((c_4-c_4^{-1})O_B^{(g_1)}O_A^{(g_2)}\Psi\big(\R_{6}^{(f)}\big)\\
&\qquad-(c_4-c_4^{-1})O_B^{(g_1)}\Psi\big(\R_{3}^{(g_2,f,1)}\big)O_A^{(1)}\Big)=0
\end{align*}
\item Relation (\ref{eq:Relation4Xg}) for $X=A$. Let $g=b^{\pm1}a^kf$ and $c_5:=c_{\pm}^{(b^{\pm1}a^k,f)}$, we get
\begin{align*}
\Psi\big(\R_{4,A}^{(b^{\pm1}a^kf)}\big)=&O_A^{(1)}O_A^{(1)}O_A^{(b^{\pm1}a^kf)}+\big(Q^2-Q^{-2}\big)^2O_A^{(b^{\pm1}a^kf)}\\
=&\big(Q^2-Q^{-2}\big)^2c_5O_A^{(f)}+\Big((c_5-c_5^{-1})O_B^{(1)}O_B^{(1)}O_A^{(f)}+c_5O_A^{(1)}O_A^{(1)}O_A^{(f)}\Big)\\
&+\frac1{\big(Q^2-Q^{-2}\big)^2}\Big((c_5-c_5^{-1})O_A^{(1)}O_A^{(1)}O_B^{(1)}O_B^{(1)}O_A^{(f)}\Big)\\
=&c_5\Psi\big(\R_{4,A}^{(f)}\big)+\frac1{\big(Q^2-Q^{-2}\big)^2}\Big((c_5-c_5^{-1})O_A^{(1)}\Psi\big(\R_{7}^{(1)}\big)O_A^{(f)}\\
&\qquad+(c_5-c_5^{-1})O_B^{(1)}O_B^{(1)}\Psi\big(\R_{4,A}^{(f)}\big)-(c_5-c_5^{-1})\Psi\big(\R_{7}^{(1)}\big)O_A^{(1)}O_A^{(f)}\Big)=0.
\end{align*}
\item Relation (\ref{eq:Relation5Xg}) for $X=A$. Let $g=b^{\pm1}a^kf$ and $c_6:=c_{\pm}^{(b^{\pm1}a^k,f)}$, we get
\begin{align*}
\Psi\big(\R_{5,A}^{(b^{\pm1}a^kf)}\big)=&O_A^{(b^{\pm1}a^kf)}O_A^{(1)}O_A^{(1)}+\big(Q^2-Q^{-2}\big)^2O_A^{(b^{\pm1}a^kf)}\\
=&\big(Q^2-Q^{-2}\big)^2c_6O_A^{(f)}+\Big((c_6-c_6^{-1})O_B^{(1)}O_B^{(1)}O_A^{(f)}+c_6O_A^{(f)}O_A^{(1)}O_A^{(1)}\Big)\\
&+\frac1{\big(Q^2-Q^{-2}\big)^2}\Big((c_6-c_6^{-1})O_B^{(1)}O_B^{(1)}O_A^{(f)}O_A^{(1)}O_A^{(1)}\Big)\\
=&c_6\Psi\big(\R_{5,A}^{(f)}\big)+\frac1{\big(Q^2-Q^{-2}\big)^2}\Big((c_6-c_6^{-1})O_B^{(1)}O_B^{(1)}\Psi\big(\R_{5,A}^{(f)}\big)\Big)=0.
\end{align*}
\item Relation (\ref{eq:Relation6g}). Let $g=a^{\pm1}b^kf$ and $c_7:=\widetilde c_{\pm}^{(a^{\pm1}b^k,f)}$
\begin{align*}
\Psi\big(\R_{6}^{(a^{\pm1}b^kf)}\big)=&O_A^{(1)}O_A^{(1)}O_B^{(a^{\pm1}b^kf)}+O_B^{(a^{\pm1}b^kf)}O_A^{(1)}O_A^{(1)}+\big(Q^2-Q^{-2}\big)^2O_B^{(a^{\pm1}b^kf)}\\
=&\big(Q^2-Q^{-2}\big)^2c_7O_B^{(f)}+\Big((2 c_7-c_7^{-1})O_A^{(1)}O_A^{(1)}O_B^{(f)}+c_7O_B^{(f)}O_A^{(1)}O_A^{(1)}\Big)\\
&+\frac1{\big(Q^2-Q^{-2}\big)^2}\Big((c_7-c_7^{-1})O_A^{(1)}O_A^{(1)}O_A^{(1)}O_A^{(1)}O_B^{(f)}+(c_7-c_7^{-1})O_A^{(1)}O_A^{(1)}O_B^{(f)}O_A^{(1)}O_A^{(1)}\Big)\\
=&c_7\Psi\big(\R_{6}^{(f)}\big)+\frac1{\big(Q^2-Q^{-2}\big)^2}\Big((c_7-c_7^{-1})O_A^{(1)}O_A^{(1)}\Psi\big(\R_{6}^{(f)}\big)\Big)=0.
\end{align*}
\end{enumerate}
\end{proof}

\subsection{Action of the shift operators}

We start this subsection by a simple observation that images of idempotents (\ref{eq:IdempotentsDef}) are idependent of $p,s$
\begin{align*}
\Psi(e_A)=-\frac1{\big(Q^2-Q^{-2}\big)^2}O_A^{(1)}O_A^{(1)}=&\left(\begin{array}{ccc}
1&0&0\\
0&0&0\\
0&0&1
\end{array}\right),\\[0.5em]
\Psi(e_B)=-\frac1{\big(Q^2-Q^{-2}\big)^2}O_B^{(1)}O_B^{(1)}=&
    \left(\begin{array}{ccc}
    \frac{1}{2}&0&-\big(Q^2-Q^{-2}\big)^{-2}\\
    0&1&0\\
    -\frac14\big(Q^2-Q^{-2}\big)^2&0&\frac{1}{2}
    \end{array}\right).
\end{align*}
Combining that with Proposition \ref{prop:SevenRelationsAnnihilated} we get 
\begin{lemma}
The following are constant matrices for all $g\in Free_2$:
\begin{equation*}
O_A^{(g)}O_A^{(g)}=\Psi(e_A),\quad O_B^{(g)}O_B^{(g)}=\Psi(e_B)\quad \in\quad\mathrm{Mat}_{3\times 3}(\mathbf k).
\end{equation*}
\end{lemma}

\begin{proposition}
For all $g\in Free_2$ we have
\begin{equation}
O_A^{(ga)}(p,s)=O_A^{(g)}(ps,s),\qquad O_B^{(ga)}(p,s)=O_B^{(g)}(ps,s).
\label{eq:aRightActionOnMatricesIsShift}
\end{equation}
\end{proposition}
\begin{proof}
We will prove both identities by induction in the length of $g$.

For the first identity, our base of induction will be the statement that identity holds for $g=1$ as well as for all $g=b^{\pm}a^k,\;k\in\mathbb Z$. The case $g=1$ is given by the fact that $O_A^{(a)}\overset{(\ref{eq:MatrixOaLeftInvariance})}{=}O_A^{(1)}$ and the fact that $O_A^{(1)}$ is a constant matrix which is independent of $p,s$. While the case $g=b^{\pm1}a^k$ follows from relation (\ref{eq:AuxiliaryConstantsRelationIdentity1Casef1}) on auxiliary constants. Indeed,
\begin{align*}
O_A^{(b^{\pm1}a^{k+1})}(p,s)\overset{(\ref{eq:RecursiveRelationOa})}{=}&c_{\pm}^{(b^{\pm1}a^{k+1})}(p,s)O_A^{(1)}-\left(c_{\pm}^{(b^{\pm1}a^{k+1},f)}(p,s)-\left(c_{\pm}^{(b^{\pm1}a^{k+1},f)}(p,s)\right)^{-1}\right)\Psi(e_B)O_A^{(1)}\\
\overset{(\ref{eq:AuxiliaryConstantsRelationIdentity1Casef1})}{=}&c_{\pm}^{(b^{\pm1}a^{k})}(ps,s)O_A^{(1)}-\left(c_{\pm}^{(b^{\pm1}a^k,f)}(ps,s)-\left(c_{\pm}^{(b^{\pm1}a^k,f)}(ps,s)\right)^{-1}\right)\Psi(e_B)O_A^{(1)}\\
\overset{(\ref{eq:RecursiveRelationOa})}{=}&O_A^{(b^{\pm1}a^k)}(ps,s).
\end{align*}

Now for the step of induction, we assume that left identity in (\ref{eq:aRightActionOnMatricesIsShift}) holds for all $g$ of length strictly less than $N$. Now let $g$ be a reduced word of length $N$, by (\ref{eq:MatrixOaLeftInvariance}) we can assume without loss of generality that $g=b^{\pm1}a^kf$, where $f$ is a reduced word of the special type (\ref{eq:BReducedWordFree2}). Note that by our base of induction we can assume that $f$ is also nontrivial reduced word of special type (\ref{eq:BReducedWordFree2}), hence so is $fa$. We get
\begin{align*}
O_A^{(b^{\pm1}a^kfa)}(p,s)
\overset{(\ref{eq:RecursiveRelationOa}),(\ref{eq:cpcmRightMultiplication})}{=}&c_{\pm}^{(b^{\pm1}a^k,f)}(ps,s)O_A^{(fa)}(p,s)-\left(c_{\pm}^{(b^{\pm1}a^k,f)}(ps,s)-\left(c_{\pm}^{(b^{\pm1}a^k,f)}(ps,s)\right)^{-1}\right)\Psi(e_B)O_A^{(fa)}(p,s)\\
\overset{(\ref{eq:aRightActionOnMatricesIsShift})}{\underset{\textrm{for}\;g=f}{=}}&c_{\pm}^{(b^{\pm1}a^k,f)}(ps,s)O_A^{(f)}(ps,s)-\left(c_{\pm}^{(b^{\pm1}a^k,f)}(ps,s)-\left(c_{\pm}^{(b^{\pm1}a^k,f)}(ps,s)\right)^{-1}\right)\Psi(e_B)O_A^{(fa)}(ps,s)\\[0.3em]
\overset{(\ref{eq:RecursiveRelationOa})}{=}\;\;&O_A^{(b^{\pm1}a^kf)}(ps,s).
\end{align*}

Finally, we prove the right identity in (\ref{eq:aRightActionOnMatricesIsShift}). In order to show the base case $g=1$ one computes the $O_B^{(a)}(ps,s)$ via recursive relation (\ref{eq:RecursiveRelationOb}) and compares the result with $O_B^{(1)}(ps,s)$. Now suppose that the right identity in (\ref{eq:aRightActionOnMatricesIsShift}) holds for all words of length strictly less than $N$ and let $g$ be a reduced word of length $N$. By (\ref{eq:MatrixObLeftInvariance}), we can assume without loss of generality that $g=a^{\pm1}b^kf$, where $f$ is a reduced word of the dual special type (\ref{eq:AReducedWordFree2}) and so is $fa$. The step of induction then follows from the following relation on auxiliary constants
\begin{equation}
\widetilde c_{\pm}^{(a^{\pm1}b^k,fa)}(p,s)\overset{(\ref{eq:ctildeDef})}{=}c_{\pm}^{(b^{\pm1}a^k,\sigma(f)b)}(s,p^{-1})\overset{(\ref{eq:cpcmRightMultiplication})}{=}c_{\pm}^{(b^{\pm1}a^k,\sigma(f))}(s,p^{-1}s^{-1})\overset{(\ref{eq:ctildeDef})}{=}\widetilde c_{\pm}^{(a^{\pm1}b^k,f)}(ps,s).
\label{eq:cTildeaRightMultiplicationIsShift}
\end{equation}
This gives
\begin{align*}
O_B^{(a^{\pm1}b^kfa)}(p,s)\overset{(\ref{eq:RecursiveRelationOb}),(\ref{eq:cTildeaRightMultiplicationIsShift})}{=}&\widetilde c_{\pm}^{(a^{\pm1}b^k,f)}(ps,s)O_B^{(fa)}(p,s)+\left(\widetilde c_{\pm}^{(a^{\pm1}b^k,f)}(ps,s)+\left(c_{\pm}^{(a^{\pm1}b^k,f)}(ps,s)\right)^{-1}\right)\Psi(e_A)O_B^{(fa)}(p,s)\\
\overset{(\ref{eq:aRightActionOnMatricesIsShift})}{\underset{\textrm{for}\;g=f}{=}}\;&\widetilde c_{\pm}^{(a^{\pm1}b^k,f)}(ps,s)O_B^{(f)}(ps,s)+\left(\widetilde c_{\pm}^{(a^{\pm1}b^k,f)}(ps,s)+\left(c_{\pm}^{(a^{\pm1}b^k,f)}(ps,s)\right)^{-1}\right)\Psi(e_A)O_B^{(f)}(ps,s)\\[0.3em]
\overset{(\ref{eq:RecursiveRelationOb})}{=}\quad&O_B^{(a^{\pm1}b^kf)}(ps,s).
\end{align*}
\end{proof}

\begin{corollary}
For all $g\in Free_2$ we have
\begin{equation}
O_A^{(gb)}(p,s)=\left(M(p,s)\right)^{-1}O_A^{(g)}(p,s/p)M(p,s),\qquad O_B^{(gb)}(p,s)=\left(M(p,s)\right)^{-1}O_B^{(g)}(p,s/p)M(p,s),
\label{eq:bRightActionOnMatricesIsConjugationWithShift}
\end{equation}
where
\begin{equation*}
M(p,s):=\left(S(p/s,p)\right)^{-1}S(1/s,p).
\end{equation*}
\end{corollary}
\begin{proof}
For the elements $O_A^{(gb)}$ we get
\begin{equation*}
\begin{aligned}
O_A^{(gb)}(p,s)\overset{(\ref{eq:SHatConjugationFlipsAB})}{=}&\left(\widehat S^{-1}O_B^{(\sigma(g)a)}\widehat S\right)(p,s)=\left(S(1/s,p)\right)^{-1}O_B^{(\sigma(g)a)}(1/s,p)S(1/s,p)\\
\overset{(\ref{eq:aRightActionOnMatricesIsShift})}{=}&\left(S(1/s,p)\right)^{-1}O_B^{(\sigma(g))}(p/s,p)S(1/s,p)\\
\overset{(\ref{eq:SHatConjugationFlipsAB})}{=}&\left(S(1/s,p)\right)^{-1}S(p/s,p)O_A^{(g)}(p,s/p)\left(S(p/s,p)\right)^{-1}S(1/s,p).
\end{aligned}
\end{equation*}
Exchanging $A$ with $B$ in the calculation above we similarly prove the second identity.
\end{proof}

\begin{remark}
While the right hand sides of (\ref{eq:aRightActionOnMatricesIsShift}) represent the action of the shift operator $\widehat\delta_{(p,s)\mapsto(ps,s)}$ on matrix elements, the right hand sides of (\ref{eq:bRightActionOnMatricesIsConjugationWithShift}) can be interpreted as the action of the shift operator $\widehat\delta_{(p,s)\mapsto(p,s/p)}$ on coefficients of certain discrete difference operator. The details of such interpretation, however, are beyond the scope of this paper.
\end{remark}

\begin{corollary}
Let 
\begin{equation*}
P\big(\Oa^{(g_1)},\dots,\Oa^{(g_n)},\Ob^{(f_1)},\dots,\Ob^{(f_m)}\big)\quad\in\quad\ker\Psi\quad\subset\quad\mathcal F
\end{equation*}
be a noncommutative polynomial in generators which is annihilated by $\Psi$, then
\begin{equation*}
P\big(\Oa^{(g_1h)},\dots,\Oa^{(g_nh)},\Ob^{(f_1h)},\dots,\Ob^{(f_mh)}\big)\quad\in\quad\ker\Psi,\qquad\textrm{for all}\quad h\in Free_2.
\end{equation*}
\label{cor:RightActionPreservesKernelOfPsi}
\end{corollary}
\begin{proof}
By induction in length of $h$. The base case $h=1$ is a tautology, whereas the step of induction follows from (\ref{eq:aRightActionOnMatricesIsShift}) for $a^{\pm1}$ and from  (\ref{eq:bRightActionOnMatricesIsConjugationWithShift}) for $b^{\pm1}$ respectively.
\end{proof}

\subsection{Homomorphism property and $PSL(2,\mathbb Z)$-equivariance}

We are now ready to prove the main theorem of this section which links the abstract algebra $\mathcal A$ over $\mathbf k=\mathbb C(Q)$ defined in Section \ref{sec:AlgebraK2} with its matrix representation.

\begin{theorem}
Homomorphism $\Psi$ descends to the quotient algebra $\mathcal A=\mathcal F/\mathcal I$. In other words
\begin{equation*}
\Psi:\mathcal A\rightarrow\mathrm{Mat}_{3\times 3}(\mathbf R),\qquad \Psi(\mathcal O_A^{(g)})=O_A^{(g)},\qquad \Psi(\mathcal O_B^{(g)})=O_B^{(g)},\qquad \textrm{for all}\quad g\in Free_2
\end{equation*}
defines a $PSL(2,\mathbb Z)$-equivariant homomorphism of algebras over the field $\mathbf k=\mathbb C(Q)$.
\label{th:HomomorphismToMatrixRepresentation}
\end{theorem}
\begin{proof}
We have to show that $\mathcal I\subset\ker\Psi$. Note that by Proposition \ref{prop:SevenRelationsAnnihilated} we already know that the subideal $\mathcal I_7\subset\mathcal I$ generated by relations (\ref{eq:Relation0Xg})--(\ref{eq:Relation7g}) is annihilated by $\Psi:\mathcal F\rightarrow\mathrm{Mat}_{3\times3}(\mathbf R)$. At the same time, by Proposition \ref{prop:MCGEquivariancePsi}, we know that $\ker\Psi$ must be invariant under the action of the order two automorphism $\mathbf s:\mathcal F\rightarrow\mathcal F$. Hence, it will be enough for us to prove that only two out four remaining types of relators are annihilated by $\Psi$, say $\R_8^{(g_1,g_2)}$ and $\R_{10}^{(g_1,g_2)}$. Moreover, by Corollary \ref{cor:RightActionPreservesKernelOfPsi} we can assume without loss of generality that $g_2=1$.

A closer look on (\ref{eq:asasOAg}),(\ref{eq:saInverseOAg}), and (\ref{eq:R10ToProveMCG}) suggests that a slightly stronger statement can be made, namely
\begin{equation}
\mathbf a\mathbf s\mathbf a\mathbf s(\Oa^{(g)})-\mathbf s\mathbf a^{-1}(\Oa^{(g)})=-\frac1{\big(Q^2-Q^{-2}\big)^2}\R_{10}^{(ga,1)}\qquad\bmod\;\mathcal I_7;
\label{eq:R10FromMCGUsingI7}
\end{equation}
as only relators from $\mathcal I_7$ were utilized on the first step of the calculation. On the other hand, by Proposition \ref{prop:MCGEquivariancePsi} combined with Theorem \ref{th:SL2ZRelationsHatD} we conclude that the left hand side of (\ref{eq:R10FromMCGUsingI7}) must be annihilated by $\Psi$. Hence $\R_{10}^{(g_1,1)}\in\ker\Psi$ for all $g_1\in Free_2$.

Next, to prove that $\R_8^{(g_1,1)}$ is annihilated by $\Psi$ we repeat the calculations similar to (\ref{eq:asasOAg}),(\ref{eq:saInverseOAg}), and (\ref{eq:R10ToProveMCG}) with the only difference that now we compute the action of the following two words representing the same element of $PSL(2,\mathbb Z)$:
\begin{equation}
\mathbf a^{-1}\mathbf s\mathbf a^{-1}\mathbf s\simeq\mathbf s\mathbf a\qquad \textrm{in}\quad PSL(2,\mathbb Z).
\label{eq:PairOfEquivalentPSL2ZElementsToObtainR8}
\end{equation}
Applying to generator $\Oa^{(g)}$ the sequence of homomorphisms of $\mathcal F$ represented by the word on the left of (\ref{eq:PairOfEquivalentPSL2ZElementsToObtainR8}) we get
\begin{equation}
\begin{aligned}
\mathbf a^{-1}\big(\mathbf s\big(\mathbf a^{-1}\big(\mathbf s\big(\Oa^{(g)}\big)\big)\big)\big)=&\mathbf a^{-1}\big(\mathbf s\big(\mathbf a^{-1}\big(\Ob^{(\sigma(g))}\big)\big)\big)\\
=&\frac1{Q^2-Q^{-2}}\mathbf a^{-1}\Big(\mathbf s\Big(Q\Ob^{(\sigma(g)a)}\Oa^{(1)}-Q^{-1}\Oa^{(1)}\Ob^{(\sigma(g)a)}\Big)\Big)\\
=&\frac1{Q^2-Q^{-2}}\mathbf a^{-1}\Big(Q\Oa^{(gb)}\Ob^{(1)}-Q^{-1}\Ob^{(1)}\Oa^{(gb)}\Big)\\
=&\frac1{\big(Q^2-Q^{-2}\big)^4}\Big(\Oa^{(1)}\Oa^{(gba)}\Oa^{(1)}\Oa^{(1)}\Ob^{(a)}-Q^2\Oa^{(1)}\Oa^{(gba)}\Oa^{(1)}\Ob^{(a)}\Oa^{(1)}\\
&\qquad-Q^{-2}\Oa^{(1)}\Ob^{(a)}\Oa^{(1)}\Oa^{(gba)}\Oa^{(1)}+\Ob^{(a)}\Oa^{(1)}\Oa^{(1)}\Oa^{(gba)}\Oa^{(1)}\Big)\\
=&\frac1{\big(Q^2-Q^{-2}\big)^2}\Big(-\Oa^{(1)}\Oa^{(gba)}\Ob^{(a)}-\Ob^{(a)}\Oa^{(gba)}\Oa^{(1)}\Big)\\
&+\frac1{\big(Q^2-Q^{-2}\big)^4}\Big(\Oa^{(1)}\Oa^{(gba)}\R_{6}^{(a)}-Q^2\Oa^{(1)}\Oa^{(gba)}\R_{3}^{(1,a,1)}-\Oa^{(1)}\R_{3}^{(gba,a,1)}\Oa^{(1)}\\
&\qquad+\Ob^{(a)}\R_{4,A}^{(gba)}\Oa^{(1)}-Q^{-2}\R_{3}^{(1,a,1)}\Oa^{(gba)}\Oa^{(1)}\Big)\\
=&\frac1{\big(Q^2-Q^{-2}\big)^2}\Big(-\Oa^{(1)}\Oa^{(gba)}\Ob^{(a)}-\Ob^{(a)}\Oa^{(gba)}\Oa^{(1)}\Big)\qquad\bmod\;\mathcal I_7.
\end{aligned}
\label{eq:aInversesaInversesOag}
\end{equation}
At the same time, for the word on the right of (\ref{eq:PairOfEquivalentPSL2ZElementsToObtainR8}) we get
\begin{equation}
\begin{aligned}
\mathbf s\big(\mathbf a\big(\Oa^{(g)}\big)\big)=&-\frac1{\big(Q^2-Q^{-2}\big)^2}\mathbf s\Big(\Oa^{(1)}\Oa^{(ga^{-1})}\Oa^{(1)}\Big)=-\frac1{\big(Q^2-Q^{-2}\big)^2}\Ob^{(1)}\Ob^{(\sigma(g)b^{-1})}\Ob^{(1)}.
\end{aligned}
\label{eq:saOag}
\end{equation}
Combining (\ref{eq:aInversesaInversesOag}) with (\ref{eq:saOag}) we have
\begin{equation}
\begin{aligned}
\mathbf a^{-1}\big(\mathbf s\big(\mathbf a^{-1}\big(\mathbf s\big(\Oa^{(g)}\big)\big)\big)\big)-\mathbf s\big(\mathbf a\big(&\Oa^{(g)}\big)\big)\\
\equiv&-\frac1{\big(Q^2-Q^{-2}\big)^2}\Big(\Oa^{(1)}\Oa^{(gba)}\Ob^{(a)}+-\Ob^{(a)}\Oa^{(gba)}\Oa^{(1)}-\Ob^{(1)}\Ob^{(\sigma(g)b^{-1})}\Ob^{(1)}\Big)\\
\equiv&-\frac1{\big(Q^2-Q^{-2}\big)^2}\R_8^{(ga^{-1},1)}\qquad\bmod\;\mathcal I_7.
\end{aligned}
\label{eq:R8FromMCG}
\end{equation}
To conclude the proof we note that by Proposition \ref{prop:MCGEquivariancePsi} combined with Theorem \ref{th:SL2ZRelationsHatD}, the left hand side of (\ref{eq:R8FromMCG}) must be annihilated by $\Psi$. Hence $\R_8^{(g_1,1)}\in\ker\Psi$ for all $g_1\in Free_2$.
\end{proof}

\section{Relation to affine Laumon space}
\label{sec:RelationToLaumon}

In this section, we suggest a relation between the construction of current paper and affine Laumon space. To state this relation, we define functions 

\begin{align*}
\psi_{11} =& \mathrm{pexp}\left( -\dfrac{(Q^4-Q^{-4})(Q^4 + s^2 Q^{-4})p}{(1-p^2)(1-s^2)} \right) \mathrm{pexp}\left( \dfrac{(-(s+2)Q^4 + (2s^2 + s)Q^{-4} ) p^2}{1-s^2} \right) \\
\psi_{12} =& - i (Q^2-Q^{-2}) \mathrm{pexp}\left( -\dfrac{(Q^4-Q^{-4})(Q^4 + s Q^{-4})p}{(1-p^2)(1-s)} \right) \mathrm{pexp}\left( \dfrac{-2(Q^4 - s Q^{-4})(s Q^4 + Q^{-4}) p^2}{1-s^2} \right) \\
\psi_{13} =& - \dfrac{(Q^2-Q^{-2})^2}{2} \mathrm{pexp}\left( -\dfrac{(Q^4-Q^{-4})(Q^4 + s^2 Q^{-4})p}{(1-p^2)(1-s^2)} \right) \mathrm{pexp}\left( \dfrac{(-(s+2)Q^4 + (2s^2 + s)Q^{-4} ) p^2}{1-s^2} \right) \\
\psi_{21} =& \mathrm{pexp}\left( -\dfrac{(Q^8-Q^{-8})sp}{(1-p^2)(1-s^2)} \right) \mathrm{pexp}\left( \dfrac{(-(s+2)Q^4 + (2s^2 + s)Q^{-4} ) p^2}{1-s^2} \right) \\
\psi_{22} =& 0 \\
\psi_{23} =& \dfrac{(Q^2-Q^{-2})^2}{2} \mathrm{pexp}\left( -\dfrac{(Q^8-Q^{-8})sp}{(1-p^2)(1-s^2)} \right) \mathrm{pexp}\left( \dfrac{(-(s+2)Q^4 + (2s^2 + s)Q^{-4} ) p^2}{1-s^2} \right) \\
\psi_{31} =& \mathrm{pexp}\left( -\dfrac{(Q^4-Q^{-4})(Q^4 + s^2 Q^{-4})p}{(1-p^2)(1-s^2)} \right) \mathrm{pexp}\left( \dfrac{(-(s+2)Q^4 + (2s^2 + s)Q^{-4} ) p^2}{1-s^2} \right) \\
\psi_{32} =& i (Q^2-Q^{-2}) \mathrm{pexp}\left( -\dfrac{(Q^4-Q^{-4})(Q^4 + s Q^{-4})p}{(1-p^2)(1-s)} \right) \mathrm{pexp}\left( \dfrac{-2(Q^4 - s Q^{-4})(s Q^4 + Q^{-4}) p^2}{1-s^2} \right) \\
\psi_{33} =& - \dfrac{(Q^2-Q^{-2})^2}{2} \mathrm{pexp}\left( -\dfrac{(Q^4-Q^{-4})(Q^4 + s^2 Q^{-4})p}{(1-p^2)(1-s^2)} \right) \mathrm{pexp}\left( \dfrac{(-(s+2)Q^4 + (2s^2 + s)Q^{-4} ) p^2}{1-s^2} \right)
\end{align*}
\smallskip\\
The so defined functions are related to the operator $O^{(1)}_B$ via

\begin{align*}
\sum\limits_{j=1}^{3} (O^{(1)}_B)_{ij} \psi_{kj} = (X_k+X_k^{-1}) \psi_{ki}    
\end{align*}
\smallskip\\
where $X_k = i Q^{k-2}$. At the same time they are related to the affine Laumon space via the following special functions, which are well-known \cite{Shiraishi'2019} to be characters (\cite{Negut'2011} eqs. (1.10), (1.11)) in K-theory of the affine Laumon space:

\begin{align*}
f^{{\widehat gl}_N} \big( x_1, \ldots, x_N, p \big\vert y_1, \ldots, y_N, s \big| q,t \big) = \sum\limits_{{\vec \lambda}} \ \prod\limits_{i,j = 1}^{N} \frac{{\mathcal N}_{\lambda^{(i)}, \lambda^{(j)}}^{(j-i)}\big( q/t \ y_j / y_i \vert q, s \big)}{{\mathcal N}_{\lambda^{(i)}, \lambda^{(j)}}^{(j-i)}\big( y_j / y_i \vert q,s \big)} \ \prod\limits_{\beta = 1}^{N} \prod\limits_{\alpha \geq 1} \left( \frac{p t x_{\alpha+\beta}}{q x_{\alpha+\beta-1}} \right)^{\lambda^{(\beta)}_\alpha}
\end{align*}
\smallskip\\
which is a sum over all $N$-tuples ${\vec \lambda} = \big\{ \lambda^{(i)} \big\}$ of integer partitions $\lambda^{(i)} = (\lambda^{(i)}_1 \geq \lambda^{(i)}_2 \geq \ldots \geq 0)$ for $i = 1, \ldots, N,$ where

\begin{align*}
{\mathcal N}_{\lambda, \mu}^{(k)}\big( u \vert q, s \big) =
\prod\limits_{b \geq a \geq 1} \ \big( 1 - u q^{-\mu_a + \lambda_{b+1}} s^{-a+b} \big)^{\delta_{a - b - k \vert N}} \ \big( 1 - u q^{\lambda_{a} - \mu_b} s^{a-b-1} \big)^{\delta_{a - b + k + 1 \vert N}}
\end{align*}
\smallskip\\
with $\delta_{n \vert N} = 1$ if $n \bmod N = 0$ and $\delta_{n \vert N} = 0$ otherwise. By convention $x_{n + N} = x_{n}$.

\begin{conjecture}
The functions $\psi_{ij}$ are specializations of the above K-theory character of affine Laumon space,

\begin{align*}
\psi_{ij} = f^{{\widehat gl}_2} \big( X_i p^{-1/2}, X_i^{-1}, p^{1/2} \big\vert X_j s^{-1/2}, X_j^{-1}, s^{1/2} \big| Q^4, -Q^{-4} \big)
\end{align*}
\smallskip\\
for $i,j \in \{0,1,2\}$.

\end{conjecture}

We have verified this conjecture in power series expansion in $p$ up to a few orders, but a full proof will not be given in this paper.

\section{Conclusion}
\label{sec:Conclusion}

In this paper, we presented an algebra which generalizes the spherical $A_1$ DAHA by replacing the two generators $\Oa$ and $\Ob$ with two infinite families of generators parametrized by elements of $Free_2$. We proved that $SL(2,{\mathbb Z})$ acts by automorphisms of this algebra, and also presented a representation of both the algebra elements and automorphisms by $3 \times 3$ matrices. 
\smallskip\\

\paragraph{\textbf{Relation to Felder-Varchenko functions.}} In \cite{FelderVarchenko'2001}, G. Felder and A. Varchenko studied solutions of q-KZB equations on tori, introducing as a result a class of special functions with interesting $SL(3,{\mathbb Z})$ transformation properties. We note that the $SL(2,{\mathbb Z})$ part of these transformations is reminiscent of the $SL(2,{\mathbb Z})$ action that we study in current paper. For example, the heat equation (10) in \cite{FelderVarchenko'2001} is reminiscent of our (\ref{eq:usefulrelation}) by the way parameters $p,s$ are shifted. The major difference is that the r.h.s. of the heat equation is not a finite sum of matrices but rather an integral. It would be interesting to further investigate this relation, especially to find out whether it is possible to establish a $SL(3,{\mathbb Z})$ group action on our algebra.

In view of this it is interesting to note that both Felder-Varchenko functions and the affine Laumon space are related to representation theory of $U_q( \widehat{gl}_n )$, albeit in a somewhat different way. The first relation has been established in \cite{Sun'2015} where Felder-Varchenko functions have been connected to traces of intertwiners of $U_q( {\hat gl}_n )$ (only for $n=2$ though). The second relation has been established in \cite{Negut'2011} where K-theory character of the affine Laumon space has been connected to traces of intertwiners of $U_q( {\hat gl}_n )$ twisted in a certain way (we thank Andrei Negut for pointing this out). Despite apparent similarity, the details of the two relations to quantum affine groups appear to differ. It would be interesting to find a precise comparison between the first and the second.

\paragraph{\textbf{Relation to knot theory.}} In \cite{AganagicShakirov'2015}, M. Aganagic and one of the authors studied a one-parameter deformation of Chern-Simons theory on a restricted class of manifolds admitting an $S^1$ action. This deformation, known as refined Chern-Simons theory, also allows to compute certain knot invariants, although in a restricted setting of torus knots. We note that $SL(2,{\mathbb Z})$ action that we introduce and study in current paper is a further elliptic 2-parameter deformation of refined Chern-Simons theory at level $K=2$ (which corresponds to constraint $t=-q^{-1}$ i.e. three-dimensional representations). We also note that the partition function of an $(n,m)$ torus knot, defined in the same way as in \cite{AganagicShakirov'2015} but with our elliptic representation of $SL(2,Z)$, is invariant under topological symmetry $(n,m) \leftrightarrow (m,n)$ in numerous cases we checked. One therefore obtains elliptic functions associated to torus knots, and it would be interesting to see if they can be related to any homological knot invariants such as Khovanov-Rozansky homology.

\appendix

\section{Calculations for Proposition \ref{prop:ATwistAlgebrAutomorphism}}.
\label{sec:IdealIInvariateAutomorphismA}

Here we present calculations omitted from the proof of Proposition \ref{prop:ATwistAlgebrAutomorphism} in the main text. Namely, we show that ideal $\mathcal I\subset\mathcal F$ is invariant under the action of (\ref{eq:aActionK2Algebra}) by computing the action of $\mathbf a$ on each of the generators. For the first seven relators we get:
\begin{align*}
\mathbf a\big(\R_{1,A}^{(g)}\big)=&\frac1{\big(Q^2-Q^{-2}\big)^4}\Big(\Oa^{(1)}\Oa^{(ga^{-1})}\Oa^{(1)}\Oa^{(1)}\Oa^{(ga^{-1})}\Oa^{(1)}-\Oa^{(1)}\Oa^{(1)}\Oa^{(1)}\Oa^{(1)}\Oa^{(1)}\Oa^{(1)}\Big)\\
=&-\frac1{\big(Q^2-Q^{-2}\big)^2}\Oa^{(1)}\R_{1,A}^{(ga^{-1})}\Oa^{(1)}\\
&+\frac1{\big(Q^2-Q^{-2}\big)^4}\Big(-\Oa^{(1)}\Oa^{(1)}\Oa^{(1)}\R_{4,A}^{(1)}+\Oa^{(1)}\Oa^{(ga^{-1})}\R_{4,A}^{(ga^{-1})}\Oa^{(1)}\Big)\quad\in\quad\mathcal I,\\[1em]
\mathbf a\big(\R_{1,B}^{(g)}\big)=&\frac1{\big(Q^2-Q^{-2}\big)^2}\Big(-Q^2\Oa^{(1)}\Ob^{(a^{-1})}\Oa^{(1)}\Ob^{(a^{-1})}+\Oa^{(1)}\Ob^{(a^{-1})}\Ob^{(a^{-1})}\Oa^{(1)}+Q^2\Oa^{(1)}\Ob^{(ga^{-1})}\Oa^{(1)}\Ob^{(ga^{-1})}\\
&\qquad-\Oa^{(1)}\Ob^{(ga^{-1})}\Ob^{(ga^{-1})}\Oa^{(1)}+\Ob^{(a^{-1})}\Oa^{(1)}\Oa^{(1)}\Ob^{(a^{-1})}-Q^{-2}\Ob^{(a^{-1})}\Oa^{(1)}\Ob^{(a^{-1})}\Oa^{(1)}\\
&\qquad-\Ob^{(ga^{-1})}\Oa^{(1)}\Oa^{(1)}\Ob^{(ga^{-1})}+Q^{-2}\Ob^{(ga^{-1})}\Oa^{(1)}\Ob^{(ga^{-1})}\Oa^{(1)}\Big)\\
=&\R_{1,B}^{(ga^{-1})}-\R_{1,B}^{(a^{-1})}+\frac1{\big(Q^2-Q^{-2}\big)^2}\Big(-Q^2\Oa^{(1)}\R_{2}^{(a^{-1},1,a^{-1})}+Q^2\Oa^{(1)}\R_{2}^{(ga^{-1},1,ga^{-1})}\\
&\qquad+\Ob^{(a^{-1})}\R_{6}^{(a^{-1})}-Q^{-2}\Ob^{(a^{-1})}\R_{3}^{(1,a^{-1},1)}-\Ob^{(ga^{-1})}\R_{6}^{(ga^{-1})}+Q^{-2}\Ob^{(ga^{-1})}\R_{3}^{(1,ga^{-1},1)}\\
&\qquad+\Oa^{(1)}\R_{1,B}^{(a^{-1})}\Oa^{(1)}-\Oa^{(1)}\R_{1,B}^{(ga^{-1})}\Oa^{(1)}-\R_{1,B}^{(a^{-1})}\Oa^{(1)}\Oa^{(1)}+\R_{1,B}^{(ga^{-1})}\Oa^{(1)}\Oa^{(1)}\Big)\qquad\in\quad\mathcal I,\\[1em]
\mathbf a\big(\R_2^{(g_1,g_2,g_3)}\big)=&\frac1{\big(Q^2-Q^{-2}\big)^4}\Big(-Q^2\Oa^{(1)}\Ob^{(g_1a^{-1})}\Oa^{(1)}\Oa^{(g_2a^{-1})}\Oa^{(1)}\Oa^{(1)}\Ob^{(g_3a^{-1})}\\
&\qquad+\Oa^{(1)}\Ob^{(g_1a^{-1})}\Oa^{(1)}\Oa^{(g_2a^{-1})}\Oa^{(1)}\Ob^{(g_3a^{-1})}\Oa^{(1)}+\Ob^{(g_1a^{-1})}\Oa^{(1)}\Oa^{(1)}\Oa^{(g_2a^{-1})}\Oa^{(1)}\Oa^{(1)}\Ob^{(g_3a^{-1})}\\
&\qquad-Q^{-2}\Ob^{(g_1a^{-1})}\Oa^{(1)}\Oa^{(1)}\Oa^{(g_2a^{-1})}\Oa^{(1)}\Ob^{(g_3a^{-1})}\Oa^{(1)}\Big)\\
=&\R_{2}^{(g_1a^{-1},g_2a^{-1},g_3a^{-1})}+\frac1{\big(Q^2-Q^{-2}\big)^2}\Big(-\Ob^{(g_1a^{-1})}\R_{4,A}^{(g_2a^{-1})}\Ob^{(g_3a^{-1})}+Q^2\R_{3}^{(1,g_1a^{-1},1)}\Oa^{(g_2a^{-1})}\Ob^{(g_3a^{-1})}\Big)\\
&+\frac1{\big(Q^2-Q^{-2}\big)^4}\Big(-Q^2\Oa^{(1)}\Ob^{(g_1a^{-1})}\Oa^{(1)}\Oa^{(g_2a^{-1})}\R_{6}^{(g_3a^{-1})}+\Oa^{(1)}\Ob^{(g_1a^{-1})}\Oa^{(1)}\Oa^{(g_2a^{-1})}\R_{3}^{(1,g_3a^{-1},1)}\\
&\qquad+Q^2\Oa^{(1)}\Ob^{(g_1a^{-1})}\Oa^{(1)}\R_{3}^{(g_2a^{-1},g_3a^{-1},1)}\Oa^{(1)}+\Ob^{(g_1a^{-1})}\Oa^{(1)}\Oa^{(1)}\Oa^{(g_2a^{-1})}\R_{6}^{(g_3a^{-1})}\\
&\qquad-Q^{-2}\Ob^{(g_1a^{-1})}\Oa^{(1)}\Oa^{(1)}\Oa^{(g_2a^{-1})}\R_{3}^{(1,g_3a^{-1},1)}-\Ob^{(g_1a^{-1})}\Oa^{(1)}\Oa^{(1)}\R_{3}^{(g_2a^{-1},g_3a^{-1},1)}\Oa^{(1)}\Big)\qquad\in\quad\mathcal I,\\[1em]
\mathbf a\big(\R_3^{(g_1,g_2,g_3)}\big)=&\frac1{\big(Q^2-Q^{-2}\big)^5}\Big(Q\Oa^{(1)}\Oa^{(g_1a^{-1})}\Oa^{(1)}\Oa^{(1)}\Ob^{(g_2a^{-1})}\Oa^{(1)}\Oa^{(g_3a^{-1})}\Oa^{(1)}\\
&\qquad-Q^{-1}\Oa^{(1)}\Oa^{(g_1a^{-1})}\Oa^{(1)}\Ob^{(g_2a^{-1})}\Oa^{(1)}\Oa^{(1)}\Oa^{(g_3a^{-1})}\Oa^{(1)}\Big)\\
=&\frac1{\big(Q^2-Q^{-2}\big)^3}Q^{-1}\Oa^{(1)}\Oa^{(g_1a^{-1})}\R_{3}^{(1,g_2a^{-1},g_3a^{-1})}\Oa^{(1)}\\
&+\frac1{\big(Q^2-Q^{-2}\big)^5}\Big(-Q^{-1}\Oa^{(1)}\Oa^{(g_1a^{-1})}\Oa^{(1)}\Ob^{(g_2a^{-1})}\R_{4,A}^{(g_3a^{-1})}\Oa^{(1)}\\
&\qquad+Q\Oa^{(1)}\Oa^{(g_1a^{-1})}\Oa^{(1)}\R_{3}^{(1,g_2a^{-1},1)}\Oa^{(g_3a^{-1})}\Oa^{(1)}\Big)\qquad\in\quad\mathcal I,\\[1em]
\mathbf a\big(\R_{4,A}^{(g)}\big)=&-\Oa^{(1)}\Oa^{(ga^{-1})}\Oa^{(1)}-\frac1{\big(Q^2-Q^{-2}\big)^6}\Oa^{(1)}\Oa^{(1)}\Oa^{(1)}\Oa^{(1)}\Oa^{(1)}\Oa^{(1)}\Oa^{(1)}\Oa^{(ga^{-1})}\Oa^{(1)}\\
=&-\frac1{\big(Q^2-Q^{-2}\big)^2}\Oa^{(1)}\R_{4,A}^{(ga^{-1})}\Oa^{(1)}-\frac1{\big(Q^2-Q^{-2}\big)^4}\Oa^{(1)}\Oa^{(1)}\Oa^{(1)}\R_{4,A}^{(ga^{-1})}\Oa^{(1)}\\
&-\frac1{\big(Q^2-Q^{-2}\big)^6}\Oa^{(1)}\Oa^{(1)}\Oa^{(1)}\Oa^{(1)}\Oa^{(1)}\R_{4,A}^{(ga^{-1})}\Oa^{(1)}\qquad\in\quad\mathcal I,\\[1em]
\mathbf a\big(\R_{4,B}^{(g)}\big)=&\big(Q^2-Q^{-2}\big)\Big(Q\Oa^{(1)}\Ob^{(ga^{-1})}-Q^{-1}\Ob^{(ga^{-1})}\Oa^{(1)}\Big)+\frac1{\big(Q^2-Q^{-2}\big)^3}\Big(Q^3\Oa^{(1)}\Ob^{(a^{-1})}\Oa^{(1)}\Ob^{(a^{-1})}\Oa^{(1)}\Ob^{(ga^{-1})}\\
&\qquad-Q\Oa^{(1)}\Ob^{(a^{-1})}\Oa^{(1)}\Ob^{(a^{-1})}\Ob^{(ga^{-1})}\Oa^{(1)}-Q\Oa^{(1)}\Ob^{(a^{-1})}\Ob^{(a^{-1})}\Oa^{(1)}\Oa^{(1)}\Ob^{(ga^{-1})}\\
&\qquad+Q^{-1}\Oa^{(1)}\Ob^{(a^{-1})}\Ob^{(a^{-1})}\Oa^{(1)}\Ob^{(ga^{-1})}\Oa^{(1)}-Q\Ob^{(a^{-1})}\Oa^{(1)}\Oa^{(1)}\Ob^{(a^{-1})}\Oa^{(1)}\Ob^{(ga^{-1})}\\
&\qquad+Q^{-1}\Ob^{(a^{-1})}\Oa^{(1)}\Oa^{(1)}\Ob^{(a^{-1})}\Ob^{(ga^{-1})}\Oa^{(1)}+Q^{-1}\Ob^{(a^{-1})}\Oa^{(1)}\Ob^{(a^{-1})}\Oa^{(1)}\Oa^{(1)}\Ob^{(ga^{-1})}\\
&\qquad-Q^{-3}\Ob^{(a^{-1})}\Oa^{(1)}\Ob^{(a^{-1})}\Oa^{(1)}\Ob^{(ga^{-1})}\Oa^{(1)}\Big)\\
=&\frac1{Q^2-Q^{-2}}\Big(Q\Oa^{(1)}\R_{4,B}^{(ga^{-1})}-Q^{-1}\R_{4,B}^{(ga^{-1})}\Oa^{(1)}-Q^{-1}\R_{2}^{(a^{-1},1,a^{-1})}\Ob^{(ga^{-1})}-Q\R_{3}^{(1,ga^{-1},1)}\Oa^{(1)}\\
&\qquad+Q\Oa^{(1)}\R_{1,B}^{(a^{-1})}\Ob^{(ga^{-1})}-Q^{-1}\R_{1,B}^{(a^{-1})}\Ob^{(ga^{-1})}\Oa^{(1)}\Big)\\
&+\frac1{\big(Q^2-Q^{-2}\big)^3}\Big(Q^3\Oa^{(1)}\Ob^{(a^{-1})}\Oa^{(1)}\R_{2}^{(a^{-1},1,ga^{-1})}-Q\Oa^{(1)}\Ob^{(a^{-1})}\Ob^{(a^{-1})}\R_{6}^{(ga^{-1})}\\
&\qquad+Q^{-1}\Oa^{(1)}\Ob^{(a^{-1})}\Ob^{(a^{-1})}\R_{3}^{(1,ga^{-1},1)}+Q\Oa^{(1)}\R_{4,B}^{(ga^{-1})}\Oa^{(1)}\Oa^{(1)}\\
&\qquad-Q\Oa^{(1)}\R_{2}^{(a^{-1},1,a^{-1})}\Ob^{(ga^{-1})}\Oa^{(1)}
-Q\Ob^{(a^{-1})}\Oa^{(1)}\Oa^{(1)}\R_{2}^{(a^{-1},1,ga^{-1})}\\
&\qquad+Q^{-1}\Ob^{(a^{-1})}\Oa^{(1)}\Ob^{(a^{-1})}\R_{6}^{(ga^{-1})}-Q^{-3}\Ob^{(a^{-1})}\Oa^{(1)}\Ob^{(a^{-1})}\R_{3}^{(1,ga^{-1},1)}\\
&\qquad-Q^{-1}\Ob^{(a^{-1})}\Ob^{(a^{-1})}\Oa^{(1)}\R_{3}^{(1,ga^{-1},1)}+Q^{-1}\Ob^{(a^{-1})}\R_{6}^{(a^{-1})}\Ob^{(ga^{-1})}\Oa^{(1)}\\
&\qquad-Q^{-1}\R_{2}^{(a^{-1},1,a^{-1})}\Ob^{(ga^{-1})}\Oa^{(1)}\Oa^{(1)}+Q\Oa^{(1)}\R_{1,B}^{(a^{-1})}\Ob^{(ga^{-1})}\Oa^{(1)}\Oa^{(1)}\Big)\qquad\in\quad\mathcal I,\\[1em]
\mathbf a\big(\R_{5,A}^{(g)}\big)=&--\Oa^{(1)}\Oa^{(ga^{-1})}\Oa^{(1)}-\frac1{\big(Q^2-Q^{-2}\big)^6}\Oa^{(1)}\Oa^{(ga^{-1})}\Oa^{(1)}\Oa^{(1)}\Oa^{(1)}\Oa^{(1)}\Oa^{(1)}\Oa^{(1)}\Oa^{(1)}\\
=&-\frac1{\big(Q^2-Q^{-2}\big)^2}\Oa^{(1)}\Oa^{(ga^{-1})}\R_{4,A}^{(1)}-\frac1{\big(Q^2-Q^{-2}\big)^4}\Oa^{(1)}\Oa^{(ga^{-1})}\Oa^{(1)}\Oa^{(1)}\R_{4,A}^{(1)}\\
&\qquad-\frac1{\big(Q^2-Q^{-2}\big)^6}\Oa^{(1)}\Oa^{(ga^{-1})}\Oa^{(1)}\Oa^{(1)}\Oa^{(1)}\Oa^{(1)}\R_{4,A}^{(1)}\qquad\in\quad\mathcal I,\\[1em]
\mathbf a\big(\R_{5,B}^{(g)}\big)=&\big(Q^2-Q^{-2}\big)\Big(Q\Oa^{(1)}\Ob^{(ga^{-1})}-Q^{-1}\Ob^{(ga^{-1})}\Oa^{(1)}\Big)\\
&+\frac1{\big(Q^2-Q^{-2}\big)^3}\Big(Q^3\Oa^{(1)}\Ob^{(ga^{-1})}\Oa^{(1)}\Ob^{(a^{-1})}\Oa^{(1)}\Ob^{(a^{-1})}-Q\Oa^{(1)}\Ob^{(ga^{-1})}\Oa^{(1)}\Ob^{(a^{-1})}\Ob^{(a^{-1})}\Oa^{(1)}\\
&\qquad-Q\Oa^{(1)}\Ob^{(ga^{-1})}\Ob^{(a^{-1})}\Oa^{(1)}\Oa^{(1)}\Ob^{(a^{-1})}+Q^{-1}\Oa^{(1)}\Ob^{(ga^{-1})}\Ob^{(a^{-1})}\Oa^{(1)}\Ob^{(a^{-1})}\Oa^{(1)}\\
&\qquad-Q\Ob^{(ga^{-1})}\Oa^{(1)}\Oa^{(1)}\Ob^{(a^{-1})}\Oa^{(1)}\Ob^{(a^{-1})}+Q^{-1}\Ob^{(ga^{-1})}\Oa^{(1)}\Oa^{(1)}\Ob^{(a^{-1})}\Ob^{(a^{-1})}\Oa^{(1)}\\
&\qquad+Q^{-1}\Ob^{(ga^{-1})}\Oa^{(1)}\Ob^{(a^{-1})}\Oa^{(1)}\Oa^{(1)}\Ob^{(a^{-1})}-Q^{-3}\Ob^{(ga^{-1})}\Oa^{(1)}\Ob^{(a^{-1})}\Oa^{(1)}\Ob^{(a^{-1})}\Oa^{(1)}\Big)\\
=&\frac1{Q^2-Q^{-2}}\Big(Q\Oa^{(1)}\R_{5,B}^{(ga^{-1})}-Q^{-1}\Ob^{(ga^{-1})}\R_{7}^{(1)}+Q^{-1}\Ob^{(ga^{-1})}\R_{4,A}^{(1)}-Q^{-1}\R_{5,B}^{(ga^{-1})}\Oa^{(1)}\\
&\qquad-Q\R_{3}^{(1,ga^{-1},1)}\Oa^{(1)}+Q\Oa^{(1)}\Ob^{(ga^{-1})}\R_{1,B}^{(a^{-1})}-Q^{-1}\Ob^{(ga^{-1})}\Oa^{(1)}\R_{1,B}^{(a^{-1})}\Big)\\
&+\frac1{\big(Q^2-Q^{-2}\big)^3}\Big(Q^3\Oa^{(1)}\Ob^{(ga^{-1})}\Oa^{(1)}\R_{2}^{(a^{-1},1,a^{-1})}-Q\Oa^{(1)}\Ob^{(ga^{-1})}\Ob^{(a^{-1})}\R_{6}^{(a^{-1})}\\
&\qquad+Q^{-1}\Oa^{(1)}\Ob^{(ga^{-1})}\Ob^{(a^{-1})}\R_{3}^{(1,a^{-1},1)}-Q\Oa^{(1)}\Ob^{(ga^{-1})}\R_{7}^{(1)}\Oa^{(1)}\\
&\qquad+2 Q\Oa^{(1)}\R_{5,B}^{(ga^{-1})}\Oa^{(1)}\Oa^{(1)}-Q\Ob^{(ga^{-1})}\Oa^{(1)}\Oa^{(1)}\R_{2}^{(a^{-1},1,a^{-1})}\\
&\qquad+Q^{-1}\Ob^{(ga^{-1})}\Oa^{(1)}\Ob^{(a^{-1})}\R_{6}^{(a^{-1})}-Q^{-3}\Ob^{(ga^{-1})}\Oa^{(1)}\Ob^{(a^{-1})}\R_{3}^{(1,a^{-1},1)}\\
&\qquad+Q^{-1}\Ob^{(ga^{-1})}\Oa^{(1)}\R_{7}^{(1)}\Oa^{(1)}+2Q^{-1}\Ob^{(ga^{-1})}\Ob^{(1)}\Ob^{(1)}\R_{4,A}^{(1)}\\
&\qquad-2Q^{-1}\Ob^{(ga^{-1})}\R_{7}^{(1)}\Oa^{(1)}\Oa^{(1)}-Q\Oa^{(1)}\Ob^{(ga^{-1})}\Oa^{(1)}\R_{1,B}^{(a^{-1})}\Oa^{(1)}\\
&\qquad+Q\Oa^{(1)}\Ob^{(ga^{-1})}\R_{1,B}^{(a^{-1})}\Oa^{(1)}\Oa^{(1)}+Q^{-1}\Ob^{(ga^{-1})}\Oa^{(1)}\Oa^{(1)}\R_{1,B}^{(a^{-1})}\Oa^{(1)}\\
&\qquad-Q^{-1}\Ob^{(ga^{-1})}\Oa^{(1)}\R_{1,B}^{(a^{-1})}\Oa^{(1)}\Oa^{(1)}\Big)\qquad\in\quad\mathcal I,\\[1em]
\mathbf a\big(\R_6^{(g)}\big)=&\big(Q^2-Q^{-2}\big)\Big(Q\Oa^{(1)}\Ob^{(ga^{-1})}-Q^{-1}\Ob^{(ga^{-1})}\Oa^{(1)}\Big)\\
&+\frac1{\big(Q^2-Q^{-2}\big)^5}\Big(Q\Oa^{(1)}\Oa^{(1)}\Oa^{(1)}\Oa^{(1)}\Oa^{(1)}\Oa^{(1)}\Oa^{(1)}\Ob^{(ga^{-1})}\\
&\qquad-Q^{-1}\Oa^{(1)}\Oa^{(1)}\Oa^{(1)}\Oa^{(1)}\Oa^{(1)}\Oa^{(1)}\Ob^{(ga^{-1})}\Oa^{(1)}+Q\Oa^{(1)}\Ob^{(ga^{-1})}\Oa^{(1)}\Oa^{(1)}\Oa^{(1)}\Oa^{(1)}\Oa^{(1)}\Oa^{(1)}\\
&\qquad-Q^{-1}\Ob^{(ga^{-1})}\Oa^{(1)}\Oa^{(1)}\Oa^{(1)}\Oa^{(1)}\Oa^{(1)}\Oa^{(1)}\Oa^{(1)}\Big)\\
=&\frac1{Q^2-Q^{-2}}\Big(Q\Oa^{(1)}\R_{6}^{(ga^{-1})}-Q^{-1}\Ob^{(ga^{-1})}\R_{4,A}^{(1)}\Big)+\frac1{\big(Q^2-Q^{-2}\big)^3}\Big(-Q\Oa^{(1)}\Oa^{(1)}\Oa^{(1)}\R_{6}^{(ga^{-1})}\\
&\qquad+Q\Oa^{(1)}\Oa^{(1)}\R_{3}^{(1,ga^{-1},1)}\Oa^{(1)}-Q\Oa^{(1)}\Ob^{(ga^{-1})}\Oa^{(1)}\R_{4,A}^{(1)}+Q^{-1}\Ob^{(ga^{-1})}\Oa^{(1)}\Oa^{(1)}\R_{4,A}^{(1)}\Big)\\
&+\frac1{\big(Q^2-Q^{-2}\big)^5}\Big(Q\Oa^{(1)}\Oa^{(1)}\Oa^{(1)}\Oa^{(1)}\Oa^{(1)}\R_{6}^{(ga^{-1})}-Q^{-1}\Oa^{(1)}\Oa^{(1)}\Oa^{(1)}\Oa^{(1)}\Oa^{(1)}\R_{3}^{(1,ga^{-1},1)}\\
&\qquad-Q\Oa^{(1)}\Oa^{(1)}\Oa^{(1)}\Oa^{(1)}\R_{3}^{(1,ga^{-1},1)}\Oa^{(1)}+Q\Oa^{(1)}\Ob^{(ga^{-1})}\Oa^{(1)}\Oa^{(1)}\Oa^{(1)}\R_{4,A}^{(1)}\\
&\qquad-Q^{-1}\Ob^{(ga^{-1})}\Oa^{(1)}\Oa^{(1)}\Oa^{(1)}\Oa^{(1)}\R_{4,A}^{(1)}\Big)\qquad\in\quad\mathcal I,\\[1em]
\mathbf a\big(\R_7^{(g)}\big)=&-\Oa^{(1)}\Oa^{(ga^{-1})}\Oa^{(1)}+\frac1{\big(Q^2-Q^{-2}\big)^4}\Big(-Q^2\Oa^{(1)}\Oa^{(ga^{-1})}\Oa^{(1)}\Oa^{(1)}\Ob^{(a^{-1})}\Oa^{(1)}\Ob^{(a^{-1})}\\
&\qquad+\Oa^{(1)}\Oa^{(ga^{-1})}\Oa^{(1)}\Oa^{(1)}\Ob^{(a^{-1})}\Ob^{(a^{-1})}\Oa^{(1)}+\Oa^{(1)}\Oa^{(ga^{-1})}\Oa^{(1)}\Ob^{(a^{-1})}\Oa^{(1)}\Oa^{(1)}\Ob^{(a^{-1})}\\
&\qquad-Q^{-2}\Oa^{(1)}\Oa^{(ga^{-1})}\Oa^{(1)}\Ob^{(a^{-1})}\Oa^{(1)}\Ob^{(a^{-1})}\Oa^{(1)}-Q^2\Oa^{(1)}\Ob^{(a^{-1})}\Oa^{(1)}\Ob^{(a^{-1})}\Oa^{(1)}\Oa^{(ga^{-1})}\Oa^{(1)}\\
&\qquad+\Oa^{(1)}\Ob^{(a^{-1})}\Ob^{(a^{-1})}\Oa^{(1)}\Oa^{(1)}\Oa^{(ga^{-1})}\Oa^{(1)}+\Ob^{(a^{-1})}\Oa^{(1)}\Oa^{(1)}\Ob^{(a^{-1})}\Oa^{(1)}\Oa^{(ga^{-1})}\Oa^{(1)}\\
&\qquad-Q^{-2}\Ob^{(a^{-1})}\Oa^{(1)}\Ob^{(a^{-1})}\Oa^{(1)}\Oa^{(1)}\Oa^{(ga^{-1})}\Oa^{(1)}\Big)\\
=&\frac1{\big(Q^2-Q^{-2}\big)^2}\Big(-\Oa^{(1)}\Oa^{(ga^{-1})}\R_{7}^{(1)}+\Oa^{(1)}\Oa^{(ga^{-1})}\R_{4,A}^{(1)}-\Oa^{(1)}\R_{7}^{(ga^{-1})}\Oa^{(1)}\\
&\qquad+Q^{-2}\Ob^{(a^{-1})}\R_{3}^{(1,a^{-1},ga^{-1})}\Oa^{(1)}-\Oa^{(1)}\Oa^{(ga^{-1})}\Oa^{(1)}\R_{1,B}^{(a^{-1})}-\Oa^{(1)}\R_{1,B}^{(a^{-1})}\Oa^{(ga^{-1})}\Oa^{(1)}\Big)\\
&+\frac1{\big(Q^2-Q^{-2}\big)^4}\Big(-Q^2\Oa^{(1)}\Oa^{(ga^{-1})}\Oa^{(1)}\Oa^{(1)}\R_{2}^{(a^{-1},1,a^{-1})}+\Oa^{(1)}\Oa^{(ga^{-1})}\Oa^{(1)}\Ob^{(a^{-1})}\R_{6}^{(a^{-1})}\\
&\qquad-Q^{-2}\Oa^{(1)}\Oa^{(ga^{-1})}\Oa^{(1)}\Ob^{(a^{-1})}\R_{3}^{(1,a^{-1},1)}+\Oa^{(1)}\Oa^{(ga^{-1})}\Oa^{(1)}\R_{7}^{(1)}\Oa^{(1)}\\
&\qquad+2\Oa^{(1)}\Oa^{(ga^{-1})}\Ob^{(1)}\Ob^{(1)}\R_{4,A}^{(1)}-2\Oa^{(1)}\Oa^{(ga^{-1})}\R_{7}^{(1)}\Oa^{(1)}\Oa^{(1)}\\
&\qquad+\Oa^{(1)}\Ob^{(a^{-1})}\Ob^{(a^{-1})}\R_{4,A}^{(ga^{-1})}\Oa^{(1)}-Q^2\Oa^{(1)}\Ob^{(a^{-1})}\R_{3}^{(1,a^{-1},1)}\Oa^{(ga^{-1})}\Oa^{(1)}\\
&\qquad-Q^{-2}\Ob^{(a^{-1})}\Oa^{(1)}\Ob^{(a^{-1})}\R_{4,A}^{(ga^{-1})}\Oa^{(1)}+\Ob^{(a^{-1})}\Oa^{(1)}\R_{3}^{(1,a^{-1},1)}\Oa^{(ga^{-1})}\Oa^{(1)}\\
&\qquad+\Oa^{(1)}\Oa^{(ga^{-1})}\Oa^{(1)}\Oa^{(1)}\R_{1,B}^{(a^{-1})}\Oa^{(1)}-\Oa^{(1)}\Oa^{(ga^{-1})}\Oa^{(1)}\R_{1,B}^{(a^{-1})}\Oa^{(1)}\Oa^{(1)}\Big)\qquad\in\quad\mathcal I.
\end{align*}

For the remaining four relations we eliminate all occurences of $\sigma(g)$ first by subtracting the appropriate combination of $\R_8,\dots,\R_{10}$ and then show that the difference belongs to the ideal generated by $\R_1,\dots,\R_7$. This gives
\begin{align*}
\mathbf a\big(\R_{8}^{(g_1,g_2)}\big)&+\frac1{Q^2-Q^{-2}}\Big(-Q\Oa^{(1)}\R_{8}^{(g_1,g_2a^{-1})}+Q^{-1}\R_{8}^{(g_1,g_2a^{-1})}\Oa^{(1)}\Big)\\
=&\frac1{Q^2-Q^{-2}}\Big(-Q\Oa^{(1)}\Oa^{(g_2a^{-1})}\Oa^{(g_1abag_2a^{-1})}\Ob^{(ag_2a^{-1})}+Q\Oa^{(1)}\Ob^{(g_2a^{-1})}\Ob^{(\sigma(g_1)g_2a^{-1})}\Ob^{(g_2a^{-1})}\\
&\qquad-Q\Oa^{(1)}\Ob^{(ag_2a^{-1})}\Oa^{(g_1abag_2a^{-1})}\Oa^{(g_2a^{-1})}+Q^{-1}\Oa^{(g_2a^{-1})}\Oa^{(g_1abag_2a^{-1})}\Ob^{(ag_2a^{-1})}\Oa^{(1)}\\
&\qquad-Q^{-1}\Ob^{(g_2a^{-1})}\Ob^{(\sigma(g_1)g_2a^{-1})}\Ob^{(g_2a^{-1})}\Oa^{(1)}+Q^{-1}\Ob^{(ag_2a^{-1})}\Oa^{(g_1abag_2a^{-1})}\Oa^{(g_2a^{-1})}\Oa^{(1)}\Big)\\
&+\frac1{\big(Q^2-Q^{-2}\big)^3}\Big(-Q^3\Oa^{(1)}\Ob^{(g_2a^{-1})}\Oa^{(1)}\Ob^{(\sigma(g_1)g_2a^{-1})}\Oa^{(1)}\Ob^{(g_2a^{-1})}\\
&\qquad+Q\Oa^{(1)}\Ob^{(g_2a^{-1})}\Oa^{(1)}\Ob^{(\sigma(g_1)g_2a^{-1})}\Ob^{(g_2a^{-1})}\Oa^{(1)}+Q\Oa^{(1)}\Ob^{(g_2a^{-1})}\Ob^{(\sigma(g_1)g_2a^{-1})}\Oa^{(1)}\Oa^{(1)}\Ob^{(g_2a^{-1})}\\
&\qquad-Q^{-1}\Oa^{(1)}\Ob^{(g_2a^{-1})}\Ob^{(\sigma(g_1)g_2a^{-1})}\Oa^{(1)}\Ob^{(g_2a^{-1})}\Oa^{(1)}+Q\Ob^{(g_2a^{-1})}\Oa^{(1)}\Oa^{(1)}\Ob^{(\sigma(g_1)g_2a^{-1})}\Oa^{(1)}\Ob^{(g_2a^{-1})}\\
&\qquad-Q^{-1}\Ob^{(g_2a^{-1})}\Oa^{(1)}\Oa^{(1)}\Ob^{(\sigma(g_1)g_2a^{-1})}\Ob^{(g_2a^{-1})}\Oa^{(1)}-Q^{-1}\Ob^{(g_2a^{-1})}\Oa^{(1)}\Ob^{(\sigma(g_1)g_2a^{-1})}\Oa^{(1)}\Oa^{(1)}\Ob^{(g_2a^{-1})}\\
&\qquad+Q^{-3}\Ob^{(g_2a^{-1})}\Oa^{(1)}\Ob^{(\sigma(g_1)g_2a^{-1})}\Oa^{(1)}\Ob^{(g_2a^{-1})}\Oa^{(1)}\Big)\\
&+\frac1{\big(Q^2-Q^{-2}\big)^5}\Big(Q\Oa^{(1)}\Oa^{(g_2a^{-1})}\Oa^{(1)}\Oa^{(1)}\Oa^{(g_1abag_2a^{-1})}\Oa^{(1)}\Oa^{(1)}\Ob^{(ag_2a^{-1})}\\
&\qquad-Q^{-1}\Oa^{(1)}\Oa^{(g_2a^{-1})}\Oa^{(1)}\Oa^{(1)}\Oa^{(g_1abag_2a^{-1})}\Oa^{(1)}\Ob^{(ag_2a^{-1})}\Oa^{(1)}\\
&\qquad+Q\Oa^{(1)}\Ob^{(ag_2a^{-1})}\Oa^{(1)}\Oa^{(g_1abag_2a^{-1})}\Oa^{(1)}\Oa^{(1)}\Oa^{(g_2a^{-1})}\Oa^{(1)}\\
&\qquad-Q^{-1}\Ob^{(ag_2a^{-1})}\Oa^{(1)}\Oa^{(1)}\Oa^{(g_1abag_2a^{-1})}\Oa^{(1)}\Oa^{(1)}\Oa^{(g_2a^{-1})}\Oa^{(1)}\Big)\\
=&\frac1{Q^2-Q^{-2}}\Big(Q^{-1}\Oa^{(g_2a^{-1})}\R_{3}^{(g_1abag_2a^{-1},ag_2a^{-1},1)}+Q^{-1}\R_{2}^{(g_2a^{-1},1,\sigma(g_1)g_2a^{-1})}\Ob^{(g_2a^{-1})}\\
&\qquad-Q\R_{3}^{(1,ag_2a^{-1},g_1abag_2a^{-1})}\Oa^{(g_2a^{-1})}\Big)\\
&+\frac1{\big(Q^2-Q^{-2}\big)^3}\Big(-Q\Oa^{(1)}\Oa^{(g_2a^{-1})}\R_{4,A}^{(g_1abag_2a^{-1})}\Ob^{(ag_2a^{-1})}-Q^3\Oa^{(1)}\Ob^{(g_2a^{-1})}\Oa^{(1)}\R_{2}^{(\sigma(g_1)g_2a^{-1},1,g_2a^{-1})}\\
&\qquad+Q\Oa^{(1)}\Ob^{(g_2a^{-1})}\Ob^{(\sigma(g_1)g_2a^{-1})}\R_{6}^{(g_2a^{-1})}-Q^{-1}\Oa^{(1)}\Ob^{(g_2a^{-1})}\Ob^{(\sigma(g_1)g_2a^{-1})}\R_{3}^{(1,g_2a^{-1},1)}\\
&\qquad+Q\Oa^{(1)}\R_{2}^{(g_2a^{-1},1,\sigma(g_1)g_2a^{-1})}\Ob^{(g_2a^{-1})}\Oa^{(1)}+Q\Ob^{(g_2a^{-1})}\Oa^{(1)}\Oa^{(1)}\R_{2}^{(\sigma(g_1)g_2a^{-1},1,g_2a^{-1})}\\
&\qquad-Q^{-1}\Ob^{(g_2a^{-1})}\Oa^{(1)}\Ob^{(\sigma(g_1)g_2a^{-1})}\R_{6}^{(g_2a^{-1})}+Q^{-3}\Ob^{(g_2a^{-1})}\Oa^{(1)}\Ob^{(\sigma(g_1)g_2a^{-1})}\R_{3}^{(1,g_2a^{-1},1)}\\
&\qquad+Q^{-1}\Ob^{(g_2a^{-1})}\Ob^{(\sigma(g_1)g_2a^{-1})}\Oa^{(1)}\R_{3}^{(1,g_2a^{-1},1)}-Q^{-1}\Ob^{(g_2a^{-1})}\R_{6}^{(\sigma(g_1)g_2a^{-1})}\Ob^{(g_2a^{-1})}\Oa^{(1)}\\
&\qquad+Q^{-1}\Ob^{(ag_2a^{-1})}\R_{4,A}^{(g_1abag_2a^{-1})}\Oa^{(g_2a^{-1})}\Oa^{(1)}+Q^{-1}\R_{2}^{(g_2a^{-1},1,\sigma(g_1)g_2a^{-1})}\Ob^{(g_2a^{-1})}\Oa^{(1)}\Oa^{(1)}\\
&\qquad-Q\R_{3}^{(1,ag_2a^{-1},1)}\Oa^{(g_1abag_2a^{-1})}\Oa^{(g_2a^{-1})}\Oa^{(1)}-Q\Oa^{(1)}\Ob^{(g_2a^{-1})}\Ob^{(\sigma(g_1)g_2a^{-1})}\Ob^{(g_2a^{-1})}\Oa^{(1)}\Oa^{(1)}\Big)\\
&+\frac1{\big(Q^2-Q^{-2}\big)^5}\Big(Q\Oa^{(1)}\Oa^{(g_2a^{-1})}\Oa^{(1)}\Oa^{(1)}\Oa^{(g_1abag_2a^{-1})}\R_{6}^{(ag_2a^{-1})}\\
&\qquad-Q^{-1}\Oa^{(1)}\Oa^{(g_2a^{-1})}\Oa^{(1)}\Oa^{(1)}\Oa^{(g_1abag_2a^{-1})}\R_{3}^{(1,ag_2a^{-1},1)}\\
&\qquad-Q\Oa^{(1)}\Oa^{(g_2a^{-1})}\Oa^{(1)}\Oa^{(1)}\R_{3}^{(g_1abag_2a^{-1},ag_2a^{-1},1)}\Oa^{(1)}\\
&\qquad+Q\Oa^{(1)}\Ob^{(ag_2a^{-1})}\Oa^{(1)}\Oa^{(g_1abag_2a^{-1})}\R_{4,A}^{(g_2a^{-1})}\Oa^{(1)}\\
&\qquad-Q^{-1}\Ob^{(ag_2a^{-1})}\Oa^{(1)}\Oa^{(1)}\Oa^{(g_1abag_2a^{-1})}\R_{4,A}^{(g_2a^{-1})}\Oa^{(1)}\Big)\\
\equiv&-\frac1{\big(Q^2-Q^{-2}\big)^3}Q\Oa^{(1)}\Ob^{(g_2a^{-1})}\Ob^{(\sigma(g_1)g_2a^{-1})}\Ob^{(g_2a^{-1})}\Oa^{(1)}\Oa^{(1)}\;\;\overset{(\ref{eq:IdempotentARightCancellation})}{\equiv}\;\;0\;\bmod\mathcal I,\\[1em]
\mathbf a\big(\R_{9}^{(g_1,g_2)}\big)&+\frac1{\big(Q^2-Q^{-2}\big)^2}\Oa^{(1)}\R_{9}^{(g_1,g_2a^{-1})}\Oa^{(1)}\\
=&\frac1{\big(Q^2-Q^{-2}\big)^2}\Big(-\Oa^{(1)}\Oa^{(g_2a^{-1})}\Oa^{(\sigma(g_1)g_2a^{-1})}\Oa^{(g_2a^{-1})}\Oa^{(1)}+\Oa^{(1)}\Oa^{(bg_2a^{-1})}\Ob^{(g_1babg_2a^{-1})}\Ob^{(g_2a^{-1})}\Oa^{(1)}\\
&\qquad+\Oa^{(1)}\Ob^{(g_2a^{-1})}\Ob^{(g_1babg_2a^{-1})}\Oa^{(bg_2a^{-1})}\Oa^{(1)}\Big)\\
&+\frac1{\big(Q^2-Q^{-2}\big)^4}\Big(-Q^2\Oa^{(1)}\Oa^{(bg_2a^{-1})}\Oa^{(1)}\Oa^{(1)}\Ob^{(g_1babg_2a^{-1})}\Oa^{(1)}\Ob^{(g_2a^{-1})}\\
&\qquad+\Oa^{(1)}\Oa^{(bg_2a^{-1})}\Oa^{(1)}\Oa^{(1)}\Ob^{(g_1babg_2a^{-1})}\Ob^{(g_2a^{-1})}\Oa^{(1)}\\
&\qquad+\Oa^{(1)}\Oa^{(bg_2a^{-1})}\Oa^{(1)}\Ob^{(g_1babg_2a^{-1})}\Oa^{(1)}\Oa^{(1)}\Ob^{(g_2a^{-1})}\\
&\qquad-Q^{-2}\Oa^{(1)}\Oa^{(bg_2a^{-1})}\Oa^{(1)}\Ob^{(g_1babg_2a^{-1})}\Oa^{(1)}\Ob^{(g_2a^{-1})}\Oa^{(1)}\\
&\qquad-Q^2\Oa^{(1)}\Ob^{(g_2a^{-1})}\Oa^{(1)}\Ob^{(g_1babg_2a^{-1})}\Oa^{(1)}\Oa^{(bg_2a^{-1})}\Oa^{(1)}\\
&\qquad+\Oa^{(1)}\Ob^{(g_2a^{-1})}\Ob^{(g_1babg_2a^{-1})}\Oa^{(1)}\Oa^{(1)}\Oa^{(bg_2a^{-1})}\Oa^{(1)}\\
&\qquad+\Ob^{(g_2a^{-1})}\Oa^{(1)}\Oa^{(1)}\Ob^{(g_1babg_2a^{-1})}\Oa^{(1)}\Oa^{(bg_2a^{-1})}\Oa^{(1)}\\
&\qquad-Q^{-2}\Ob^{(g_2a^{-1})}\Oa^{(1)}\Ob^{(g_1babg_2a^{-1})}\Oa^{(1)}\Oa^{(1)}\Oa^{(bg_2a^{-1})}\Oa^{(1)}\Big)\\
&+\frac1{\big(Q^2-Q^{-2}\big)^6}\Oa^{(1)}\Oa^{(g_2a^{-1})}\Oa^{(1)}\Oa^{(1)}\Oa^{(\sigma(g_1)g_2a^{-1})}\Oa^{(1)}\Oa^{(1)}\Oa^{(g_2a^{-1})}\Oa^{(1)}\\
=&\frac1{\big(Q^2-Q^{-2}\big)^2}Q^{-2}\Ob^{(g_2a^{-1})}\R_{3}^{(1,g_1babg_2a^{-1},bg_2a^{-1})}\Oa^{(1)}\\
&+\frac1{\big(Q^2-Q^{-2}\big)^4}\Big(-\Oa^{(1)}\Oa^{(g_2a^{-1})}\R_{4,A}^{(\sigma(g_1)g_2a^{-1})}\Oa^{(g_2a^{-1})}\Oa^{(1)}\\
&\qquad-Q^2\Oa^{(1)}\Oa^{(bg_2a^{-1})}\Oa^{(1)}\Oa^{(1)}\R_{2}^{(g_1babg_2a^{-1},1,g_2a^{-1})}-Q^{-2}\Oa^{(1)}\Oa^{(bg_2a^{-1})}\Oa^{(1)}\Ob^{(g_1babg_2a^{-1})}\R_{3}^{(1,g_2a^{-1},1)}\\
&\qquad-\Oa^{(1)}\Oa^{(bg_2a^{-1})}\Ob^{(g_1babg_2a^{-1})}\Oa^{(1)}\R_{3}^{(1,g_2a^{-1},1)}+\Oa^{(1)}\Oa^{(bg_2a^{-1})}\R_{6}^{(g_1babg_2a^{-1})}\Ob^{(g_2a^{-1})}\Oa^{(1)}\\
&\qquad+\Oa^{(1)}\Ob^{(g_2a^{-1})}\Ob^{(g_1babg_2a^{-1})}\R_{4,A}^{(bg_2a^{-1})}\Oa^{(1)}-Q^2\Oa^{(1)}\Ob^{(g_2a^{-1})}\R_{3}^{(1,g_1babg_2a^{-1},1)}\Oa^{(bg_2a^{-1})}\Oa^{(1)}\\
&\qquad-Q^{-2}\Ob^{(g_2a^{-1})}\Oa^{(1)}\Ob^{(g_1babg_2a^{-1})}\R_{4,A}^{(bg_2a^{-1})}\Oa^{(1)}+\Ob^{(g_2a^{-1})}\Oa^{(1)}\R_{3}^{(1,g_1babg_2a^{-1},1)}\Oa^{(bg_2a^{-1})}\Oa^{(1)}\Big)\\
&+\frac1{\big(Q^2-Q^{-2}\big)^6}\Oa^{(1)}\Oa^{(g_2a^{-1})}\Oa^{(1)}\Oa^{(1)}\Oa^{(\sigma(g_1)g_2a^{-1})}\R_{4,A}^{(g_2a^{-1})}\Oa^{(1)}\qquad\in\quad\mathcal I,\\[1em]
\mathbf a\big(\R_{10}^{(g_1,g_2)}\big)&+\frac1{Q^2-Q^{-2}}\Big(-Q\Oa^{(1)}\R_{10}^{(g_1,g_2a^{-1})}+Q^{-1}\R_{10}^{(g_1,g_2a^{-1})}\Oa^{(1)}\Big)\\
=&\frac1{Q^2-Q^{-2}}\Big(-Q\Oa^{(1)}\Oa^{(g_2a^{-1})}\Oa^{(g_1a^{-1}b^{-1}a^{-1}g_2a^{-1})}\Ob^{(a^{-1}g_2a^{-1})}+Q\Oa^{(1)}\Ob^{(g_2a^{-1})}\Ob^{(\sigma(g_1)g_2a^{-1})}\Ob^{(g_2a^{-1})}\\
&\qquad-Q\Oa^{(1)}\Ob^{(a^{-1}g_2a^{-1})}\Oa^{(g_1a^{-1}b^{-1}a^{-1}g_2a^{-1})}\Oa^{(g_2a^{-1})}+Q^{-1}\Oa^{(g_2a^{-1})}\Oa^{(g_1a^{-1}b^{-1}a^{-1}g_2a^{-1})}\Ob^{(a^{-1}g_2a^{-1})}\Oa^{(1)}\\
&\qquad-Q^{-1}\Ob^{(g_2a^{-1})}\Ob^{(\sigma(g_1)g_2a^{-1})}\Ob^{(g_2a^{-1})}\Oa^{(1)}+Q^{-1}\Ob^{(a^{-1}g_2a^{-1})}\Oa^{(g_1a^{-1}b^{-1}a^{-1}g_2a^{-1})}\Oa^{(g_2a^{-1})}\Oa^{(1)}\Big)\\
&+\frac1{\big(Q^2-Q^{-2}\big)^3}\Big(-Q^3\Oa^{(1)}\Ob^{(g_2a^{-1})}\Oa^{(1)}\Ob^{(\sigma(g_1)g_2a^{-1})}\Oa^{(1)}\Ob^{(g_2a^{-1})}\\
&\qquad+Q\Oa^{(1)}\Ob^{(g_2a^{-1})}\Oa^{(1)}\Ob^{(\sigma(g_1)g_2a^{-1})}\Ob^{(g_2a^{-1})}\Oa^{(1)}+Q\Oa^{(1)}\Ob^{(g_2a^{-1})}\Ob^{(\sigma(g_1)g_2a^{-1})}\Oa^{(1)}\Oa^{(1)}\Ob^{(g_2a^{-1})}\\
&\qquad-Q^{-1}\Oa^{(1)}\Ob^{(g_2a^{-1})}\Ob^{(\sigma(g_1)g_2a^{-1})}\Oa^{(1)}\Ob^{(g_2a^{-1})}\Oa^{(1)}+Q\Ob^{(g_2a^{-1})}\Oa^{(1)}\Oa^{(1)}\Ob^{(\sigma(g_1)g_2a^{-1})}\Oa^{(1)}\Ob^{(g_2a^{-1})}\\
&\qquad-Q^{-1}\Ob^{(g_2a^{-1})}\Oa^{(1)}\Oa^{(1)}\Ob^{(\sigma(g_1)g_2a^{-1})}\Ob^{(g_2a^{-1})}\Oa^{(1)}-Q^{-1}\Ob^{(g_2a^{-1})}\Oa^{(1)}\Ob^{(\sigma(g_1)g_2a^{-1})}\Oa^{(1)}\Oa^{(1)}\Ob^{(g_2a^{-1})}\\
&\qquad+Q^{-3}\Ob^{(g_2a^{-1})}\Oa^{(1)}\Ob^{(\sigma(g_1)g_2a^{-1})}\Oa^{(1)}\Ob^{(g_2a^{-1})}\Oa^{(1)}\Big)\\
&+\frac1{\big(Q^2-Q^{-2}\big)^5}\Big(Q\Oa^{(1)}\Oa^{(g_2a^{-1})}\Oa^{(1)}\Oa^{(1)}\Oa^{(g_1a^{-1}b^{-1}a^{-1}g_2a^{-1})}\Oa^{(1)}\Oa^{(1)}\Ob^{(a^{-1}g_2a^{-1})}\\
&\qquad-Q^{-1}\Oa^{(1)}\Oa^{(g_2a^{-1})}\Oa^{(1)}\Oa^{(1)}\Oa^{(g_1a^{-1}b^{-1}a^{-1}g_2a^{-1})}\Oa^{(1)}\Ob^{(a^{-1}g_2a^{-1})}\Oa^{(1)}\\
&\qquad+Q\Oa^{(1)}\Ob^{(a^{-1}g_2a^{-1})}\Oa^{(1)}\Oa^{(g_1a^{-1}b^{-1}a^{-1}g_2a^{-1})}\Oa^{(1)}\Oa^{(1)}\Oa^{(g_2a^{-1})}\Oa^{(1)}\\
&\qquad-Q^{-1}\Ob^{(a^{-1}g_2a^{-1})}\Oa^{(1)}\Oa^{(1)}\Oa^{(g_1a^{-1}b^{-1}a^{-1}g_2a^{-1})}\Oa^{(1)}\Oa^{(1)}\Oa^{(g_2a^{-1})}\Oa^{(1)}\Big)\\
=&\frac1{Q^2-Q^{-2}}\Big(Q^{-1}\Oa^{(g_2a^{-1})}\R_{3}^{(g_1a^{-1}b^{-1}a^{-1}g_2a^{-1},a^{-1}g_2a^{-1},1)}+Q^{-1}\R_{2}^{(g_2a^{-1},1,\sigma(g_1)g_2a^{-1})}\Ob^{(g_2a^{-1})}\\
&\qquad-Q\R_{3}^{(1,a^{-1}g_2a^{-1},g_1a^{-1}b^{-1}a^{-1}g_2a^{-1})}\Oa^{(g_2a^{-1})}\Big)\\
&+\frac1{\big(Q^2-Q^{-2}\big)^3}\Big(-Q\Oa^{(1)}\Oa^{(g_2a^{-1})}\R_{4,A}^{(g_1a^{-1}b^{-1}a^{-1}g_2a^{-1})}\Ob^{(a^{-1}g_2a^{-1})}\\
&\qquad-Q^3\Oa^{(1)}\Ob^{(g_2a^{-1})}\Oa^{(1)}\R_{2}^{(\sigma(g_1)g_2a^{-1},1,g_2a^{-1})}+Q\Oa^{(1)}\Ob^{(g_2a^{-1})}\Ob^{(\sigma(g_1)g_2a^{-1})}\R_{6}^{(g_2a^{-1})}\\
&\qquad-Q^{-1}\Oa^{(1)}\Ob^{(g_2a^{-1})}\Ob^{(\sigma(g_1)g_2a^{-1})}\R_{3}^{(1,g_2a^{-1},1)}+Q\Oa^{(1)}\R_{2}^{(g_2a^{-1},1,\sigma(g_1)g_2a^{-1})}\Ob^{(g_2a^{-1})}\Oa^{(1)}\\
&\qquad+Q\Ob^{(g_2a^{-1})}\Oa^{(1)}\Oa^{(1)}\R_{2}^{(\sigma(g_1)g_2a^{-1},1,g_2a^{-1})}-Q^{-1}\Ob^{(g_2a^{-1})}\Oa^{(1)}\Ob^{(\sigma(g_1)g_2a^{-1})}\R_{6}^{(g_2a^{-1})}\\
&\qquad+Q^{-3}\Ob^{(g_2a^{-1})}\Oa^{(1)}\Ob^{(\sigma(g_1)g_2a^{-1})}\R_{3}^{(1,g_2a^{-1},1)}+Q^{-1}\Ob^{(g_2a^{-1})}\Ob^{(\sigma(g_1)g_2a^{-1})}\Oa^{(1)}\R_{3}^{(1,g_2a^{-1},1)}\\
&\qquad-Q^{-1}\Ob^{(g_2a^{-1})}\R_{6}^{(\sigma(g_1)g_2a^{-1})}\Ob^{(g_2a^{-1})}\Oa^{(1)}+Q^{-1}\Ob^{(a^{-1}g_2a^{-1})}\R_{4,A}^{(g_1a^{-1}b^{-1}a^{-1}g_2a^{-1})}\Oa^{(g_2a^{-1})}\Oa^{(1)}\\
&\qquad+Q^{-1}\R_{2}^{(g_2a^{-1},1,\sigma(g_1)g_2a^{-1})}\Ob^{(g_2a^{-1})}\Oa^{(1)}\Oa^{(1)}-Q\R_{3}^{(1,a^{-1}g_2a^{-1},1)}\Oa^{(g_1a^{-1}b^{-1}a^{-1}g_2a^{-1})}\Oa^{(g_2a^{-1})}\Oa^{(1)}\\
&\qquad-Q\Oa^{(1)}\Ob^{(g_2a^{-1})}\Ob^{(\sigma(g_1)g_2a^{-1})}\Ob^{(g_2a^{-1})}\Oa^{(1)}\Oa^{(1)}\Big)\\
&+\frac1{\big(Q^2-Q^{-2}\big)^5}\Big(Q\Oa^{(1)}\Oa^{(g_2a^{-1})}\Oa^{(1)}\Oa^{(1)}\Oa^{(g_1a^{-1}b^{-1}a^{-1}g_2a^{-1})}\R_{6}^{(a^{-1}g_2a^{-1})}\\
&\qquad-Q^{-1}\Oa^{(1)}\Oa^{(g_2a^{-1})}\Oa^{(1)}\Oa^{(1)}\Oa^{(g_1a^{-1}b^{-1}a^{-1}g_2a^{-1})}\R_{3}^{(1,a^{-1}g_2a^{-1},1)}\\
&\qquad-Q\Oa^{(1)}\Oa^{(g_2a^{-1})}\Oa^{(1)}\Oa^{(1)}\R_{3}^{(g_1a^{-1}b^{-1}a^{-1}g_2a^{-1},a^{-1}g_2a^{-1},1)}\Oa^{(1)}\\
&\qquad+Q\Oa^{(1)}\Ob^{(a^{-1}g_2a^{-1})}\Oa^{(1)}\Oa^{(g_1a^{-1}b^{-1}a^{-1}g_2a^{-1})}\R_{4,A}^{(g_2a^{-1})}\Oa^{(1)}\\
&\qquad-Q^{-1}\Ob^{(a^{-1}g_2a^{-1})}\Oa^{(1)}\Oa^{(1)}\Oa^{(g_1a^{-1}b^{-1}a^{-1}g_2a^{-1})}\R_{4,A}^{(g_2a^{-1})}\Oa^{(1)}\Big)\;\;\overset{(\ref{eq:IdempotentARightCancellation})}{\equiv}\;\;0\;\;\bmod\mathcal I,\\[1em]
\mathbf a\big(\R_{11}^{(g_1,g_2)}\big)&+\frac1{\big(Q^2-Q^{-2}\big)^2}\Oa^{(1)}\R_{11}^{(g_1,g_2a^{-1})}\Oa^{(1)}\\
=&\frac1{\big(Q^2-Q^{-2}\big)^2}\Big(-\Oa^{(1)}\Oa^{(g_2a^{-1})}\Oa^{(\sigma(g_1)g_2a^{-1})}\Oa^{(g_2a^{-1})}\Oa^{(1)}+\Oa^{(1)}\Oa^{(b^{-1}g_2a^{-1})}\Ob^{(g_1b^{-1}a^{-1}b^{-1}g_2a^{-1})}\Ob^{(g_2a^{-1})}\Oa^{(1)}\\
&\qquad+\Oa^{(1)}\Ob^{(g_2a^{-1})}\Ob^{(g_1b^{-1}a^{-1}b^{-1}g_2a^{-1})}\Oa^{(b^{-1}g_2a^{-1})}\Oa^{(1)}\Big)\\
&+\frac1{\big(Q^2-Q^{-2}\big)^4}\Big(-Q^2\Oa^{(1)}\Oa^{(b^{-1}g_2a^{-1})}\Oa^{(1)}\Oa^{(1)}\Ob^{(g_1b^{-1}a^{-1}b^{-1}g_2a^{-1})}\Oa^{(1)}\Ob^{(g_2a^{-1})}\\
&\qquad+\Oa^{(1)}\Oa^{(b^{-1}g_2a^{-1})}\Oa^{(1)}\Oa^{(1)}\Ob^{(g_1b^{-1}a^{-1}b^{-1}g_2a^{-1})}\Ob^{(g_2a^{-1})}\Oa^{(1)}\\
&\qquad+\Oa^{(1)}\Oa^{(b^{-1}g_2a^{-1})}\Oa^{(1)}\Ob^{(g_1b^{-1}a^{-1}b^{-1}g_2a^{-1})}\Oa^{(1)}\Oa^{(1)}\Ob^{(g_2a^{-1})}\\
&\qquad-Q^{-2}\Oa^{(1)}\Oa^{(b^{-1}g_2a^{-1})}\Oa^{(1)}\Ob^{(g_1b^{-1}a^{-1}b^{-1}g_2a^{-1})}\Oa^{(1)}\Ob^{(g_2a^{-1})}\Oa^{(1)}\\
&\qquad-Q^2\Oa^{(1)}\Ob^{(g_2a^{-1})}\Oa^{(1)}\Ob^{(g_1b^{-1}a^{-1}b^{-1}g_2a^{-1})}\Oa^{(1)}\Oa^{(b^{-1}g_2a^{-1})}\Oa^{(1)}\\
&\qquad+\Oa^{(1)}\Ob^{(g_2a^{-1})}\Ob^{(g_1b^{-1}a^{-1}b^{-1}g_2a^{-1})}\Oa^{(1)}\Oa^{(1)}\Oa^{(b^{-1}g_2a^{-1})}\Oa^{(1)}\\
&\qquad+\Ob^{(g_2a^{-1})}\Oa^{(1)}\Oa^{(1)}\Ob^{(g_1b^{-1}a^{-1}b^{-1}g_2a^{-1})}\Oa^{(1)}\Oa^{(b^{-1}g_2a^{-1})}\Oa^{(1)}\\
&\qquad-Q^{-2}\Ob^{(g_2a^{-1})}\Oa^{(1)}\Ob^{(g_1b^{-1}a^{-1}b^{-1}g_2a^{-1})}\Oa^{(1)}\Oa^{(1)}\Oa^{(b^{-1}g_2a^{-1})}\Oa^{(1)}\Big)\\
&+\frac1{\big(Q^2-Q^{-2}\big)^6}\Oa^{(1)}\Oa^{(g_2a^{-1})}\Oa^{(1)}\Oa^{(1)}\Oa^{(\sigma(g_1)g_2a^{-1})}\Oa^{(1)}\Oa^{(1)}\Oa^{(g_2a^{-1})}\Oa^{(1)}\\
=&\frac1{\big(Q^2-Q^{-2}\big)^2}Q^{-2}\Ob^{(g_2a^{-1})}\R_{3}^{(1,g_1b^{-1}a^{-1}b^{-1}g_2a^{-1},b^{-1}g_2a^{-1})}\Oa^{(1)}\\
&+\frac1{\big(Q^2-Q^{-2}\big)^4}\Big(-\Oa^{(1)}\Oa^{(g_2a^{-1})}\R_{4,A}^{(\sigma(g_1)g_2a^{-1})}\Oa^{(g_2a^{-1})}\Oa^{(1)}\\
&\qquad-Q^2\Oa^{(1)}\Oa^{(b^{-1}g_2a^{-1})}\Oa^{(1)}\Oa^{(1)}\R_{2}^{(g_1b^{-1}a^{-1}b^{-1}g_2a^{-1},1,g_2a^{-1})}\\
&\qquad-Q^{-2}\Oa^{(1)}\Oa^{(b^{-1}g_2a^{-1})}\Oa^{(1)}\Ob^{(g_1b^{-1}a^{-1}b^{-1}g_2a^{-1})}\R_{3}^{(1,g_2a^{-1},1)}\\
&\qquad-\Oa^{(1)}\Oa^{(b^{-1}g_2a^{-1})}\Ob^{(g_1b^{-1}a^{-1}b^{-1}g_2a^{-1})}\Oa^{(1)}\R_{3}^{(1,g_2a^{-1},1)}\\
&\qquad+\Oa^{(1)}\Oa^{(b^{-1}g_2a^{-1})}\R_{6}^{(g_1b^{-1}a^{-1}b^{-1}g_2a^{-1})}\Ob^{(g_2a^{-1})}\Oa^{(1)}\\
&\qquad+\Oa^{(1)}\Ob^{(g_2a^{-1})}\Ob^{(g_1b^{-1}a^{-1}b^{-1}g_2a^{-1})}\R_{4,A}^{(b^{-1}g_2a^{-1})}\Oa^{(1)}\\
&\qquad-Q^2\Oa^{(1)}\Ob^{(g_2a^{-1})}\R_{3}^{(1,g_1b^{-1}a^{-1}b^{-1}g_2a^{-1},1)}\Oa^{(b^{-1}g_2a^{-1})}\Oa^{(1)}\\
&\qquad-Q^{-2}\Ob^{(g_2a^{-1})}\Oa^{(1)}\Ob^{(g_1b^{-1}a^{-1}b^{-1}g_2a^{-1})}\R_{4,A}^{(b^{-1}g_2a^{-1})}\Oa^{(1)}\\
&\qquad+\Ob^{(g_2a^{-1})}\Oa^{(1)}\R_{3}^{(1,g_1b^{-1}a^{-1}b^{-1}g_2a^{-1},1)}\Oa^{(b^{-1}g_2a^{-1})}\Oa^{(1)}\Big)\\
&+\frac1{\big(Q^2-Q^{-2}\big)^6}\Oa^{(1)}\Oa^{(g_2a^{-1})}\Oa^{(1)}\Oa^{(1)}\Oa^{(\sigma(g_1)g_2a^{-1})}\R_{4,A}^{(g_2a^{-1})}\Oa^{(1)}\qquad\in\quad\mathcal I.
\end{align*}

\bibliographystyle{alpha}
\bibliography{references}

\end{document}